\documentclass[11pt, reqno]{amsart} 
\usepackage{fullpage, amsmath, amssymb, young}
\usepackage[all]{xypic}
\usepackage{palatino}

\title{A new construction of cyclic homology}
\author{Victor Ginzburg}
\address{V.G.:
Department of Mathematics, University of Chicago,  Chicago, IL 
60637, USA.}
\email{ginzburg@math.uchicago.edu}
\author{Travis Schedler}
\address{T.S.: Department of Mathematics, MIT, Cambridge, MA 02139, USA}
\email{trasched@gmail.com}

\usepackage{url}            
\usepackage{graphicx}

\numberwithin{equation}{subsection}
\theoremstyle{plain}
\newtheorem{theorem}[equation]{Theorem}
\newtheorem{lemma}[equation]{Lemma}
\newtheorem{proposition}[equation]{Proposition}
\newtheorem{corollary}[equation]{Corollary}

\theoremstyle{definition}

\newtheorem{example}[equation]{Example}

\theoremstyle{remark}
\newtheorem{remark}[equation]{Remark}

\newtheorem{claim}[equation]{Claim}

\newcommand{\act}{\operatorname{act}}
\newcommand{\Span}{\operatorname{Span}}
\newcommand{\Id}{\operatorname{Id}}

\newcommand{\ad}{\operatorname{ad}}

\newcommand{\im}{\operatorname{im}}
\newcommand{\bk}{\mathbf{k}}

\newcommand{\cl}{\operatorname{cl}}
\newcommand{\oA}{\overline{A}}
\newcommand{\oOm}{{\overline{\Omega}}}
\DeclareMathOperator{\sd}{{\mathsf{d}}}
\newcommand{\osd}{{\mathsf{\bar d}}}

\DeclareMathOperator{\GL}{\mathrm{GL}}
\DeclareMathOperator{\msb}{{\mathsf{b}}}
\DeclareMathOperator{\msB}{{\mathsf{B}}}
\DeclareMathOperator{\sN}{{\mathsf{N}}}
\DeclareMathOperator{\si}{{\mathsf{i}}}
\DeclareMathOperator{\DR}{\mathrm{DR}}
\DeclareMathOperator{\sign}{sign}
\DeclareMathOperator{\oDR}{\,\overline{\!DR\!}\,}
\DeclareMathOperator{\ohDRt}{\oDR_t^\wedge}
\newcommand{\ab}{\operatorname{\mathrm{c}}}
\DeclareMathOperator{\HD}{\mathrm{HD}}
\DeclareMathOperator{\EHD}{\mathrm{EHD}}
\DeclareMathOperator{\CC}{{CC}}
\DeclareMathOperator{\oCC}{{\overline{CC}}}
\newcommand{\HDe}{{{}^{\mathcal{D}\!}{H}}}
\newcommand{\oHDe}{{{}^{\mathcal{D}\!}{\,\overline{\!H\!}\,}}}

\DeclareMathOperator{\oHD}{{\,\overline{\!\HD\!}\,}}

\DeclareMathOperator{\HH}{{\mathrm{HH}}}
\DeclareMathOperator{\EHH}{{\mathrm{EHH}}}
\newcommand{\oHH}{\,\overline{\!\HH\!}\,}

\DeclareMathOperator{\HC}{{\mathrm{HC}}}
\DeclareMathOperator{\EHC}{{\mathrm{EHC}}}

\newcommand{\oHC}{\,\overline{\!\HC\!}\,}

\newcommand{\oEHC}{\,\overline{\!\EHC\!}\,}
\newcommand{\oEHH}{\,\overline{\!\EHH\!}\,}
\newcommand{\oEHD}{\,\overline{\!\EHD\!}\,}

\newcommand{\EE}{{\mathsf{E}}}

\newcommand{\bR}{\mathbb{R}}
\newcommand{\bQ}{\mathbb{Q}}
\newcommand{\bZ}{\mathbb{Z}}

\newcommand{\iso}{{\;\stackrel{_\sim}{\to}\;}}

\newcommand{\Rep}{\operatorname{Rep}}

\newcommand{\pr}{\operatorname{\mathsf w}}

\newcommand{\mfg}{\mathfrak{g}}
\newcommand{\Vect}{\operatorname{Vect}}
\newcommand{\Mat}{\operatorname{Mat}}

\newcommand{\gl}{\mathfrak{gl}}
\newcommand{\g}{\mathfrak{g}}
\newcommand{\Z}{{\mathbb Z}}
\newcommand{\gr}{\operatorname{\mathsf{gr}}}
\newcommand{\Spec}{\operatorname{\mathsf{Spec}}}

\newcommand{\onto}{\twoheadrightarrow}
\newcommand{\into}{\hookrightarrow}

\newcommand{\ev}{\operatorname{ev}}
\newcommand{\tr}{\operatorname{tr}}

\newcommand{\dis}{\displaystyle}
\newcommand{\inv}{^{-1}}
\newcommand{\sdi}{{\mathsf{\tilde d}^{-1}}}
\newcommand{\ooplus}{\ \oplus\  }

\def\ccirc{{{}_{\,{}^{^\circ}}}}
\newcommand{\mmod}{{\ \operatorname{\mathsf{mod}\,}}}

\renewcommand{\o}{\otimes }
\newcommand{\Om}{\Omega }
\newcommand{\per}{{\operatorname{per}}}
\newcommand{\hdot}{{\:\raisebox{2pt}{\text{\circle*{1.5}}}}}
\newcommand{\idot}{{\:\raisebox{2pt}{\text{\circle*{1.5}}}}}
\newcommand{\too}{\,\longrightarrow\,}
\newcommand{\mto}{\mapsto }
\newcommand{\bi}{{\mathbf{i}}}

\newcommand{\bbu}{(\!(u)\!)}
\newcommand{\dpu}{(\!(u)\!)}
\newcommand{\dspu}{[\![u]\!]}
\newcommand{\ta}{\tau }
\newcommand{\half}{\mbox{$\frac{1}{2}$}}
\newcommand{\sZ}{\mathsf{Z}}
\newcommand{\bplus}{\mbox{$\bigoplus$}}

\newcommand{\en}{\enspace }

\newcommand{\vi}{${\en\sf {(i)}}\;$}
\newcommand{\vii}{${\;\sf {(ii)}}\;$}
\newcommand{\viii}{${\sf {(iii)}}\;$}

\newcommand{\DDD}{{D_{_{\![\,]}}}}
\newcommand{\sset}{\subset}

\newcommand{\oom}{\oOm }
\newcommand{\kap}{\kappa }
\newcommand{\il}{\mbox{$\left[\frac{\ell}{2}\right]$}}
\newcommand{\abb}{{\rm{ab}}}
\newcommand{\cc}{_{\natural}}

\newcommand{\su}{[\![u]\!]}
\newcommand{\ddt}{\mbox{$\frac{d}{dt}$}}

\newcommand{\cll}{C_\ell }

\begin{document}
\begin{abstract}
  Based on the ideas of Cuntz and Quillen, we give a simple
  construction of cyclic homology of unital algebras in terms of the
  noncommutative de Rham complex and a certain differential similar to
  the equivariant de Rham differential. We describe the Connes exact
  sequence in this setting.

  We define {\em equivariant Deligne cohomology} and construct, for
  each $n\geq1,$ a natural map from cyclic homology of an algebra to
  the $\GL_n$-equivariant Deligne cohomology of the variety of
  $n$-dimensional representations of that algebra. The bridge between
  cyclic homology and equivariant Deligne cohomology is provided by
  \emph{extended cyclic homology}, which we define and compute here,
  based on the extended noncommutative de Rham complex introduced
  previously by the authors.
\end{abstract}
\maketitle
\tableofcontents

\section{Introduction}
\subsection{}
There are several definitions of cyclic homology. The original
definition, due to A. Connes \cite{Con-ndg} (see also Tsygan
\cite{Tsy-hmlarhh}), works for algebras over fields of characteristic
zero and is based on a cyclic version of the Hochschild complex.
Another definition that involves a double complex and that works in
arbitrary characteristic is due to Loday and Quillen \cite{LQ-chlahm}
(motivated by Tsygan \cite{Tsy-hmlarhh}).  This definition may be
reformulated in terms of a slightly different $(\msb,\msB)$-complex,
where $\msb$ and $\msB$ are the Hochschild and the Connes
differentials, respectively.  There is yet another definition of
cyclic homology as a certain nonabelian derived functor, due to Feigin
and Tsygan \cite{FT-aktcc}.  All the above definitions are known to be
equivalent (see, e.g., \cite{L}), and each one has certain advantages.

In this paper, we give one more definition of (reduced) cyclic
homology for unital algebras over fields of characteristic zero (more
generally, over commutative rings $\bk$ containing $\bQ$, when the
algebra is $\bk$-split). Our approach is partly motivated by a well
known analogy, that goes back to Rinehart \cite{Rin-dfgca}, between
the Connes differential $\msB$ and the familiar de Rham differential
on differential forms. In the case of commutative algebras, this
analogy can be made precise (see, e.g., \cite[\S3.4]{L}). In the
general case of not necessarily commutative algebras, Karoubi
\cite{Karhckt} introduced (motivated by an earlier construction of
Connes) a certain {\em noncommutative de Rham complex} that comes
equipped with a natural differential $\sd$, the Karoubi-de Rham
differential. The cohomology of the noncommutative de Rham complex is
not directly related to cyclic homology, however.  A precise relation
between the two is given by the so-called Connes-Karoubi theorem; see
\cite[Theorem 2.6.7]{L} or \eqref{e:connes-seq} below.

In our approach, cyclic homology is constructed directly in terms of
the noncommutative de Rham complex so that the Karoubi-de Rham
differential $\sd$ literally replaces the Connes differential
$\msB$. There is also another differential involved in the
construction, a counterpart of the Hochschild differential
$\msb$. This new differential $\si$, introduced first in \cite[\S 2.8,
\S 3.1]{CBEG}, anti-commutes with $\sd$ (while $\msb$ does not
anti-commute with $\sd$) and is, we believe, more closely related to
geometry than $\msb$.  Specifically, for every (not necessarily
commutative) algebra $A$, there is a canonical derivation $\Delta:
A\to A\o A$ given by the formula $\Delta(a):= 1\o a-a\o1$.  The
differential $\si$ may be thought of as a contraction operation of
noncommutative differential forms with the derivation $\Delta$; see
\emph{op.~cit.}, \cite{GScyc}, and Section \ref{ss:eqcoh} below for
more details.

There are two features of the construction of cyclic homology given in
this paper which, we believe, are especially appealing.  The first one
is its close analogy with equivariant cohomology.  A well known result
of Goodwillie \cite{Go} says that the cyclic homology of (the de Rham
DG algebra of) a manifold $X$ is isomorphic to the $S^1$-equivariant
cohomology of $L(X)$, the corresponding free loop space. The role of
the {\em extended noncommutative de Rham complex}, introduced in
\S\ref{ss:eqcoh} below, is somewhat analogous to that of the
equivariant de Rham complex of the free loop space. Specifically, our
extended noncommutative de Rham complex comes equipped with a pair of
anti-commuting differentials, the Karoubi-de Rham differential $\sd$
and a differential $\si_t$ which is analogous to the $S^1$-equivariant
differential.  Cyclic homology is then related to the cohomology of
the extended noncommutative de Rham complex via a map, which we call
the `Tsygan map', that intertwines the differentials $\msb$ and
$\msB$, respectively, in the standard $(\msb,\msB)$-complex with the
differentials $\si_t$ and $\sd$, respectively, in the extended
noncommutative de Rham complex.  The Tsygan map may be viewed as a
noncommutative analogue of multiplication by the function $\exp(t)$.

The second feature of our construction is its obvious similarity with
{\em truncated de Rham complexes}. The latter show up in the standard
description of cyclic homology of the coordinate ring of an algebraic
variety. Thus, our construction provides an analogous description in
the noncommutative setting.  Some of the above mentioned analogies can
be made precise via the notion of representation functor and
equivariant Deligne cohomology.  This will be discussed in
\S\ref{s:eqcoh}.

\subsection{Layout of the paper} In \S \ref{sec2}, we introduce the
noncommutative de Rham complex and discuss harmonic decomposition, a
powerful technical tool discovered by Cuntz and Quillen \cite{CQ1},
\cite{CQ2}.

In \S \ref{s:mainconstr}, we give our new construction of cyclic
homology, as well as similar constructions for periodic and negative
cyclic homology. We then use harmonic decomposition to prove that our
construction and the standard one are equivalent.

Section \ref{s:connes} is devoted to the Connes exact sequence.  We
show how to construct such a sequence entirely within our approach. We
also establish the equivalence with the standard construction.

Section \ref{ss:eqcoh} contains some of the most important results of
the paper. We introduce the extended noncommutative de Rham complex
and define extended cyclic homology. Extended cyclic homology groups
come equipped with an additional `weight grading.'  We show that the
positive weight part is isomorphic to the Karoubi-de Rham homology,
while the negative weight part reduces to the usual cyclic homology.
A key role in proving these results is played by the Tsygan map.  We
have arrived at the definition of the Tsygan map by analyzing Tsygan's
construction \cite{T11} (that construction was motivated, in turn, by
our results in \cite{GScyc}).

Section \ref{heq_pf} is devoted to the proof of Theorem \ref{t:heq},
which is perhaps the most difficult result of the paper.  In
\S\ref{negative}, we discuss `extended versions' of periodic and
negative cyclic homology.

In the last section we relate our results to geometry. We begin by
introducing {\em equivariant Deligne cohomology}, an equivariant
counterpart of Deligne cohomology, i.e., of the cohomology theory
based on truncated algebraic de Rham complexes. The {\em equivariant}
Deligne cohomology theory that we define appears to be new.

Next, we give a short review of some standard constructions related to
the representation functor. Recall that, associated with a (not
necessarily commutative) algebra $A$ and an integer $n\geq 1$, there
is an affine scheme $\Rep_n A$ parametrizing $n$-dimensional
representations of $A$, cf.~\cite{CBEG}. This scheme comes equipped
with a natural action of the algebraic group $\GL_n$, by base change
transformations.  An important application of our approach is a
construction of a canonical map from the cyclic homology of $A$ to the
equivariant Deligne cohomology of the scheme $\Rep_n A$.  We should
point out that considering equivariant (as opposed to ordinary)
cohomology of representation schemes is, in a sense, the only natural
thing to do.  Indeed, two representations of the algebra $A$ are
equivalent if and only they belong to the same $\GL_n$-orbit in
$\Rep_n A$.  Thus, the object which is most relevant here is the
quotient stack $\Rep_n A/\GL_n$ rather than the scheme $\Rep_n A$
itself. The $\GL_n$-equivariant Deligne cohomology of the scheme
$\Rep_n A$ may be viewed as the ordinary Deligne cohomology of the
stack $\Rep_n A/\GL_n$.

In the special case where $n=1$ and $A$ is the coordinate ring of a
smooth affine algebraic variety $X$, one has $\Rep_1 A=X$.  In this
case, our construction reduces to, and provides a more explicit form
of, the well known isomorphism between cyclic homology of the
coordinate ring and Deligne cohomology of $X$, cf.~\cite[\S3.6]{L}.

A different relation between cyclic homology of an algebra and its
representation schemes is studied in the recent paper by Berest,
Khachatryan, and A. Ramadoss \cite{BKR}. It does not consider
equivariant cohomology, however.

\subsection{Acknowledgements} We are very much indebted to Boris
Tsygan for his kind explanations of the construction in \cite{T11},
which inspired the present work.

The first author was supported in part by the NSF grant DMS-1001677.
The second author is a five-year fellow of the American Institute of
Mathematics, and was partially supported by the ARRA-funded NSF grant
DMS-0900233.

\section{Noncommutative differential forms}
\label{sec2}
\subsection{The Karoubi-de Rham complex}\label{sec21}
Fix $\bk$, a unital commutative $\bQ$-algebra, and put $\o=\o_\bk$.

Throughout the paper, we fix a unital associative algebra $A$ over the
ground ring $\bk$ satisfying the additional assumption that the natural
map $\ \bk \to A\ $ is a $\bk$-split injection, i.e., there exists a
$\bk$-module direct sum decomposition $A=\bk\oplus \oA$, where
$\oA=A/\bk$.  This assumption trivially holds if $\bk$ is a field.

Associated with $A$, there is a dg algebra $(\Omega A,\sd)$ of
\emph{noncommutative differential forms}, the differential envelope of
$A$ \cite{CQ1}. By definition, $\Om A$ is the quotient of $T_\bk(A\oplus
\oA)$, a free tensor algebra of the $\bk$-module $A\oplus \oA$, by the
relations of the form $a\o b= ab$ and $\sd(ab) = \sd\! a \o b + a \o
\sd\!b$, for all $a,b\in A$, where we use the notation $\sd\!a$
for the class of $a\in A$ in $\oA$.  One puts a grading $\Om^\hdot
A=\bigoplus_{n \geq 0} \Om^nA$, on $\Om A$, by placing $A$ in degree
$0$ and $\oA$ in degree $1$.  The differential $\sd: \Om^\hdot A\to
\Om^{\hdot+1} A$ is given by the obvious assignment: $a\mapsto \sd\!
a\mapsto 0$.  This differential kills $\bk\sset\Om^0A$ and hence
descends to $\oOm A:=\Om A/\bk$, the reduced complex.

As explained in \cite[\S 1]{CQ1}, one has a canonical isomorphism of
left $A$-modules
\begin{equation}\label{aaa}
  A \otimes \oA^{\otimes n}\ \iso\ \Om^n A,
  \quad 
  a_0 \o \sd\!a_1  \o \sd\! a_2 \o  \cdots \o  \sd\! a_{n}\
  \mapsto\
  a_0 \sd\! a_1 \sd\! a_2 \cdots \sd\! a_{n},
  \quad  a_0,\ldots,a_{n} \in 
  A.
\end{equation}
From this and the assumption that $\bk\into A$ is split it follows
that $(\oOm A, \sd)$ is acyclic; see \cite[\S 1]{CQ1}.

For every graded algebra $B$ and graded vector subspace $V\subseteq
B$, let $[V,V]$ denote the linear span of (super)commutators
$[v_1,v_2],\ v_1,v_2\in V$.  Following Karoubi \cite{Karhckt}, define
the noncommutative de Rham complex of $A$ and its reduced version by
$$
\DR A:=\Omega A/[\Omega A,\,\Omega A],\quad\text{and}\quad \oDR A:=
(\DR A)/\bk\ =\ \oOm A/\overline{[\Om A,\,\Om A]}.
$$
We write $f\cc$ for the image of an element $f\in\Om A$ under the
natural projection $\Om A\onto\DR A$.  

The differential $\sd$ on $\Omega A$ descends to $\DR A$. Let
$\HD_n(A)$ and $\oHD_n(A)$ denote the $n$-th homology of $(\DR A,
\sd)$ and $(\oDR A, \sd)$, respectively. We will call these the
\emph{Karoubi-de Rham homology} and its reduced version.

Following \cite{CQ1, CQ2}, the Hochschild differential on $\Omega A$
(or on $\oOm A$) is given by the formula
\begin{equation}\label{e:bfla}
  \msb (\alpha \sd\! a) = (-1)^{n-1}[\alpha, a], \enspace
  \alpha\in \Om^n A,\ a\in A,\enspace\text{and}\enspace \msb|_{\Om^0A} = 0.
\end{equation}

A crucial role below will also be played by another map $\bi: \Om A
\to \Omega A$, defined by
\begin{equation}\label{si}
  \bi(a_0 \sd\! a_1 \cdots \sd\! a_n) = \sum_{\ell=1}^n 
  (-1)^{(\ell-1)(n-1)+1}[a_\ell,\
  \sd\! a_{\ell+1} \cdots \sd\! a_n a_0 \sd\! a_1 \cdots \sd\!
  a_{\ell-1}]
  \ \in\ [A, \Om^{n-1}A].
\end{equation}
It is straightforward to check that the map $\bi$ vanishes on $[\Om A,
\Om A]$ and hence descends to a well defined map $\si: \DR^\hdot
A\to\Om^{\hdot-1}A$.  It follows that $\bi^2=0$. In addition, one
can check that $\sd\si+\si\sd=0$.

The map $\bi$ was first introduced in \cite{CBEG}.  A more conceptual
definition of this map (see \S\ref{ss:eqcoh}) was discovered in
\cite{GScyc}.  From that definition the above stated properties become
immediate consequences of noncommutative calculus.

\subsection{Harmonic decomposition} \label{ss:harmdec} Since the
algebra $A$ is fixed throughout the paper, we will often use
simplified notation $\Om=\Om A,\ \oOm=\oOm A,\ \DR=\DR A$, etc.

Following \cite{Karhckt}, define the Karoubi operator as
\begin{equation}\label{kap}
  \kappa: \oOm \to \oOm, \quad\kappa(\alpha\, \sd\! a) = 
  (-1)^{|\alpha|} \sd\! a\,
  \alpha,
  \quad \alpha\in \Om,\ a \in A; 
  \quad \kappa|_{\oOm^0} = 0.
\end{equation}
The Karoubi operator is related to the operators $\msb$ and $\sd$ by
\cite{Karhckt,CQ1}
\begin{equation}\label{CQid1} 
\msb \sd +\sd\msb =\Id-\kappa.
\end{equation}

According to \cite[\S 2]{CQ2}, there is a direct sum decomposition
$$
\oOm = P\oOm \oplus P^\perp \oOm, \quad P \oOm := \ker(\Id-\kappa)^2,
\quad P^\perp \oOm := \im(\Id - \kappa)^2.
$$
Let $P$ and $P^\perp$ denote the projections onto the first and second
summands of this decomposition, respectively, which we will call the
harmonic and antiharmonic parts.

It follows from \eqref{CQid1} that $\kappa$ commutes with $\sd$ and
$\msb$. Hence, harmonic decomposition is stable under $\kappa$,
$\msb$, and $\sd$, and it induces a similar decomposition on
$H(\Om,\msb)$ and $H(\Om,\sd)$.

Let $\sN$ be the grading operator on $\oOm$, defined by $\sN|_{\oOm^n}
= n \cdot \Id$.  With the above notation, formula \eqref{si} may be
rewritten in the form \cite{GScyc}
\begin{equation} \label{e:iP} \bi|_{\oOm^n}\ =\ (\Id + \kappa + \cdots
  + \kappa^{n-1}) \msb|_{\oOm^n A}\ =\ \msb \sN P|_{\oOm^n}.
\end{equation}

Finally, following \cite{CQ2}, one defines the Connes differential by
\begin{equation} \label{connes} \msB|_{\oOm^n}\ :=\ (\Id + \kappa +
  \cdots + \kappa^n) \sd|_{\oOm^n} \ =\ \sN \sd P|_{\oOm^n}.
\end{equation}

The above formulas show that the operators $\bi$ and $\msB$ respect
harmonic decomposition and one has
$$\bi\sd=\msb\sN\sd P=\msb\msB,
\quad\text{and}\quad \sd\bi=\sd\msb\sN P=\sN \sd\msb P=\msB\msb.$$
Thus, the equations $\sd\bi+\bi\sd=0$ and $\msB\msb+\msb\msB=0$ are
equivalent.

The next lemma collects various technical properties of harmonic
decomposition which will be used in various proofs (but not in the
statements) in subsequent sections.

\begin{lemma}\label{tech} One has
\begin{gather}
  P\overline{[\Om,\Om]}\ =\ P\msb\oom\ =\ \bi\, \oom\label{tech1}\\
  P^\perp\overline{[\Om,\Om]}\ =\ P^\perp\oom\label{tech2}\\
  P^\perp\sd\oom\ =\ (\Id-\kap)\sd\oom\ =\ \overline{[\sd\Om,\sd\Om]}
  \label{tech3}\\
  P^\perp\msb\oom\ =\ (\Id-\kap)\msb\oom.\label{tech4}
\end{gather}
\end{lemma}
\begin{proof} Given a $\kap$-action on a vector space $V$, let
  $V^\kap=\ker(\Id-\kap)|_V$ denote the spaces of $\kap$-invariants.
  If the $\kap$-action on $V$ has finite order then 
  $V=V^\kap\oplus (\Id-\kap)V$; in particular, the operator
  $(\Id-\kap)^2$ acts on $V$ by zero if and only if $V=V^\kap$.

  It was observed in \cite{CQ2} that $\kap$ has finite order on
  $\sd\oom$ and $\oom/\msb\oom$. It follows that $\kap$ has finite
  order also on $\oom/\sd\oom\cong\sd\oom$ 
  and on $\msb\oom = \msb(\oom/\msb\oom)$.  Each of these spaces is
  $\kap$-stable.  The first equation in \eqref{tech3} and equation
  \eqref{tech4} follow. Similarly, we deduce that the operator
  $\Id-\kap$ annihilates $P\oom/P\msb\oom$, i.e., $(\Id-\kap)P\oom
  \sset \msb P\oom$.

  Formula \eqref{kap} readily implies the second equation in
  \eqref{tech3}; it shows also that $(\Id-\kappa)\oOm=\overline{[\sd\!
    A,\Om]}$.  Similarly, from formula \eqref{e:bfla} we get that
  $\msb \oOm=\overline{[A,\Om]}$.  Thus we obtain (cf.~\cite{CQ2}):
\begin{equation} \label{bracket1}
\overline{[\Om,\Om]} = \overline{[A,\Om]} +\overline{[\sd\!A,\Om]} =
\msb\oOm+(\Id-\kappa)\oOm.
\end{equation}
Applying $P$ to this equation and using the inclusion $(\Id-\kap)P\oom
\sset \msb P\oom$ proved above, we deduce the first equation in
\eqref{tech1}. The second equation in \eqref{tech1} is clear from
\eqref{e:iP}.

Finally, since $\Id-\kappa$ is invertible on $P^\perp\oOm$, formula
\eqref{bracket1} shows that $P^\perp\oOm \subseteq (\Id-\kappa)\oOm
\subseteq \overline{[\Om,\Om]}$. This yields the inclusion
$P^\perp\oOm \subseteq P^\perp\overline{[\Om,\Om]}$. The opposite
inclusion is obvious, proving \eqref{tech2}.
\end{proof}

\section{Main constructions}\label{s:mainconstr}

\subsection{Construction of Hochschild and cyclic homology}
Let $\oOm A \bbu$ be the space of formal Laurent series with
coefficients in $\oOm A$.  We assign the variable $u$ degree $-2$, so
that the total degree of an element $f\in\oOm^pA\cdot u^{-r}$ equals
$|f|=p+2r$.  Thus, each of the differentials $u\msB$ and $\msb$ has
degree $-1$.

Standard constructions of various versions of cyclic homology involve
the $\bk[u]$-modules $\bk[\![u]\!]\sset\bk\bbu$ and $R :=
\bk(\!(u)\!)/u\bk[\![u]\!]$.  Specifically, reduced {\em relative} (to
$\bk$) Hochschild and cyclic homology of $A$ are defined as follows
\begin{align}
  \oHH_\idot(A) &= H(\oOm A,\, \msb), & \oHC_\idot(A) &=
  H(\oOm A \otimes R,\  \msb - u \msB),\label{e:gs-hh} \\
  \oHC^\per_\idot(A) &= H(\oOm A\bbu,\ \msb - u \msB), &
  \oHC^-_\idot(A) &= H(\oOm A[\![u]\!],\ \msb - u \msB).\label{e:hhh}
\end{align}

A central result of this article is 
\begin{theorem}\label{HHHC} For a $\bk$-algebra
$A$ satisfying our standing assumptions,
one has canonical isomorphisms:
\begin{align}
\HH_\idot(A) \ &\iso\ \ker(\si: \oDR^\hdot \to
\oOm^{\hdot-1}).\label{hhdef}\\
\HC_\idot(A)\ &\iso\ \ker(\si: \oDR^\hdot / \sd\! \oDR^{\hdot-1} \to
 \oOm^{\hdot-1} 
 / \sd\! \oOm^{\hdot-2} ).\label{hcdef}
\end{align}
Here, the first isomorphism is induced by the map $\oOm \to \oDR,\
f\mto f\cc$ and the second isomorphism is induced by the map $\oOm \o
R \to \oDR,\ \sum_{k\geq0}\ f_ku^{-k} \mto (f_0)\cc$.
\end{theorem}

Isomorphism \eqref{hhdef} has been already established 
 in \cite[Theorem 4.1.1]{GScyc}. However, we will reproduce the proof,
 since most of the intermediate steps will also be used
in the proof of ~\eqref{hcdef}.

\proof[Proof of Theorem \ref{HHHC}] First of all, from harmonic
decomposition for $\oom$ and $\overline{[\Om,\Om]}$, respectively,
using equation \eqref{tech2}, we see that the projection $\oom\onto
P\oom$ induces an isomorphism $\oom/\overline{[\Om,\Om]}\ \iso$
$P\oom/P\overline{[\Om,\Om]}$.  Further, equation \eqref{tech1} yields
$P\oom/\msb P\oom=P\oom/P\overline{[\Om,\Om]}$.  Also, the map $\si$
annihilates the space $P^\perp\oOm$, by \eqref{e:iP}. Thus, we obtain
a commutative diagram
\begin{equation}\label{e:pdr}
  \xymatrix{
    P\oom/\msb P\oom\ \ar[dr]_{\msb\sN}\ar@{=}[r]&
    \ P \oOm / P \overline{[\Om,\Om]}\ \ar@{=}[r]\ar[d]^{\bi}& 
    \ \oOm /\overline{[\Om,\Om]}\ \ar@{=}[r]&\ \oDR \ar[d]^\si \\
    & \ P \oOm\  \ar@{^{(}->}[rr] &&\  \oOm.
  }
\end{equation}

We see from the diagram that the assignment $f\mto f\cc$ maps the
space $H(P\oOm, \msb)\ =\ \ker(\msb: P\oOm / \msb P\oOm \to P \oOm)$
isomorphically onto the space $\ker(\si: \oDR \to \oOm)$.

 Next, since $(\oOm, \sd)$ is acyclic, it follows that its harmonic
 and anti-harmonic parts, $(P \oOm, \sd)$ and $(P^\perp \oOm, \sd)$,
 are acyclic as well.  Hence also $(P \oOm, \msB)$ is acyclic, by
 \eqref{connes}.  Cuntz and Quillen observed in \cite{CQ2} that this
 implies that the cohomology of all of the complexes appearing in
 \eqref{e:gs-hh}--\eqref{e:hhh} are not affected by the replacement of
 the space $\oOm$ by its harmonic part $P \oOm$. Specifically, it
 follows from the above, using equations \eqref{CQid1} and
 \eqref{connes}, that the following complexes are acyclic:
\[
(P^\perp \oOm, \msb),\quad (P^\perp \oOm(\!(u)\!), \msb - u
\msB),\quad (P^\perp \oOm[\![u]\!], \msb - u \msB),\quad (P^\perp
\oOm\o R, \msb - u \msB).
\]

Thus, the assignment $f\mto f\cc$ gives an isomorphism
$H(P^\perp \oOm, \msb)\iso0$, and \eqref{hhdef} follows.

To proceed further, we need the following result of Cuntz and Quillen,
\cite[Proposition ~3.1]{CQ2}.

\begin{lemma}\label{cq} The projection
modulo $u^{-1}$ and $\im(\msB)$ induces an isomorphism 
\begin{equation}\label{cqmap}
  \oHC_\idot(A)  = H_\idot(\oOm \otimes R, \msb - u \msB)\ \iso\
  H_\idot(\oOm/\msB\oOm, \msb),
  \quad \mbox{$\sum$}_{k\geq0}\ f_k\cdot u^{-k}\ \mto\ f_0\mmod \msB\oOm.
\end{equation}
\end{lemma}
\begin{proof}
Introduce an increasing filtration on $\oOm \otimes R$ 
as follows:
$$F_j (P\oOm \otimes R)\
:=\ P\oOm \otimes (u^{-j}\bk[\![u]\!]/u\bk[\![u]\!])\ =\
P\oOm\o\Span(u^{-j}, u^{-j+1}, \ldots, u^{-1}, 1).$$

The fact that the complex $(P\oOm \otimes R,\ \msB)$ is acyclic
implies that the standard spectral sequence associated to our
filtration collapses at the second page to $H(P\oOm / \msB P \oOm,
\msb)$.  We conclude that the map $\sum\nolimits_{k\geq0}\ f_k\cdot
u^{-k}\ \mto\ f_0$ yields an isomorphism $H(P\oOm \otimes R,\
\msb-u\msB)\iso H(P\oOm / \msB P \oOm,\ \msb)$, of harmonic
components.  The same map also gives an isomorphism $H(P^\perp\oOm
\otimes R,\ \msb-u\msB) \iso H(P^\perp\oOm / \msB P^\perp \oOm,\
\msb)$, of anti-harmonic components. Since $\msB(P^\perp \oOm)=0$, the
second cohomology group is $H(P^\perp\oOm, \msb)=0$, so these
cohomology groups vanish.
\end{proof}



The isomorphism in \eqref{hcdef} is now obtained as a composition:
\begin{align*}
  \oHC(A) &=\ H(P\oOm / \msB P\oOm,\ \msb)
  \qquad\text{by  the proof of Lemma \ref{cq}}\\
  &=\
  \frac{\{f\in P\oOm \mid \msb f\in \msB P\oOm\}}{\msb P\oOm +\msB P \oOm}\\
  &=\ \ker(\msb: P\oOm / (\msb P \oOm + \msB\! P \oOm) \to P \oOm /
  \msB\! P
  \oOm)\\
  & =\ \ker\bigl(\si: (P\oOm /\si P \oOm)/\sd(P\oOm /\si P \oOm) \to P
  \oOm / \sd\! P \oOm\bigr)
  \quad\text{by  \eqref{e:iP}--\eqref{connes}}\\
  &=\ \ker(\si: \oDR / \sd\! \oDR \to \oOm / \sd\! \oOm),
  \qquad\text{by \eqref{e:pdr}.} \qedhere
\end{align*}

\subsection{Negative and periodic cyclic
  homology}\quad \label{ss:negper}
Following \cite{GScyc}, we consider the bicomplex $(\oOm\bbu,$ $ u\sd,
\bi),$ where each of the differentials $u\sd$ and $\bi$ has degree
$-1$. The result below provides an interpretation of periodic,
negative, and ordinary cyclic homology in terms of the noncommutative
de Rham complex (in the latter case, this differs from the description
of the previous section).

Note that the space $\overline{[\sd\! \Om , \sd\! \Om]} \subseteq \oOm
$ is annihilated by each of the differentials $\sd$ and $\bi$. We view
it also as a subspace of $\oOm[\![u]\!]$ (and $\oOm(\!(u)\!)$), as
constant terms in $u$.
\begin{theorem}\label{t:perneg}
There are natural isomorphisms
\begin{align}
  \oHC^\per_\idot(A) \label{ccper} &\cong H_{\idot}(\oOm \bbu, \bi -
  u\sd)
  \\
  \oHC^-_\idot(A) \label{hc-} &\cong H_\idot(\oOm [\![u]\!], \bi-u
  \sd)/\overline{[\sd\! \Om , \sd\! \Om]}
  \\
  \oHC_\idot(A) \label{e:bihc} &\cong H_\idot(\oOm \otimes R, \bi -
  u\sd)/\overline{[\sd\!\Om,\sd\!\Om]}.
\end{align}
\end{theorem}
In the case of \eqref{hc-}, the $\overline{[\sd\! \Om, \sd\!\Om]}$
includes into $H_\idot(\oOm [\![u]\!], \bi-u \sd)$ via the
aforementioned inclusion into $\oOm[\![u]\!]$ as constant terms in
$u$. As we will see in the proof of theorem, the target coincides with
the antiharmonic part of $H_\idot(\oOm [\![u]\!], \bi-u \sd)$, and
hence the inclusion is canonically split.

In \eqref{e:bihc}, the inclusion of $\overline{[\sd\! \Om, \sd\!\Om]}$
into $H_\idot(\oOm \otimes R, \bi - u\sd)$ is more complicated, but
can also be given explicitly (in the proof of the theorem, we will see
that the target is the antiharmonic part, so this inclusion is also
canonically split). The inclusion in question is defined as follows.

Observe first that the map $\osd: \oOm / \sd \oOm \to \sd \oOm$,
induced by $\sd$, is an isomorphism since the complex $(\oOm, \sd)$ is
acyclic.  Write $\osd^{-1}: \sd \oOm \iso \oOm / \sd \oOm$ for an
inverse isomorphism.  Further, let $\sdi: \sd \oOm \to \oOm$ be a lift
of $\osd^{-1}$.  In other words, we fix an arbitrary (set-theoretic)
section, $\sdi$, of the surjection $\sd: \oOm \onto \sd\oOm$.

With the above notation, the inclusion used in \eqref{e:bihc} is given
by
\[ \overline{[\sd\! \Om, \sd\!\Om]}\ \into\ H_\idot(\oOm \otimes R,
\bi - u\sd),\quad f \mapsto (\Id - u^{-1} \sdi \bi)^{-1} \sdi f.
\]
One can show that this formula makes sense and the result is
independent of the choice of a section $\sdi$, in a manner similar to
the proof of Proposition \ref{inverse} below.


In discussing Connes sequences involving these groups below, we will
not use \eqref{e:bihc} but rather the description of \eqref{hcdef}.
Isomorphism \eqref{ccper} was already proved in \cite[Theorem
4.2.2]{GScyc}; we included it in the above theorem for completeness
only.  For a related discussion of negative cyclic homology see also
\cite{GScyc}.

\begin{proof}[Proof of \eqref{hc-} and \eqref{e:bihc}]
  We first deal with \eqref{hc-}. Similarly to the arguments in the
  proof of Theorem \ref{HHHC},
\[
\oHC^-(A)=H(\oOm[\![u]\!], \ \msb-u\msB)=
H(P\oOm[\![u]\!], \ \msb-u\msB) =
H(P \oOm[\![u]\!],\ \bi - u \sd).
\]

Now, thanks to equation \eqref{tech3}, we have embeddings
$\overline{[\sd\! \Om, \sd\! \Om]} \iso P^\perp \sd\oOm \into P^\perp
\oOm \into P^\perp \oOm[\![u]\!]$.  To complete the proof, it suffices
to show that the composite embedding yields a quasi-isomorphism of
complexes
\begin{equation}\label{incl}
  \overline{[\sd\! \Om, \sd\! \Om]} \to (P^\perp \oOm[\![u]\!], \bi - u \sd),
\end{equation}
where the LHS is considered as a complex with zero differential.

To this end, we observe that the operator $\bi$ vanishes on
$P^\perp\oOm$ and the complex $(\oOm, \sd)$ is acyclic.  Therefore, we
find
\[
H_\idot(P^\perp \oOm[\![u]\!], \bi - u \sd) = H_\idot(P^\perp
\oOm[\![u]\!], -u \sd) \ \cong\ \sd\! P^\perp \oOm^\hdot.
\]
The term on the right equals $\overline{[\sd\! \Om, \sd\! \Om]}$
by \eqref{tech3}, proving that \eqref{incl}
is a quasi-isomorphism.

For the isomorphism \eqref{e:bihc}, we can apply the same argument as
above to conclude that
\[
\oHC(A)=H(P\oOm \otimes R, \ \msb-u\msB)= H(P\oOm \otimes R, \
\bi-u\sd).
\]

Then, since $\bi$ vanishes on $P^\perp\oOm$ and the complex
$(P^\perp\oOm,\sd)$ is
 acyclic, we obtain
$$H(P^\perp\oOm \otimes R,\,\sd)\ =\ 
P^\perp\oOm/u\sd(u^{-1} P^\perp\oOm)\ \cong\ P^\perp\oOm/\sd
P^\perp\oOm \ \mathop{\iso}^{\osd}\ \sd P^\perp\oOm\ \cong\
\overline{[\sd\! \Om, \sd\! \Om]}.
$$

Finally, note that the resulting isomorphism $H(P^\perp \oOm \otimes
R, \sd) \cong \overline{[\sd\! \Om, \sd\! \Om]}$ is indeed inverted by
the claimed formula $(\Id - u^{-1} \sdi \bi)^{-1} \sdi$, provided the
chosen section $\sdi$ respects harmonic decomposition (or at least
preserves antiharmonic forms). In fact, in this case, since $\bi$ is
zero on antiharmonic forms, the resulting map reduces to $\sdi =
P^{\perp} \sdi$. On the other hand, as observed above, the formula
does not actually depend on the choice of $\sdi$.
\end{proof}

\section{The Connes exact sequence} \label{s:connes}
\subsection{Connes exact sequence via noncommutative differential
  forms}
We now construct a version of Connes' exact sequence, involving
reduced homology groups, using the interpretation for cyclic and
Hochschild homology provided by Theorem \ref{HHHC}.

\begin{theorem}\label{con}
There is an exact sequence,
\begin{equation}\label{e:connes-seq}
  0 \to \oHD_n(A) \mathop{\too}^{S_2} \oHC_n(A) \mathop{\too}^B 
  \oHH_{n+1}(A) \mathop{\too}^I \oHC_{n+1}(A) \mathop{\too}^{S_1} 
  \oHD_{n-1}(A) \to 0,
\end{equation}
where the maps $S_2, B, I$, and $S_1$ are defined by the following
assignments:
\begin{equation}\label{e:connes-seq-maps}
  S_2 f\cc =  f\cc, \quad B f\cc = (\sd\!f)\cc , \quad I f\cc =  f\cc, 
  \quad S_1( f\cc) = (\osd^{-1}\si f\cc)\cc.
\end{equation}
\end{theorem}
Note here that the map $B: \oHC_n(A) \to \oHH_{n-1}(A)$ is
distinguished by its font from the differential $\msB$ on $\Om$ and
$\oOm$.  We let $S := S_2 \ccirc S_1:\ \oHC_n(A) \to
\oHC_{n-2}(A)$. This is the standard periodicity map which splices the
above exact sequence into the more familiar Connes' long exact
sequence.  The exact sequence of Theorem \ref{con} incorporates
additional information, implying that $\oHD_\idot(A)=\ker(B)$ (the
Connes-Karoubi theorem).

We now prove that the above maps are well-defined.  We will prove that
they coincide with the usual definitions in \S \ref{ss:comp-maps}, and
give a direct proof of exactness (using the right hand sides of
\eqref{hhdef}--\eqref{hcdef} as definitions) in \S \ref{ss:exact}.

The map $S_2$ is obtained from the clearly well-defined map
$\ker(\sd|_{\oDR})/\sd \oDR \to \oDR / \sd\! \oDR$; we need to check
that the image is contained in the kernel of $\si$.  Indeed, if
$(\sd\!f)\cc = 0$, then $\sd \si f\cc = -\si (\sd\!f)\cc = 0$, and
hence $\si f\cc \in \sd\! \oOm$, as required.

We now show that $B$ is well-defined. First note that, if $ f\cc \in
\oDR / \sd\! \oDR$, then $(\sd\!f)\cc \in \oDR$ makes sense
independently of the choice of representative element $f \in \oOm$.
Next, if $\si f\cc \in \sd\! \oOm$, then $\si(\sd\!f)\cc = -\sd \si
f\cc = 0$, so $(\sd\!f)\cc \in \oDR$ defines a class in Hochschild
homology.

It is immediate that $I$ is well-defined.

To prove that $S_1$ is well defined, we compute
$$
\osd^{-1} \si \sd \oDR = \osd^{-1} \sd \si \oDR = (\si \oDR +
\sd\oOm)/\sd\oOm.
$$
Since $\si \oDR \subseteq \overline{[A, \Om]}$, by \eqref{si}, the
right most term above projects to zero in $\oDR / \sd \oDR$. It
follows that the formula for $S_1$ gives a well defined map $S_1:\
\oHC(A) \to \oDR/\sd\oDR$. It remains to check that $\sd S_1(f)=0$ for
every $f \in \oDR$ such that $\si f\in\sd\oOm$.  Indeed, $\sd
S_1(f)=\sd (\osd^{-1} \si\!f)\cc = (\si\!f)\cc =0$, since $\si \oDR
\subseteq \overline{[A, \Om]}$. \qed

\subsection{Comparison of maps in \eqref{e:connes-seq} with the usual
  definitions} \label{ss:comp-maps} We now show that the maps in
\eqref{e:connes-seq} coincide with the usual ones up to nonzero
integer multiples depending on degree (i.e., a nonvanishing expression
involving the aforementioned degree operator $\sN$). We will make use
of the fact that every homology class can be represented by a harmonic
cycle.

Here and later on we will make use of the notation $f_0$ for the
constant term in $u$ of an element $f \in \oOm\dpu$ (or $f \in \oOm
\dspu$).

\subsubsection{The map $B$}
If $f$ is a harmonic cycle with respect to the usual definition of
Hochschild or cyclic homology, then \eqref{e:iP} implies that $\msB f
= \sN \sd\! f$. Since the usual $B$ is essentially defined by applying
$\msB$ and our version is defined by applying $\sd$, this shows that
the difference is post-composition with $\sN$, which is nonvanishing
since $\sd$ has de Rham degree $1$.

\subsubsection{The map $I$}
The map $I$ is defined by leaving the harmonic cycle representing a
homology class unchanged, and mapping it from one complex to another
(throwing out the coefficients of $u^{> 0}$ in the case of $I$ acting
on the usual definition of cyclic homology).  Since the isomorphisms
from ordinary versions to our versions of Hochschild and cyclic
homology also leave the harmonic cycles unchanged, we obtain the
desired compatibility.

\subsubsection{The maps $S_1$ and $S_2$}
Let us consider $S_1$.  Suppose that $f \in \oOm(\!(u)\!)/u
\oOm[\![u]\!]$ is a harmonic cycle of degree $n+1$, i.e., $|f|=n+1$
under the total degree in which $|\oOm^j| = j$ and $|u|=-2$.  The
usual periodicity operation is $f \mapsto u \cdot f$.  We need to show
that $\si (f_0)\cc = \sd((uf)_0)\cc$.  Since $(\bi - u \sd) f = 0$,
this is immediate.

Regarding the map $S_2$, we only need to observe that the usual
formula for $S_2$ is uniquely determined by the fact that the part
mapping to a coefficient of $u^0$ is the obvious map (leaving harmonic
cycles unchanged).  (To obtain the coefficients of $u^{<0}$, we can
use the usual periodicity map $S = S_2 \ccirc S_1$, which is
multiplication by $u$. Then, our construction yields the explicit
formula ~${\osd^{-1}\si}\ $\ for this operation, by the previous
paragraph.)

\subsection{Direct proof of exactness}\label{ss:exact}
Since we showed that the maps in \eqref{e:connes-seq} are compatible
with the isomorphisms between the new versions of cyclic and
Hochschild homology and the usual ones, exactness follows from the
usual construction of \eqref{e:connes-seq} by a commutative diagram of
short exact sequences of the complexes computing Hochschild and cyclic
homologies (cf.~\cite{L}).

Nonetheless, here we will prove exactness directly using our formulas
\eqref{hhdef}--\eqref{hcdef}.

\underline{Injectivity of $S_2$:} By definition, $\oHD(A) \subseteq
\oDR/\sd\! \oDR$.

\underline{$\im(S_2) = \ker(B)$:} By definition, the cycles in
$\ker(B)$ are represented by $ f\cc \in \oDR^n$ such that $\sd f\cc =
0$ and $\si f\cc \in \sd\! \oOm$. The second condition is equivalent
to $\sd \si f\cc = 0$, which follows from the first condition. We
deduce that $\ker(B) = \oHD_n(A)$, as desired.

\underline{$\im(B) = \ker(I)$:} The cycles in $\ker(I)$ are those
elements $ f\cc \in \oDR^n$ such that $\si f\cc = 0$ and $ f\cc \in
\sd\! \oDR^{n-1}$.  Given the second condition, we can assume $f =
\sd\!g$. Then $\si (\sd\!g)\cc = \si \sd g\cc = 0$ holds if and only
if $\si g\cc \in \sd\! \oOm$, i.e., $ g\cc \in \oHC_{n-1}(A)$.
Conversely, it is clear that $\im(B) \subseteq \ker(I)$. We conclude
that $\ker(I) = \im(B)$.

\underline{$\im(I) = \ker(S_1)$:} Every cycle in $\ker(S_1)$ is
represented by a harmonic element $ f\cc \in \oDR$ such that $\si f\cc
\in \sd \overline{[\Om,\Om]}$. By \eqref{tech1}, $P
\overline{[\Om,\Om]}=\bi\oom$. We deduce that
$$\si f\cc\ =\ P\si f\cc\ \in\  P\sd 
\overline{[\Om,\Om]}\ =\ \sd P \overline{[\Om,\Om]}\ = \ \sd \bi\oom\
=\ \bi\sd\!\oOm.
$$
It follows that one can find $g,h\in\oom$ such that $\si( g\cc )=0$
and $ f\cc= g\cc +\sd h\cc $. Thus, in $\oHC(A)$, we obtain $f\cc
\mmod \sd\oDR =I( g\cc )$, as desired.

 \underline{Surjectivity of $S_1$:} Every class in $\oHD(A)$ is
 represented by a harmonic element $f\in\oom$ such that $\sd f\cc =
 0$.  This means that $\sd\! f \in P\overline{[\Om,\Om]}=\bi\oom$,
 thanks to \eqref{tech1}. Thus, there exists $g\in\oom$ such that, in
 $\oHD(A)$, we obtain $ f\cc=\osd^{-1}\bi(g)=S_1( g\cc ),$ as desired.

 \subsection{Connes' exact sequences for periodic and negative cyclic
   homology}
Now, using Theorem \ref{t:perneg} (for periodic and negative cyclic
homology) as well as Theorem \ref{HHHC} (for Hochschild and ordinary
cyclic homology), we construct versions of Connes' exact sequences
involving periodic and negative cyclic homology.

It will be convenient below to introduce, for each $n=0,1,\ldots,$ the
following space of degree $n$ elements in $\oOm[\![u]\!]$:
\begin{align}
  \sZ_n(A) &:=\{f\in (\oOm[\![u]\!])_n\mid (\bi-u \sd)f\quad\text{is
    independent of $u$}\}\label{sz}\\
  &\,=\{f=f^n+u f^{n+2}+u^2 f^{n+4}+\ldots\enspace\big|\enspace
  f^j\in\oOm^j,\ \sd\!f^j=\bi(f^{j+2}),\enspace \forall
  j=n,n+2,\ldots\}.  \nonumber
\end{align}

We observe that, for every element $f=f^n+u f^{n+2}+u^2 f^{n+4}+\ldots
\in \sZ_n(A)$, in $\oDR^{n+1}$, one has
$\sd(f^n)\cc=(\bi(f^{n+2}))\cc=0$.  Therefore, the assignment sending
$f=f^n+u f^{n+2}+u^2 f^{n+4}+\ldots$ to $(f^n)\cc\in \oDR^n$ gives a
well defined map $I':\ \sZ(A)\to \{\alpha\in \oDR \mid
\sd\!\alpha=0\}$.  In particular, one has an induced map
$$I'_{\HD}:\ \sZ(A)\ \to\ \oHD(A)=\{\alpha\in \oDR \mid \sd\!\alpha=0\}
/\sd\oDR.
$$

Next, it is clear from the definition that, for $f=f^n+u f^{n+2}+u^2
f^{n+4}+\ldots \in \sZ_n(A)$, the condition $(\bi-u \sd)f=0$ is
equivalent to $\bi(f^n)=0$. Also, since $\bi\oOm$ projects to zero in
$\oDR$, the map $I'$ sends the subspace $(\bi-u
\sd)\oOm[\![u]\!]\subseteq \sZ(A)$ of boundaries to zero.  Hence,
using \eqref{hhdef}, we deduce that the map $I'$ gives a well defined
morphism
\[I'_{\HH}:\ H_\idot(\oOm[\![u]\!],\  \bi-u \sd)\
\to\ \oHH_\idot(A).
\]

We define
\begin{align}
  \oHD'(A)\  &:=\ \im[I'_{\HD}:\ \sZ(A)\to \oHD(A)]\nonumber\\
  \oHH'(A)\ &:=\ \im[I'_{\HH}:\ H_\idot(\oOm[\![u]\!],\ \bi-u \sd)\
  \to\ \oHH(A)].\label{IHH}
\end{align}

\begin{theorem}  There is 
 a commutative diagram with exact rows,
\begin{equation}\label{e:bigdiag}
  \xymatrix{
    0\, \ \to\ \,\oHH_n'\ \,\ar@<4mm>[d]^{J} \ar[r] &
    \ \oHH_n\  \ar[r]^{B^-} \ar[d]^I & \ \oHC^-_{n+1}\  \ar[r]^{S^-}
    \ar@<-1mm>@{=}[d] & \ \oHC^-_{n-1}\  \ar[r]^{I^-} \ar[d]!<25pt,0cm>^{p^-} &
    \ \, \oHH'_{n-1}\ \, \to\ \, 0  \ar@<-6mm>[d]^{J}  \\
    0 \ \to\ \oHD'_n\  \ar@<4mm>@{^{(}->}[d] \ar[r]^<>(0.5){p_2} & 
    \ \oHC_{n}\  \ar[r]^{B'} \ar@{=}[d] &\ 
    \oHC^-_{n+1}\  \ar[r]^-{p^-}
    \ar@<-1mm>[d]^{I^-} &\  \oHC^\per_{n+1} \ {\displaystyle \mathop{=}^{\cdot
        u}} \ \oHC^\per_{n-1}\   \ar[r]^-{p_1} & \ \oHD'_{n-1} \ \to\ 0
    \ar@<-6mm>@{^{(}->}[d]  \\
    0\, \ \to\ \, \oHD_n\  \ar[r]^<>(0.5){S_2} & 
    \ \oHC_n \ \ar[r]^-B & \ \oHH_{n+1} \ \ar[r]^I & \ \oHC_{n+1} \ \ar[r]^<>(0.5){S_1}
    \ar@{<-}[u]!<-25pt,0cm>^p & \ \oHD_{n-1} \, \to\,  0
  }
\end{equation}
\end{theorem}

The maps in the diagram are defined by the assignments
\begin{gather}
  p_2 f\cc = f\cc, \quad B^-( f\cc) = \sd\!f, \quad B'( f\cc) =
  \sd\!f,
  \quad S^-(f) = uf, \quad p^-(f)=f,\nonumber \\
  I^-(f) = (f_0)\cc , \quad p_1(f) = (f_0)\cc , \quad J f\cc = f\cc,
  \quad p = p_2 \ccirc p_1.  \nonumber\end{gather}

Let us explain why all the maps above are well-defined.  For $B^-$, it
is immediate that $B^- = B' I$, so it is enough to show that $B'$ is
well-defined.  To this end, let $f \in \oOm$ satisfy $\si f\cc \in
\sd\! \oOm$. Then, $(\bi - u \sd) \sd\!f = 0$, so that $\sd\!f$ indeed
defines a class in negative cyclic homology.  We then need to show
that $\sd \overline{[\Om,\Om]} \subseteq (\bi - u \sd) \oOm[\![u]\!] +
\overline{[\sd\!\Om,\sd\!\Om]}$.  To prove this, we write
$\overline{[\Om,\Om]}=P\overline{[\Om,\Om]}+P^\perp\overline{[\Om,\Om]}$.
We know that $P\overline{[\Om,\Om]}=\bi\oom$, by \eqref{tech1}.
Further, we compute
$$P^\perp\overline{[\Om,\Om]}\ \mathop{=}^{\eqref{bracket1}}\ 
\ P^\perp\msb\oom+P^\perp(\Id-\kap)\oom\ \mathop{=}^{\eqref{tech4}}\
(\Id-\kap)\msb\oom+(\Id-\kap)P^\perp\oom\ \sset \ (\Id-\kap)\oom \ =\
\overline{[\sd\!A, \Om]}.
$$
Thus, we have proved the inclusion we need:
$\overline{[\Om,\Om]}\sset\bi\oom +\overline{[\sd\Om, \Om]}$.

It is immediate that $S^-$ is well-defined.  To show that $p^-$ is
well-defined, we only need to show that the space
$\overline{[\sd\!\Om,\sd\!\Om]}$ is contained in the image of the
differential $\bi-u\sd$, on $\oOm\bbu$.  This follows from the fact
that $\bi \overline{[\Om, \sd\! \Om]} = 0$.

For the map $I^-$, we observe that this map is induced by the map
\eqref{IHH}. This makes sense since the map $I'$ annihilates the space
$\overline{[\sd\Om,\,\sd\Om]}\subseteq \sZ(A)$.

Finally, it is immediate that $p_1$ and $J$ are well-defined, and that
the diagram commutes.

\subsection{Comparison of the maps in \eqref{e:bigdiag} with the
usual definitions}
The fact that $B^-$ and $B'$ coincide with the usual definitions is
the same as the reason we gave for $B$ in \S \ref{ss:comp-maps}: these
maps take harmonic representatives $f$ to $\sd\!f$ under our
definitions, and to $\msB f$ in the usual definitions, which differ by
a scalar depending on degree. Similarly, for $I^-$ and $J$, these maps
leave harmonic representatives unchanged.

The fact that the periodicity map $S^-$ is compatible with our
isomorphism is obvious, since it is defined by multiplication by $u$
in both cases.

Finally, the maps $p$ and $p^-$, like the case of the maps of type
$I$, send harmonic cycles to the classes represented by the same
harmonic cycles (in the case of $p$, taking the coefficients of
$u^{\leq 0}$ or ~$u^0$).
\subsection{Direct proof of exactness in diagram \eqref{e:bigdiag}}
\subsubsection{The middle row of \eqref{e:bigdiag}}
\underline{Injectivity of $p_2$:} This is clear from construction (or
from injectivity of $S_2$).  \pagebreak[3]

\underline{$\im(p_2) = \ker(B')$:} The cycles in $\ker(B')$ are
represented by $ f\cc \in \oHC(A)$ such that $\sd\!f = (\bi - u \sd)g
+ h$ for some $g \in \oOm[\![u]\!]$ and some $h \in
\overline{[\sd\!\Om,\sd\!\Om]}$. Restricting to harmonic $f$ and $g$,
by \eqref{tech3} we can also set $h=0$. Then, $(\bi - u \sd)(f + ug) =
\si f\cc \in \sd\! \oOm$.  So $\ker(B') \subseteq \im(p_2)$. For the
opposite inclusion, if $f \in \sZ(A)$, then $(\bi - u\sd) f =
\si(f_0)\cc \in \sd \oOm$.  Thus, $\sd\!f_0 = (\bi - u \sd) u^{-1}
(f_0-f)$, so $(f_0)\cc \in \ker(B')$.

\underline{$\im(B') = \ker(p^-)$:} The elements in $\ker(p^-)$ are
represented by harmonic elements $f \in \oOm[\![u]\!]$ such that $f =
(\bi - u \sd) g$ where $g \in \oOm(\!(u)\!)$ is harmonic. Replacing
$f$ with $f - (\bi - u \sd) \sum_{j \geq 0} g_j u^j$, we can assume
that $g_j = 0$ for $j \geq 0$.  Then, $f = f_0 = \sd\! g_{-1}$, and
$\si (g_{-1})\cc \in \sd\! \oOm$.  Hence, $(g_{-1})\cc \in \oHC(A)$
and $f = B' (g_{-1})\cc$. Therefore, $\ker(p^-) \subseteq
\im(B')$. For the opposite inclusion, suppose that $f \in
\im(B')$. That is, there exists $h \in \oOm$ satisfying $f = \sd\! h$
and $\si h\cc \in \sd\! \oOm$. Now, we can set $g_{-1} := h$, and
inductively we can choose $g_{-j}$ for $j > 1$ such that $\sd g_{-j} =
\bi g_{1-j}$ for all $j > 1$.  We deduce that $-(\bi - u \sd)g = \sd\!
h = f$.  Therefore, $\im(B') \subseteq \ker(p^-)$.

\underline{$u\im(p^-) = \ker(p_1)$:} Elements of $\ker(p_1)$ are $f
\in \oOm(\!(u)\!)$ such that $f_0 \in \sd\! \oOm +
\overline{[\Om,\Om]}$ and $(\bi - u \sd) f = 0$. We claim first that
$f_{j} \in \sd\! \oOm + \overline{[\Om,\Om]}$ for all $j \leq 0$.
Indeed, $\sd\! f_j = \si(f_{j+1})\cc$ for $j < 0$.  So, if $f_{j+1}
\in \sd\! \oOm + \overline{[\Om, \Om]}$, then $\sd\! f_j =
\si(f_{j+1})\cc \in \bi \sd\! \oOm = \sd \bi \oOm \subseteq \sd
\overline{[\Om,\Om]}$. By acyclicity of $(\oOm, \sd)$, we conclude
that $f_j \in \sd\! \oOm + \overline{[\Om, \Om]}$ as well, so by
induction, we conclude the desired result.

Suppose further that $f$ is harmonic.  Then, $f_0 \in \sd\! P \oOm +
\bi P \oOm$ since $P \overline{[\Om,\Om]}=\bi P \oOm$, by
\eqref{tech1}. Writing $f_0 = \sd\! g + \si h\cc $, we conclude that
$f' := f + (\bi - u \sd)(u^{-1} g - h)$ satisfies $f'_0 = 0$, and
$f'_j \in \sd\! \oOm + \overline{[\Om,\Om]}$ for $j < 0$. Applying the
same reasoning to $uf'$, we eventually conclude that $f$ is homologous
to a harmonic cycle $\tilde f \in u P\oOm[\![u]\!]$.  Therefore, $f
\in u \cdot \im(p^-)$.

The converse containment, $u \cdot \im(p^-) \subseteq \ker(p_1)$, is
obvious.

\underline{Surjectivity of $p_1$:} Suppose that $f \in \sZ(A)$.  Then
$(\bi - u\sd)f = \si(f_0)\cc \in \sd \oOm$.  As observed above,
setting $g_0 = f_0$, we can choose $g_{-j}$ for $j \geq 1$ such that
$\sd g_{-j} = \bi g_{1-j}$ for all $j \geq 1$.  Set $h := g +
(f-f_0)$.  Then, $(\bi - u \sd)h = -u \sd\!f_0 - \si(f_0)\cc + (u
\sd\!f_0 + \si (f_0)\cc ) = 0$.  Since $(h_0)\cc = (f_0)\cc $, we
conclude that $(f_0)\cc = p_1(h)$ is in the image of $p_1$.

\subsubsection{The top row of \eqref{e:bigdiag}}  

The first map is tautologically injective.

\underline{$\ker(B^-) = \oHH'(A)$:} Assume that $f$ is harmonic and
$\si f\cc=0$, and suppose that $ f\cc \in \oHH(A)$ satisfies $\sd\!f =
(\bi - u \sd)g$ for some harmonic $g \in \oOm[\![u]\!]$.  Then, $(\bi
- u \sd)(f - ug) = 0$, so $ f\cc \in \oHH'_n(A)$. Thus $\ker(B^-)
\subseteq \oHH'(A)$.  Conversely, if $f \in \ker(\bi - u \sd)
\subseteq \oOm[\![u]\!]$, then $\sd\!f_0 = (\bi - u \sd) (f - f_0)$,
so $(f_0)\cc \in \ker(B^-)$.

\underline{$\im(B^-) = \ker(S^-)$:} The kernel of $S^-$ is represented
by elements of the form $u^{-1} (\bi - u \sd)g$, where $g \in P
\oOm[\![u]\!]$ satisfies $\si(g_0)\cc = 0$. Up to coboundary, this is
the same as elements $\sd g_0$ such that $\si(g_0)\cc =0$, i.e., the
image of $B^-$.  Conversely, if $\si(g_0)\cc = 0$, then $B^-(g_0)\cc =
\sd g_0 = -u^{-1} (\bi - u \sd)(g_0)$ is in the kernel of $S^-$.

\underline{$\im(S^-) = \ker(I^-)$:} The kernel of $I^-$ is represented
by elements $f \in \oOm[\![u]\!]$ such that $f_0 \in
\overline{[\Om,\Om]}$.  If we assume that $f$ is harmonic, then $f_0
\in \msb P \oOm = \bi P \oOm$, see \eqref{tech1}. Write $f_0 = \si
g\cc $ for some $g$. Then, $f - (\bi - u \sd)g \in u
\oOm[\![u]\!]$. Hence, the class of $f$ is in the image of $S^-$.
Conversely, it is clear that the image of $S^-$ maps to zero under
$I^-$.

\underline{Surjectivity of $I^-$:} This is immediate from the
definition of $\oHH'(A)$.

\section{Extended Karoubi-de Rham complex
and cyclic homology}\label{ss:eqcoh}
In this section, we will consider, following \cite{GScyc}, an
``extended'' version, $\oDR_t A$, of the noncommutative de Rham
complex.  One of the reasons for introducing such an extended version
is that it maps naturally to the {\em equivariant} (as opposed to the
ordinary) de Rham complex of representation varieties of the algebra
in question, see \S\ref{s:eqcoh}.

Another important feature of $\oDR_t$ is that it comes equipped with
natural anti-commuting differentials $\sd$ and $\si_t$.  We will show
that the homology of the complex $(\oDR_t A \otimes R,\ \si_t-u\sd)$
captures both the cyclic homology $\oHC(A)$ and the Karoubi-de Rham
homology $\oHD(A)$ at the same time.

\subsection{Extended noncommutative de Rham complex}
Let $B *_\bk C$ denote the free product of unital algebras $B$ and $C$
over $\bk$. Define
\begin{equation}\label{omt}
  \Om_t = \Om_t A := \Om (A *_\bk \bk[t]) / (\sd\! t) \cong \Om A *_\bk \bk[t],
\end{equation}
where $(\sd\! t)$ denotes the two-sided ideal generated by the element
$\sd\! t$.  For all $p$ and $q$, let $\Om_t^{p,q} \subseteq \Om_t$
denote the subspace of de Rham degree $p$ and degree $q$ in $t$ (note
that this notation differs from that of \cite{GScyc}).  The
isomorphism on the right of \eqref{omt} respects the bigradings.

Next, we define the commutator quotient
$$
\DR_t = \DR_t A := \Om_t  / [\Om_t , \Om_t ],
\qquad \DR_t  =\bplus_{p,q}\ \DR_t^{p,q} 
\quad\text{where}\quad \DR_t^{p,q}=(\Om_t^{p,q})\cc.
$$

As observed in \cite{GScyc}, for each $q\geq 1$, there is a canonical
isomorphism $\DR_t^{\hdot,q}\cong \Om^{\otimes q}/(\bZ/q)$. This
yields a direct sum decomposition
\begin{equation}\label{oplus}
\DR_t =\bplus_{q\geq 0}\ \DR_t^{\hdot,q}\ \cong\
 \DR\  \oplus\ 
\Om \  \oplus\  \Om ^{\otimes 2}/(\bZ/2)\  \oplus \cdots.
\end{equation}
In particular, there is a canonical isomorphism $\iota:\
\DR_t^{\hdot,1}=(t \Om^\hdot)\cc\ \iso\ \Om^\hdot $.

One has a natural differential $\sd:\ \Om_t^{\hdot,q}\to
\Om_t^{\hdot+1,q} $ that restricts to the ordinary Karoubi-de Rham
differential on $\Om $, and kills $t$. One also has a derivation
$\ddt:\ \Om_t^{p,\hdot}\to \DR_t^{p,\hdot-1}$ that kills $\Om\sset
\Om_t$ and sends $t$ to $1$.  These maps descend to $\DR_t$.

There are also reduced versions of all the above.  In particular,
$\oOm_t=\oOm_t A := \Om_t / \bk[t]$ and $\oDR_t=\oDR_tA := \oOm_t
/\overline{[\Om_t, \Om_t]} = \DR_t /\bk[t]$.  Further, we put
$\oDR^+_t :=\oplus_{q>0}\ \oDR^{\hdot,q}= (t \oOm_t)\cc \subseteq
\oDR_t$.

\begin{lemma}\label{qis0}
The complex $(\oDR^+_t,\,\sd)$ is acyclic.
\end{lemma}\label{l:drt-acy}

\begin{proof} 
  From \eqref{oplus}, we see that $\DR^+_t=\Om \, \oplus\, \Om
  ^{\otimes 2}/(\bZ/2)\, \oplus\, \cdots$.  Now, $H(\Om, \sd) \cong
  \bk$, spanned by the natural inclusion $\bk \subseteq \Om$.  Taking
  coinvariants with respect to the cyclic group $\bZ/q$ is exact for
  all $q \geq 1$, thanks to our assumption that $\bk \supseteq \bQ$.
  Hence, $H(\DR^+_t,\sd) = \bigoplus_{q \geq 1} \bk^{\otimes q}$,
  spanned by the inclusions $\bk^{\otimes q} \cong \bk \cdot (t^q)\cc
  \subseteq \DR^+_t$.  The result now follows from the long exact
  sequence on homology associated to the short exact sequence $0 \to t
  \cdot \bk[t] \to \DR^+_t \to \oDR^+_t \to 0$.
\end{proof}
As an immediate consequence of Lemma \ref{qis0} we obtain

\begin{corollary}\label{qis1} The natural  projection 
  $\ \oDR_t \onto \oDR_t /\oDR^+_t=\oDR$, modulo $t$, induces an
  isomorphism $H(\oDR_t,\,\sd) \iso H(\oDR,\,\sd)$, of cohomology.
\end{corollary}

\subsection{}\label{zt}
Following \cite{GScyc}, we observe that the derivation $(-\ad t)$ of
the algebra $A*_\bk\bk[t]$ naturally gives rise to a contraction
operation $\bi_t: \Om_t^{p,q} \to \Om_t^{p-1,q+1} $. By definition,
the map $\bi_t$ is the (super) derivation of the algebra $\Om_t $ that
acts on generators by
\[
\bi_t:\ t \mto 0,\quad \Om^0=A \mto 0,\quad \Om^1\supset\sd\!
A\ni\sd\! a\mto [a,t]\quad\forall a\in A.
\]

One checks that $\bi_t^2=0$ and that $\bi_t$ descends to a map $\si_t:
\DR_t \to \DR_t $.  The operations $\bi_t$ and $\si_t$ also descend to
the corresponding reduced versions.

We are now able to give, following \cite{GScyc}, a more conceptual
definition of the map $\si: \DR^\hdot \to \Om^{\hdot-1}$ that was
introduced in \S\ref{sec2} using an explicit formula
\eqref{si}. Specifically, one easily checks that
\begin{equation}\label{i=i}\si = \iota \ccirc \si_t, 
  \quad\text{as    maps}\quad
  \DR^\hdot =\DR_t^{\hdot,0}\ \too\
  \Om^{\hdot-1}=\DR_t^{\hdot-1,1}.
\end{equation}

It follows immediately that the differentials $\sd$ and $\bi_t$
(hence, also $\si_t$) commute with $\ddt$, and that the differentials
$\sd$ and $\si_t$, on $\oDR_t$, anti-commute (see \cite{CBEG} and
\cite{GScyc} for the latter fact).

\subsection{The Tsygan map}\label{ss:tsygan-map}
Let $\ohDRt^\hdot = \ohDRt(A)^\hdot := \prod_{q\geq0}\
\oDR_t^{\hdot,q}$ be the $t$-adic completion of $\oDR_t$.  We define
the {\em Tsygan map} as a $\bk$-linear map $\EE:\ \oOm
\to\ohDRt^\hdot$ given, for any $a_0, a_1, \ldots, a_n \in A$, by
\begin{equation}\label{e:Tdefn}
  \EE(a_0 \sd\! a_1 \cdots \sd\! a_n) =
  \sum_{k_0, \ldots, k_n \geq 0}\ \frac{1}{(k_0 + \cdots + k_n+n)!}
  (a_0 t^{k_0} \sd\! a_1 t^{k_1} \cdots \sd\! a_n t^{k_n})\cc.
\end{equation}

The map $\EE|_{\oOm^n}$ may be viewed as a noncommutative analogue of
multiplication by $\frac{1}{n!}\exp(t)$.  To explain this, we view the
algebra $\Om \o\bk[t]$ as a quotient of the algebra $\Om *_\bk\bk[t]$
by the two-sided ideal generated by $[t, \Om ]$. Thus, one has the
natural algebra homomorphism $\Om_t \to \Om [t]$ that `makes the
variable $t$ central.'  This homomorphism clearly descends to a linear
map $\oDR_t \to \oDR [t]$. The latter map can further be extended, by
continuity, to a map $\ab_t:\ \ohDRt \to \oDR[\![t]\!]$, between the
corresponding $t$-adic completions. Then, one has
\begin{lemma}\label{n!} For every $f\in \oOm^n$,
  in $\oDR[\![t]\!]$ we have: $\dis\ \ab_t\ccirc\EE(f) \ =\
  \mbox{$\frac{1}{n!}$}\cdot\exp(t)\, f\cc$.
\end{lemma}
\proof Given positive integers $k$ and $n$, let $P_n(k)$ denote the
number of partitions $k_0 + \cdots + k_n=k$, of $k$ into $n+1$ parts.
That number is given by the formula $P_n(k)=\frac{(k+n)!}{k!\,n!}.\ $
Using this and the definition of the map $\EE$, for every $ f=a_0
\sd\! a_1 \cdots \sd\! a_n\in \oOm^n$, we compute
\begin{multline*}
  \ab_t\ccirc\EE(f)\ =\ \sum_{k_0, \ldots, k_n \geq 0} \frac{1}{(k_0 +
    \cdots + k_n+n)!}\cdot
  \ab_t((a_0 t^{k_0} \sd\! a_1 t^{k_1} \cdots \sd\! a_n t^{k_n})\cc)\\
  \sum_{k_0, \ldots, k_n \geq 0} \frac{1}{(k_0 + \cdots +
    k_n+n)!}\cdot t^{k_0+\ldots+k_n}\,(a_0 \sd\! a_1 \cdots \sd\!
  a_n)\cc \ =\ \sum_{k\geq0} P_n(k) \frac{1}{(k+n)!}\cdot t^k\,
  f\cc\\
  =\ \sum_{k\geq0}\frac{(k+n)!}{k!\,n!} \frac{1}{(k+n)!}\cdot t^k\,
  f\cc \ =\ \sum_{k\geq0}\frac{1}{k!\,n!}\cdot t^k\, f\cc\ =\
  \frac{1}{n!}\cdot\exp(t)\, f\cc.\quad \tag*{$\Box$}
\end{multline*}

\begin{theorem}\label{t:Tthm}
The map $\EE$ has the properties:
\begin{equation}\label{p:ddtT}
  \mathsf{(i)}\en\  \mbox{$\frac{d}{dt}$}\ccirc\EE = \EE;\qquad \ 
  \mathsf{(ii)}\en\ 
  \EE \ccirc \msB = \sd \ccirc \EE;\qquad \ \mathsf{(iii)}\en\ 
  \EE \ccirc \msb = \si_t \ccirc \EE.
\end{equation}
\end{theorem}

This theorem will be proved in \S\ref{ts_pf}.
\begin{remark} We have arrived at the definition of the map $\EE$ by
  analyzing a construction of Tsygan \cite{T11} (that construction
  was, in its turn, motivated by our results \cite{GScyc}).  In
  particular, equations \eqref{p:ddtT}(ii)--(iii) are an adaptation of
  the main result of \cite{T11}.  It is likely that a proof of these
  equations, which is more conceptual than the one given in
  \S\ref{ts_pf} below, can be extracted from \cite{T11}, although the
  arguments in \emph{op.~cit.}~are written in the setting of the
  standard (i.e., commutative) equivariant cochain complex rather than
  the extended noncommutative de Rham complex.  However, such a
  noncommutative adaptation of the argument of \cite{T11} would be
  less direct since it would have to go through an auxiliary, much
  larger, complex $\oOm (\oOm A)$ (this would allow one to write an
  analogue of $\EE$ as an exponential of a certain contraction map on
  this larger complex and to apply the noncommutative calculus
  developed in \cite{TT}).
\end{remark}

\begin{remark}
  One can use \eqref{p:ddtT}(i) to find the coefficients of the map
  $\EE$ in the summation \eqref{e:Tdefn}.  Although one might expect
  that these coefficients should be $\frac{1}{\deg_t !}$, inspecting
  the LHS of \eqref{p:ddtT}(i) more carefully one sees that it must be
  $\frac{1}{(\deg_t +\deg_{\DR})!}$ as indicated.
\end{remark}

\begin{remark}\label{r:gs-per-chain} Restricting equations   
  \eqref{p:ddtT}(ii)--(iii) to the harmonic part of $\oOm$ and using
  the chain map $\sN!$, defined as $\sN!|_{\oOm^n} = n! \Id$,
  one obtains $\EE|_{P\oOm} \sN!  \ccirc \sd = \sd \ccirc \EE|_{P\oOm}
  \sN!$ and $\EE|_{P\oOm} \sN! \ccirc \bi = \si_t \ccirc \EE|_{P\oOm}
  \sN!$.

  Curiously, we do not know if there is any explicit formula, without
  using harmonic decomposition, for a map $\EE': \oOm \to \ohDRt$
  which satisfies $\EE' \ccirc \sd = \sd \ccirc \EE'$ and $\EE' \ccirc
  \bi = \si_t \ccirc \EE'$. (To define such a map using harmonic
  decomposition, one could set $\EE' = \EE \sN! P$.)

  Perhaps this is related to the fact that (the antiharmonic part of)
  the bicomplex $(\oOm, \sd, \bi)$ is somewhat badly behaved, i.e.,
  the homologies of $(\oOm \otimes R, \bi - u \sd)$ and
  $(\oOm[\![u]\!], \bi - u\sd)$ pick up the huge extra factor
  $\overline{[\sd \Om, \sd \Om]}$ (see Theorem \ref{t:perneg}).  On
  the other hand, we will see in \S \ref{ext_sec} and \S
  \ref{negative} below that the homologies of $(\oDR_t \otimes R,
  \si_t - u \sd)$ and $(\oDR_t[\![u]\!], \si_t - u\sd)$ are nicely
  expressible in terms of ordinary cyclic and Karoubi-de Rham
  homology.  This suggests that the bicomplex $(\oDR_t, \si_t, \sd)$
  should indeed be more closely related to $(\oOm, \msb, \msB)$ than
  to $(\oOm, \bi, \sd)$.
\end{remark}

Next, we recall that the differentials $\msB$ and $\msb$, as well as
$\sd$ and $\si_t$, anti-commute.  Therefore, thanks to Theorem
\ref{t:Tthm}, the map $\EE$ gives a morphism $(\oOm,\ \msB,\msb)\ \to
(\ohDRt,\ \sd, \si_t)$, of bicomplexes. This induces maps
$H^\hdot(\oOm, \msb) \to H^\hdot(\ohDRt, \si_t)$ and $H^\hdot(\oOm /
\msB \oOm, \msb) \to H^\hdot(\ohDRt / \sd \ohDRt, \si_t)$ on homology.

Observe that, since the differential $\si_t$ is homogeneous in $t$,
each individual homology group of the complexes $(\oDR_t, \si_t)$ and
$(\oDR_t / \sd \oDR_t, \si_t)$ has an additional grading
\[H^p(\oDR_t,\ \si_t)=
\oplus_q\ H(\oDR_t,\ \si_t)^{p,q},\en\text{and}\en
H^p(\oDR_t / \sd \oDR_t,\ \si_t)=
\oplus_q\ H(\oDR_t / \sd \oDR_t,\ \si_t)^{p,q},\]
where $p$ refers to the de Rham degree
and $q$  refers to   degree in $t$.
Therefore, the map $\EE$ produces, by
separating degrees in $t$, 
an infinite collection $\{\EE^q,\ q=0,1,\ldots\}$ of maps 
\begin{equation}\label{hee}
H^\hdot(\oOm, \msb) \stackrel{\EE^q}\too H^\hdot(\oDR_t,
  \si_t)^{\hdot,q},\en\text{and}\en
H^\hdot(\oOm / \msB \oOm,\ \msb)\
\stackrel{\EE^q}\too H(\oDR_t / \sd \oDR_t,\  \si_t)^{\hdot,q}.
\end{equation}
For  $q=0$, the above maps   reduce to the maps
$f\mto f\cc$ and
 $f \mmod(\im\msB)\ \mto\ f\cc\mmod(\im\sd)$.

 Observe further that the map $\ddt$ commutes with $\sd$ and $\si_t$,
 and hence it induces endomorphisms of the complex $(\oDR_t, \si_t)$
 and $(\oDR_t / \sd \oDR_t, \si_t)$.  On homology, this gives maps
\[
H(\oDR_t,
  \si_t)^{p,\hdot}
\stackrel{\frac{d}{dt}}\too H(\oDR_t,
  \si_t)^{p,\hdot-1},\en\text{and}\en
H(\oDR_t / \sd \oDR_t,\ \si_t)^{p,\hdot}\stackrel{\frac{d}{dt}}\too
 H(\oDR_t / \sd \oDR_t,\ \si_t)^{p,\hdot-1}.
 \]
It is clear from \eqref{p:ddtT}(i) that the maps $\EE^q,\ q\geq 0,$
satisfy a chain of relations
\begin{equation}\label{ddtH}
\mbox{$\frac{d^k}{dt^k}$}\ccirc
 \EE^q = \EE^{q-k}, \quad\text{for every}\en  k \leq q.
\end{equation}

Our second important result about the Tsygan map, that will be proved
in section \ref{heq_pf}, reads

\begin{theorem} \label{t:heq} Each of the maps in \eqref{hee} is an
  isomorphism, for every $q \geq 0$.
\end{theorem}

\subsection{Proof of Theorem  \ref{t:Tthm}}\label{ts_pf}
Equation  \eqref{p:ddtT}(i) follows from the expansion
\[
\mbox{$\frac{d}{dt}$}(t^{k_0} a_0 t^{k_1} \sd\! a_1 ... t^{k_n} \sd\!
a_n)\ =\ \sum\nolimits_j\ k_j t^{k_0} a_0 ... t^{k_{j-1}} \sd\!
a_{j-1} t^{k_j - 1} \sd\!  a_{j} t^{k_{j+1}} \cdots t^{k_n} \sd\! a_n,
\]
generalizing the proof that $\frac{d}{dt} \exp(t) = \exp(t)$, in terms
of the expansion $\exp(t) = \sum_{m \geq 0} \frac{1}{m!} t^m$.

To prove \eqref{p:ddtT}(ii), we compute

\begin{multline*}
  \!\!\EE\ccirc\! \msB(a_0 \sd\! a_1 \cdots \sd\! a_n) = \sum_{j=0}^n
  \sum_{k_0,\ldots,k_{n+1}}\!  \frac{1}{(k_0 + \cdots +
    k_{n+1}+(n+1))!}  (t^{k_0} \sd\! a_j t^{k_1} \sd\! a_{j+1} \cdots
  t^{k_n} \sd\! a_{j-1}
  t^{k_{n+1}})\cc \\
  = \sum_{k_0, \ldots, k_n} \frac{1}{(k_0+ \cdots + k_n+n)!} (\sd\!a_0
  t^{k_0} \cdots \sd\!a_n t^{k_n})\cc = \sd \ccirc \EE (a_0 \sd\! a_1
  \cdots \sd\! a_n).
\end{multline*}
Here and  below, the indices
$k_0,\ldots,k_{n+1}$ always run over all nonnegative integers.

To prove \eqref{p:ddtT}(iii) we will use the
identities, for every $n\geq1$:
\begin{multline*}
\sum_{p,q\geq0}\ \frac{t^p\,[t,a]\,t^q}{(p+q+n)!}=
\sum_{p,q\geq0}\ \frac{t^{p+1}a t^q-t^pat^{q+1}}{(p+q+n)!}\ =
\qquad\text{(taking $k:=p+1$ and $\ell:=q+1$)}\\
=
\left(\sum_{k,q\geq0}\frac{t^ka t^q}{(k-1+q+n)!} -
\sum_{q\geq0}\frac{at^q}{(q+n-1)!}\right)
-\left(\sum_{p,\ell\geq0}\frac{t^pat^\ell}{(p+\ell-1+n)!}
-\sum_{p\geq0}\frac{t^pa}{(p+n-1)!}\right)\\
=\sum_{p\geq0}\ \frac{t^pa}{(p+n-1)!}-\sum_{q\geq0}\ \frac{at^q}{(q+n-1)!}
=\sum_{j\geq0}\ \frac{t^ja-at^j}{(j+n-1)!}.
\end{multline*}

Using the above formula, for every $a_0, \ldots, a_n\in A$,
in $\oOm_t$, we find
\begin{align}
  &\en\sum_{j=1}^n \sum_{k_0, \ldots, k_n
  }\frac{(-1)^{j}}{(k_0+\cdots+k_{n}+n)!}  a_0 t^{k_0} \sd\! a_1
  t^{k_1} \cdots \sd\!a_{j-1} t^{k_{j-1}}\, [t,a_j]\, t^{k_j}
  \sd\! a_{j+1} t^{k_{j+1}} \cdots \sd\! a_n t^{k_n}& \nonumber\\
  &= \sum_{j=1}^n \sum_{k_0, \ldots, \hat k_j, \ldots, k_n
  }\frac{(-1)^{j}}{(k_0+\cdots \hat k_j \cdots
    +k_{n}+(n-1))!}\bigl(a_0 t^{k_0} \sd\! a_1 t^{k_1} \cdots \sd\!
  a_{j-1} t^{k_{j-1}} a_j
  \sd\! a_{j+1} t^{k_{j+1}} \cdots \sd\! a_n t^{k_n} &\nonumber\\
  &\hskip50mm- a_0 t^{k_0} \sd\! a_1 t^{k_1} \cdots \sd\! a_{j-1}\,
  a_j t^{k_{j-1}}
  \sd\! a_{j+1} t^{k_{j+1}} \cdots \sd\! a_n t^{k_n}\bigr) &\label{long}\\
  &= \sum_{j=1}^{n-1}\left(\sum_{k_0, \ldots, \hat k_j, \ldots, k_n }
    \frac{(-1)^{j}}{(k_0+\cdots \hat k_j \cdots +k_{n}+(n-1))!}  a_0
    t^{k_0} \sd\! a_1 \cdots \sd\! a_{j-1} t^{k_{j-1}} \sd(a_j
    a_{j+1})
    t^{k_{j+1}}  \cdots \sd\! a_n t^{k_n}\right) &\nonumber\\
  &\quad+ \sum_{k_1, \ldots, k_n } \frac{1}{(k_1+\cdots
    +k_{n}+(n-1))!}  (a_0 a_1 t^{k_1} \sd\! a_2 t^{k_2} \cdots \sd\!
  a_n t^{k_n} + (-1)^n a_0 t^{k_1} \sd\! a_1 t^{k_2} \cdots \sd\!
  a_{n-1} t^{k_n} a_n).&\nonumber
\end{align}

Now, by cyclic symmetry, in $\oDR_t$, 
$$(a_0 t^{k_1} \sd\! a_1 t^{k_2} \cdots \sd\! a_{n-1} t^{k_n} a_n)\cc
=(a_na_0 t^{k_1} \sd\! a_1 t^{k_2} \cdots \sd\! a_{n-1} t^{k_n})\cc.$$

Therefore, the image in $\oDR_t$ of 
the expression in the last two lines of \eqref{long} may be written
as
\begin{multline*}
\qquad\EE\bigl(\sum_{j=1}^{n-1} \
(-1)^{j}
a_0  \sd\! a_1 \cdots \sd\! a_{j-1}  \sd(a_j a_{j+1})
\sd\! a_{j+2}  \cdots \sd\! a_n\bigr)\\
+\EE(a_0 a_1  \sd\! a_2  \cdots \sd\! a_n)
+(-1)^n \EE(a_na_0 \sd\! a_1  \sd\! a_2  \cdots \sd\! a_{n-1}).\qquad
\end{multline*}
Further, the standard formula for the Hochschild differential written
in terms of differential forms \cite[formula (10)]{CQ1} reads:
\begin{multline*}
  \msb(a_0 \sd\! a_1 \cdots \sd\! a_n)= \sum_{j=1}^{n-1} (-1)^{j} a_0
  \sd\! a_1 \cdots \sd\! a_{j-1} \sd(a_j a_{j+1})
  \sd\! a_{j+2}  \cdots \sd\! a_n\\
  +a_0 a_1 \sd\! a_2 \cdots \sd\! a_n +(-1)^n a_na_0 \sd\! a_1 \sd\!
  a_2 \cdots \sd\! a_{n-1}.
\end{multline*}

Thus, combining everything together and using that the expression in
the top line of \eqref{long} equals $\si_t \ccirc \EE (a_0 \sd\! a_1
\cdots \sd\! a_n)$, we deduce that $\si_t \ccirc \EE (a_0 \sd\! a_1
\cdots \sd\! a_n)= \EE\ccirc\msb(a_0 \sd\! a_1 \cdots \sd\! a_n)$.
\qed

\subsection{Extended  cyclic homology}\label{ext_sec}
It will be convenient to introduce an additional parameter $u$, and
consider the $\bk[u]$-module $\oDR_t \otimes R$. We equip this module
with two different gradings.  The first grading, referred to as {\em
  homological grading}, is defined by assigning the component
$\oDR_t^{p,q} \cdot u^{-r}$ homological degree $\ell=p+2r$. Thus, the
$u$ variable has homological degree $-2$. The $t$ parameter has
homological degree zero, so it does not contribute to the homological
grading.  The second grading, referred to as {\em internal grading},
is defined by assigning the component $\oDR_t^{p,q}\cdot u^{-r}$
internal degree $p+q+r$.

On $\oDR_t \otimes R$, we have a pair of anti-commuting differentials,
$u \sd$ and $\si_t$.  Each of these differentials has homological
degree $-1$ and internal degree zero.  We define $\oEHC(A)$, the {\em
  reduced extended cyclic homology} of $A$, as
\begin{equation}\label{ehc}
\oEHC_\ell(A)=H^\ell(\oDR_t \otimes R,\  \si_t-u\sd),
\end{equation}
where $\ell$ on the RHS stands for  homological degree.

Since the internal grading is preserved by both
differentials, it provides a direct sum  decomposition of complexes,
$$ \oDR_t \otimes R\ =\ \bplus_k\ (\oDR_t \otimes R)_k,\qquad
(\oDR_t \otimes R)_k\ :=\ \bigoplus_{p+q+r=k} \oDR^{p,q}\cdot u^{-r},
$$
where, $(\oDR_t \otimes R)_k$, the homogeneous component of internal
degree $k$, is a subcomplex of $\oDR_t \otimes R$ with the inherited
differential $\si_t-u\sd$ and inherited homological grading.  We write
$\pr_k:\ \oDR_t \o R \to (\oDR_t \o R)_k$ for the projection to that
homogeneous component.

For every integer $m$, we let
\[\oEHC_\ell(A,m):=H^\ell\bigl((\oDR_t \otimes R,\  
\si_t-u\sd)_{\ell-m}\bigr).\] Thus, one obtains a `{\em weight
  decomposition}': $\dis \ \oEHC_\ell(A)=\oplus_m\ \oEHC_\ell(A,m)$,
for $m\in\Z$.

Our main result about extended cyclic homology is the following
theorem. Part (i) of the theorem shows that the knowledge of the
extended cyclic homology is equivalent to the knowledge of both the
usual cyclic homology and the Karoubi-de Rham homology.

\begin{theorem}\label{e:ehc-cyc2} \vi One has a canonical 
$\bk[u]$-module isomorphism
\[
\oEHC(A)\ \cong\ (\oHD(A)\o u^{-1}\bk[u^{-1}]) \en \bplus\en
\oHC(A)[t],
\]
where the variable `$u$' annihilates the second summand, and on
$\oHD(A) \o u^{-1}$ restricts to the natural inclusion $\oHD(A) \into
\oHC(A)$ into the second summand in degree zero in $t$.  Furthermore,
the group $\oEHC_\ell(A,m)$ vanishes for all $m>\il$, and
\begin{equation}\label{qq}
\oEHC_\ell(A,m)\ \cong\ \begin{cases}
\oHD_{\ell-2m}(A)&\qquad\qquad 0<m\leq\il,\\
\oHC_\ell(A)&\qquad\qquad m\leq 0.
\end{cases}
\end{equation}

\vii For every $q\geq0$, there are canonical graded space isomorphisms
\begin{align}
  H(\oDR_t,\
  \si_t)^{\hdot,q}\ &\cong\  \oHH_\idot(A),\label{qhh}\\
  H(\oDR_t / \sd \oDR_t,\ \si_t)^{\hdot,q}\ &\cong\
  \oHC_\idot(A).\label{qhc}
\end{align}
\end{theorem}

\begin{remark}\label{0rem}
In the special case $q=0$, using the identification in \eqref{i=i},
the above becomes
\begin{align}
  H(\oDR_t,\ \si_t)^{\hdot,0}\ &=\ \ker(\si: \oDR^\hdot \to
  \oOm ^{\hdot-1})\label{brem}\\
  H(\oDR_t/ \sd \oDR_t,\ \si_t)^{\hdot,0}\ &=\ \ker\bigl(\si:
  \oDR^\hdot\!/ \sd\! \oDR^{\hdot-1}\to \oOm ^{\hdot-1}\!/ \sd\!
  \oOm^{\hdot-2} \bigr)\label{orem}.
\end{align}
Therefore,  for $q=0$,  isomorphisms \eqref{qhh}--\eqref{qhc}  may be
equivalently rewritten as 
\begin{align*}
  &\ker(\si: \oDR^\hdot \to \oOm ^{\hdot-1})\ \cong\
  \oHH_\idot(A),\\
  \ker\bigl(\si:& \oDR^\hdot\!/ \sd\! \oDR^{\hdot-1} \to \oOm
  ^{\hdot-1}\!/ \sd\! \oOm^{\hdot-2} \bigr)\ \cong\ \oHC_\idot(A).
\end{align*}
These are nothing but our isomorphisms \eqref{hhdef}--\eqref{hcdef}.
Thus, we see that Theorem \ref{e:ehc-cyc2}(ii) incorporates Theorem
\ref{HHHC}, as the special case $q=0$.
\end{remark}

\subsection{Theorem \ref{t:heq} implies Theorem
  \ref{e:ehc-cyc2}}\label{tt}
We consider part (ii) of Theorem \ref{e:ehc-cyc2} first.  The proof of
this part involves the Tsygan map $\EE$ in a crucial way.

In more detail, the Hochschild homology of $A$ is computed by the
complex $(\oOm, \msb)$. Thus, the map $H(\oOm, \msb)\to
H^\hdot(\oDR_t, \si_t)^{\hdot,q}$ in \eqref{hee}, which is an
isomorphism thanks to Theorem \ref{t:heq}, provides isomorphism
\eqref{qhh}.  Similarly, in the cyclic homology case, we use Lemma
\ref{cq}. Thus, we define isomorphism \eqref{qhc} to be the composite
of the isomorphism $\oHC_\idot(A)\iso H^\hdot(\oOm / \msB \oOm,
\msb)$, of Lemma \ref{cq}, with the map $\EE^q$, where the latter is
an isomorphism by Theorem \ref{t:heq}.

For part (i), we first prove

\begin{proposition}\label{p:ext-cyc} \vi The 
  assignments
\[\pi_m:\ \sum\nolimits_{k\geq0}\ f_k\cdot u^{-k}\ \mto\ 
f_m\mmod (t), \qquad \pi_-:\ \sum\nolimits_{k\geq0}\ f_k\cdot u^{-k}\
\mto\ f_0\mmod (\im\sd),
\]
induce isomorphisms
\begin{align}
\pi_m:\ \oEHC_\ell(A,m)\ &\iso\ \oHD_{\ell-2m}(A),
 & m>0;\label{pim}\\
\pi_-:\ \oEHC_\ell(A,m)\ &\iso\ 
H(\oDR_t/ \sd \oDR_t,\ \si_t)^{\ell,-m},&m\leq0.\label{pi-}
\end{align}

\vii For every $m\leq0$, the following diagram commutes:
\begin{equation}\label{pidiag}
  \xymatrix{
    H^\ell(\oOm\o R,\ \msb-u\msB)\ar[d]_<>(0.5){
      \text{Lemma \ref{cq}}}^<>(0.5){\cong}
    \ar[rr]^<>(0.5){H_\ell(\pr_{\ell-m}\ \circ\ \EE)}&& \ \oEHC_\ell(A,m)\
    \ar[d]^<>(0.5){\pi_-}\\
    H^\ell(\oOm/\msB\oOm, \msb)\ar[rr]^<>(0.5){\EE^{-m}}&&
    H(\oDR_t / \sd \oDR_t,\ \si_t)^{\ell,-m}.
  }
\end{equation}
\end{proposition}

\begin{proof}
  We view $\oDR_t \otimes R$ as a bicomplex $M$ of the form $M_{p,q} =
  \oDR_t^{p-q} \cdot u^{-q}$.  The differentials are $\si_t:
  M_{p,q}\to M_{p-1,q}$ and $u \sd: M_{p,q}\to M_{p,q-1}$.  We apply
  the spectral sequence associated with that bicomplex, by taking
  cohomology with respect to $u\sd$ first.

  To compute the first page of the spectral sequence, we use the
  natural direct sum decomposition $\oDR_t\o R=\oDR_t \ooplus
  u\inv\oDR_t[u\inv]$.  For the cohomology of the first differential,
  this gives a direct sum decomposition $H(\oDR_t\o R,\ u\sd)=
  \oDR_t/\sd\oDR_t\ \oplus\ u\inv H(\oDR_t,\sd)[u\inv]$.  By Corollary
  \ref{qis1}, the projection modulo $t$ induces an isomorphism
  $H(\oDR_t,\sd)\iso H(\oDR,\sd)$. Therefore, the first page of the
  spectral sequence takes the form $E_1=\oDR_t/\sd\oDR_t\ \oplus\
  u\inv H(\oDR,\sd)[u\inv]$.  The differential $\si_t: E_1\to E_1$, on
  the first page, clearly increases degree in $t$ by 1.  Hence, this
  differential annihilates the direct summand $u\inv
  H(\oDR,\sd)[u\inv]$.  We conclude that the second page of the
  spectral sequence reads $H(\oDR_t/\sd\oDR_t, \si_t)\ \oplus\ u\inv
  H(\oDR,\sd)[u\inv]$, which is the right-hand side of the isomorphism
  of the proposition.  It is immediate to see that the spectral
  sequence collapses at the second page.  This proves part (i). Part
  (ii) follows directly from the formula for the map $\pi_-$.
\end{proof}

To complete the proof of Theorem \ref{e:ehc-cyc2}(i), we note that the
group that appears in the upper left corner of diagram \eqref{pidiag}
equals $\oHC_\ell(A)$, by ~\eqref{e:gs-hh}. Thus, we may (and will)
define isomorphism \eqref{qq} by combining isomorphisms \eqref{qhc}
and those of Proposition \ref{p:ext-cyc} together.

We conclude that the proof of Theorem \ref{e:ehc-cyc2} would be
complete once we prove Theorem \ref{t:heq}.

\begin{remark} An analogue involving the maps $\pi_m$ of the statement
  that diagram \eqref{pidiag} commutes will be given in Proposition
  \ref{p:Timh}(i) in \S\ref{negative}.
\end{remark}

\subsection{}\label{inv_sec}
It is possible to write a relatively explicit formula for the inverse
of the isomorphisms \eqref{pim}--\eqref{pi-}.  To this end, recall
first that the complex $(\oDR_t^+,\sd)$ is acyclic, by Lemma
\ref{qis0}.  Therefore, the Karoubi-de Rham differential $\sd$ induces
an isomorphism $\oDR^+_t/\sd\oDR^+_t$ $\iso \sd\oDR^+_t$.

Let $\sdi: \sd \oDR_t^+ \to \oDR_t^+$ be a (set-theoretic) section of
the surjection $\sd: \oDR^+_t\onto \sd\oDR^+_t$ (we use the same
notation as for the similar map in \S\ref{ss:negper}).  The choice of
$\sdi$ will not affect anything below, and we need it only to write
formulas which are independent of this choice.

\begin{proposition}\label{inverse} 
The inverses of the isomorphisms \eqref{pim} and
\eqref{pi-}
 of Proposition \ref{p:ext-cyc} are  induced respectively
by the maps
$$
\ker\bigl(\sd|_{\oDR}\bigr)\to H(\oDR_t\o R,\ \si_t-u\sd)
\quad\text{\em and}\quad \ker\bigl(\si_t|_{\oDR_t/\sd\oDR_t}\bigr) \to
H(\oDR_t\o R,\ \si_t-u\sd)
$$
given by the assignments
\[
f\ \longmapsto\ u^m(\Id - u^{-1}\sdi \si_t)^{-1}(f) \quad\text{\em
  and}\quad f\ \longmapsto\ (\Id - u^{-1}\sdi \si_t)^{-1}(f).
\]
\end{proposition}

\begin{remark} In view of \eqref{e:connes-seq-maps}, the map $\sdi
  \si_t$ on $\ker\bigl(\si_t|_{\oDR_t/\sd\oDR_t}\bigr)$ may be thought
  of as an extended analogue of the periodicity operator $S$.  Via
  Theorem \ref{t:heq}, it actually becomes $S$ after making the above
  identification.
 \end{remark}

\begin{proof} 
  The statement of the proposition is essentially a simple consequence
  of the construction of differentials in the spectral sequence of a
  double complex.  Below, we explain the case of the map $\pi_-$. The
  case of the maps $\pi_m,\ m>0,$ is similar (and even easier).

  Let $\ker(\sd\si_t)$ be the kernel of the map $\sd\ccirc\si_t:\
  \oDR_t\to\oDR_t$.  We claim first that $\sdi \si_t$ is a
  well-defined set-theoretical endomorphism of $\ker(\sd\si_t)$. To
  see this, observe that $\si_t(\oDR_t)\sset \oDR_t^+$.  Hence, thanks
  to acyclicity of the complex $(\oDR_t^+, \sd)$, see Lemma
  \ref{qis0}, we deduce that an element $x\in\oDR_t$ belongs to
  $\ker(\sd\si_t)$ if and only if $\si_t x =\sd y$, for some $y\in
  \oDR_t^+$. Thus, for such an $x$, the element $\sdi \si_t x=\sdi\sd
  y$ is indeed well-defined.  Furthermore, we compute $\sd\si_t(\sdi
  \si_t x)=-\si_t\sd (\sdi \si_t x)= -(\si_t)^2(x)=0$. Thus, we have
  proved that $\sdi \si_t x\in \ker(\sd\si_t)$.

  The operator $\sdi\si_t$ on $\ker(\sd \si_t)$ has de Rham degree
  $-2$, and hence $\Id-u^{-1}\sdi\si_t$ is invertible.  Explicitly, an
  inverse is given, for every $x\in \ker(\sd\si_t)$, by the formula
  $(\Id - u^{-1}\sdi \si_t)^{-1}(x)=\sum_{k\geq0} \ u^{-k}(\sdi
  \si_t)^k(x)$, where the terms $(\sdi \si_t)^k(x)$ vanish for all $k$
  greater than half the de Rham degree of $x$.  Further, using that
  $\sd\sdi =\Id$, we compute
\begin{multline*}
  \qquad (\si_t-u\sd)(\Id - u^{-1}\sdi \si_t)^{-1}
  =(\sd\sdi\si_t-u\sd)(\Id - u^{-1}\sdi \si_t)^{-1}\\=
  u\sd(u^{-1}\sdi\si_t-\Id)(\Id - u^{-1}\sdi \si_t)^{-1}= -u\sd.
  \hskip15mm
\end{multline*}

We see that for every $z$ in the image of the map $(\Id - u^{-1}\sdi
\si_t)^{-1}$, the element $(\si_t-u\sd)(z)$ belongs to $u\oDR_t$.  It
follows that $z\mmod u\oDR_t[\![u]\!]$ is a cycle in the complex
$(\oDR_t\o R,\ \si_t-u\sd)$.  Thus, we have constructed the map
\[
(\Id - u^{-1}\sdi \si_t)^{-1}:\ \ker(\sd\si_t) \too H(\oDR_t\o R,\
\si_t-u\sd).
\]

Now, it is immediate from the above construction that the above
constructed map is a right inverse to the map $\pi_-$ in the sense
that, for every $x\in \ker(\sd\si_t)$, the class of the element $x$ in
$H(\oDR_t/\sd\oDR_t,\ \si_t)$ equals $\pi_-(\Id - u^{-1}\sdi
\si_t)^{-1}(x)$.  By definition, one has
$$H(\oDR_t/\sd\oDR_t,\ \si_t)\
=\ \frac{\ker(\sd\si_t)}{\ker \sd+\im\si_t}.
$$
Therefore, for every $x\in \ker\sd+\im\si_t$, we must have $\pi_-(\Id
- u^{-1}\sdi \si_t)^{-1}(x)=0$.  Since $\pi_-$ is an isomorphism, by
Proposition \ref{p:ext-cyc}, we deduce that the map $(\Id - u^{-1}\sdi
\si_t)^{-1}$ takes $\ker\sd+\im\si_t$ to zero on homology. We conclude
that the map given in Proposition \ref{inverse} is well
defined. Moreover, it is a right inverse, and hence a two-sided
inverse, to the isomorphism $\pi_-$.
\end{proof}

\section{Proof of Theorem \ref{t:heq}}\label{heq_pf}
\subsection{} 
First of all, thanks to equation \eqref{ddtH}, $\frac{d^q}{dt^q}
\ccirc \EE^q = \EE^0$.  The maps $\EE^0$ in \eqref{hee} are
isomorphisms by Theorem \ref{HHHC} together with Lemma \ref{cq}.  It
follows that, for every $q\geq0$, the map $\EE^q$ is injective.
Therefore, we only have to show that $\EE^q$ is surjective.

To this end, we introduce a filtration on $\oDR_t$ by the degree in
$A$: precisely, the degree $\leq m$ part of the filtration is spanned
by monomials in $t$, $A$, and $\sd\! A$ with the elements of the
latter two appearing at most $m$ times.  This descends to filtrations
under the surjections $\Om_t \onto \oOm_t \onto \oDR_t$.  The
differentials $\si_t$ and $\sd$ preserve the resulting filtration. So
we obtain a filtered bicomplex $(\oDR_t, \sd, \si_t)$.

It is immediate to check that $\gr (\oDR_t, \sd, \si_t)$, the
associated graded bicomplex, is naturally isomorphic to $(\oDR_t A',
\sd, \si_t)$, where $A' := \oA \oplus \bk$ is the algebra with the
trivial multiplication $\oA \cdot \oA = 0$. Let $\EE'$ denote the
Tsygan map for the algebra $A'$. Clearly, one has that $\EE'=\gr\EE$.

Thus, we have reduced the proof of the theorem to

\begin{proposition}\label{p:heq}
  The maps $\ (\EE')^q:\ H(\oOm A', \msb) \to H(\oDR_t A',
  \si_t)^{\hdot,q}\ $ and $\ (\EE')^q:\ H(\oOm A'/\msB\oOm A', \msb)
  \to H(\oDR_t A' / \sd \oDR_t A', \si_t)^{\hdot,q}\ $ are both
  surjective for every $q\geq0$.
\end{proposition}

The rest of \S\ref{heq_pf} is devoted to the proof of this
proposition. The idea of the proof is to reduce the result further to
the special case where $\oA=\bk$ (and to lift from cyclic tensors to
ordinary tensors).  In this case, we compute the cohomology of the
complexes $(\oDR_t A', \si_t)$ and $(\oDR_t A' / \sd \oDR_t A',
\si_t)$ explicitly.  More precisely, we will define, for each
$\ell\geq 1$, a complex $C_\ell$ equipped with an action of $\Z/\ell$,
the cyclic group.  We identify the complex $(\oDR_t A', \si_t)$ with
$\bigoplus_{\ell \geq 1} (C_\ell)_{\Z/\ell}$, a direct sum of
complexes of coinvariants of cyclic groups. We then compute explicitly
the homology of each complex $C_\ell$.  The computation is, however,
quite technical; it will not be used elsewhere in the paper.

\subsection{Linear algebra lemmas}
\label{s:linalg} We begin with a few elementary lemmas from linear
algebra.

Let $V$ be a vector space over ${\mathbb R}$ equipped with a positive
definite inner product $\langle -,-\rangle:\ V\times V \to {\mathbb
  R}$.  Recall that a linear operator $T: V\to V$ is called
\emph{positive} if it is self-adjoint and $\langle Tv, v \rangle \geq
0$ for all $v\in V$. For a positive operator $T$ and $v\in V$, the
equation $\langle Tv, v \rangle=0$ implies that $Tv=0$.

\begin{lemma} \label{l:tp} Suppose that $T_0, \ldots, T_k$ are
  positive operators on vector spaces $V_0, \ldots, V_k$. Let $T :=
  \sum_{i=0}^k \Id^{\otimes i} \otimes T_i \otimes \Id^{\otimes k-i}$.
  Then $T$ is positive, and $\ker(T) = \bigotimes_{i=0}^k \ker(T_i)$.
\end{lemma}

\begin{proof}
  Clearly each $\Id^{\otimes i} \otimes T_i \otimes \Id^{\otimes k-i}$
  is positive, and so the sum is positive. Since the kernel of $T$ is
  the space of vectors $v$ such that $\langle Tv, v \rangle = 0$, it
  is the intersection of the kernels of $\Id^{\otimes i} \otimes T_i
  \otimes \Id^{\otimes k-i}$, i.e., $\bigotimes_{i=0}^k \ker(T_i)$.
\end{proof}

\begin{lemma} \label{l:pos-crit} Let $B = (b_{ij}) \in \Mat_n(\bR)$ be
  a symmetric $n\times n$-matrix such that $b_{ii} \geq \sum_{j \neq
    i} |b_{ij}|$ for all $i$.  Then $B$ is positive. Assume, in
  addition, that the graph obtained from $B$ with vertex set
  $\{1,\ldots,n\}$ and edges between $i$ and $j$ if $b_{ij} \neq 0$ is
  connected.  Then $\dim \ker(B) \leq 1$. This kernel is nonzero if
  and only if:
\begin{itemize}
\item The equality $b_{ii} = \sum_{j \neq i} |b_{ij}|$ holds for all $i$;
\item There exist signs $\lambda_i \in \{1,-1\}$
 such that $\sign(b_{ij}) = -\lambda_i \lambda_j$ for all $i \neq j$.
\end{itemize}
Then the kernel is spanned by the vector $(\lambda_1, \ldots,
\lambda_n)^\top$.
\end{lemma}
Here and below, we use the notation $(\lambda_1,\ldots,\lambda_n)^\top
= \begin{pmatrix} \lambda_1 \\ \vdots \\ \lambda_n \end{pmatrix}$,
where $\top$ stands for ``transpose.''

\begin{proof}
Let $v = (v_1, \ldots, v_n)^\top$. Then
\[
\langle Bv, v \rangle = \sum_{i < j} b_{ij}(v_i^2 + 2 v_i v_j + v_j^2)
+ \sum_i (b_{ii} - \sum_{j \neq i} b_{ij}) v_i^2 \geq 0.
\]
Hence $B$ is positive. Now, suppose $Bv = 0$ and $v \neq 0$.  Let $i$
be such that $|v_i| \geq |v_j|$ for all $j$.  Then $0 = b_{ii} v_i +
\sum_{j \neq i} b_{ij} v_j$ and $b_{ii} \geq \sum_{j \neq i} |b_{ij}|$
implies that the equality holds, and that $|v_j| = |v_i|$ whenever
$b_{ij} \neq 0$. Now assume that the graph of $B$ is connected. Then
we conclude that $|v_i|$ are all equal.  Hence the equality $b_{ii} =
\sum_{j \neq i} |b_{ij}|$ holds for all $i$.  Now, $Bv = 0$ if and
only if $\sign(b_{ij}) = - v_i / v_j$ whenever $i \neq j$.
\end{proof}
\begin{remark}
  Lemma \ref{l:pos-crit} can be significantly generalized: for
  arbitrary complex $B$ which is not necessarily symmetric, the
  condition $b_{ii} \geq \sum_{j \neq i} |b_{ij}|$ guarantees that the
  real parts of all eigenvalues of $B$ are nonnegative, and zero only
  for the zero eigenvalue. Let $Q(B)$ be a directed graph with vertex
  set $\{1,2,\ldots,n\}$ such that there is an edge from vertex $i$ to
  $j$ whenever $b_{ij} \neq 0$ and $i \neq j$.  If the graph $Q(B)$ is
  strongly connected (i.e., there exists a directed path from every
  vertex to every other vertex), then $\dim \ker(B) \leq 1$ and it is
  nonzero if and only if $b_{ii} = \sum_{j \neq i} |b_{ij}|$ for all
  $i$ and there exist complex numbers $\lambda_1, \ldots, \lambda_n$
  of absolute value $1$ such that $b_{ij} / |b_{ij}| =
  -\lambda_i/\lambda_j$ whenever $b_{ij} \neq 0$. In this case the
  kernel is spanned by $(\lambda_1, \ldots, \lambda_n)^\top$.
\end{remark}

\begin{lemma} \label{l:gc} Suppose that $(V^\idot,d)$ is a complex of
  vector spaces with $d: V^\idot \to V^{\idot+1}$. Suppose we are
  given an operator $h: V^\idot \to V^{\idot-1}$ such that
\[
V = \ker [h,d]\ \bplus\  \im [h,d], \quad \ker d = \ker([h,d]|_{\ker d})
\ \bplus\  \im([h,d]|_{\ker d}).
\]
(E.g., this holds if $[h,d]$ acts semisimply on $V$ or its base change
to an algebraically closed field.)  Then, the inclusion $(\ker[h,d],
d|_{\ker[h,d]}) \to (V, d)$ is a quasi-isomorphism.
\end{lemma}

\begin{proof}
  Since $(V,d) = (\ker [h,d],\, d|_{\ker [h,d]}) \ \bplus\ (\im
  [h,d],\, d|_{\im[h,d]})$ is a direct sum of complexes, we only need
  to show that the second factor is acyclic.  Since $\ker(d) =
  \ker(d)|_{\ker[h,d]} \oplus \ker(d)|_{\im[h,d]}$, and the first
  factor equals $\ker([h,d]|_{\ker d})$, the assumption implies that
  $\ker(d)|_{\im[h,d]} = \im([h,d]|_{\ker d}) = \im(dh|_{\ker d})$.
  So every cycle in $(\im [h,d], d|_{\im[h,d]})$ is a boundary.
\end{proof}
\begin{remark}
  More generally, in Lemma \ref{l:gc} we could replace $\ker[h,d]$ and
  $\im[h,d]$ by $\ker[h,d]^N$ and $\im[h,d]^N$ for every $N \geq
  1$. In particular, if $N \geq 1$ is such that $V = \bigoplus_{i \in
    \Z} V_i$ is a graded vector space with $\dim V_i \leq N$ for all
  $i$, then the conditions of the lemma will necessarily be satisfied
  with this modification.
\end{remark}

In particular, Lemma \ref{l:gc} applies when $V^\idot$ is a
finite-dimensional complex of rational vector spaces and $[h,d]$ is a
positive operator.

\subsection{An auxiliary complex} Our ultimate goal is to define a
complex $C_\ell$ and compute its homology, as well as that of a
certain quotient $C_\ell/d_s C_\ell$; here we begin, as an
intermediate step, with a somewhat simpler complex.  Consider the
supercommutative algebra $S := \bk[s]/(s^2)$, with $s$ odd, and the
left $S$-module
\[
M := \Span(1, s, t^i, s t^i, t^i s, s t^i s \mid i \geq 1)\
 \subset\ \bk\langle s,t \rangle / (s^2).
\]
Assign $M$ the grading in which $1$ is odd, $s$ is even, and for all
$i \geq 1$, $t^i$ and $st^i s$ are even and $st^i$ and $t^i s$ are
odd. (In other words, we modify the natural grading from the algebra
$\bk\langle s,t \rangle / (s^2) \supset M$ by shifting the degree of
the submodule $S \subset M$.) Moreover, equip $M$ with the right
$S$-module structure $\cdot_R$,
\[
(f) \cdot_R 1 = f, \quad (f) \cdot_R s = -fs,
\]
which differs from the natural multiplication from the algebra
$\bk\langle s,t \rangle / (s^2)$ by a sign.

From now until the end of \S\ref{ss:cllhom}, we will use simplified
notation $\o=\o_{_S}$.  For convenience, equip $T_S M$ with a
partially defined multiplication,
\[
(f_1 \otimes \cdots \otimes f_i) \cdot (g_1 \otimes \cdots \otimes
g_j) = f_1 \otimes \cdots \otimes f_i g_1 \otimes \cdots \otimes g_j,
\]
where here $f_i g_1$ is the multiplication in the ring $\bk\langle s,t
\rangle / (s^2)$ (not using $\cdot_R$). This is defined whenever $f_i
g_1 \in M$. Further, put $1_m := 1^{\otimes m} = 1\otimes
1\otimes\cdots \otimes 1$ ($m$ factors).

For each $\ell\geq 1$,
we define a triple, $\tilde d_s, \tilde \partial^0,
\tilde \partial^1$, of differentials on $M^{\otimes \ell}$, as
follows.  First, we set
\begin{align*}
  \tilde d_s &:\ 1_k\ \mto\ s \cdot 1_{k}\ \mto\ 0,\\
  \tilde \partial^0(1_{k}) &:= \sum\nolimits_{j=1}^{k-2}\ (-1)^{j-1}
  (1_{j} \otimes t \otimes
  1_{k-j-1}), \\
  \tilde \partial^0(s \cdot 1_{k}) &:= \sum\nolimits_{j=1}^{k-2} \
  (1_{j} \otimes (st-ts)
  \otimes 1_{k-j-1}), \\
  \tilde \partial^1(1_{k}) &:= 0, \quad \tilde \partial^1(s \cdot
  1_{k}) := t \otimes 1_{k-1} + (-1)^k 1_{k-1} \otimes t.
\end{align*}

Now, let $\tilde d$ stand for one of the three symbols $\tilde d_s,
\tilde \partial^0, \tilde \partial^1$.  Then, using the above
formulas, for every $m \geq 1$ and $h_0, \ldots, h_n \in \{1_k, s
\cdot 1_k \mid k \geq 1\}$, and $i_1, \ldots, i_n \geq 1$, we define
the corresponding differential by the formula
\[
\tilde d(h_0 t^{i_1} h_1 \cdots t^{i_n} h_n) = \sum_{j=0}^n (-1)^{|h_0
  t^{i_1} h_1 \cdots t^{i_{j-1}} h_{j-1}|}\cdot h_0 t^{i_1} h_1 \cdots
t^{i_{j-1}} h_{j-1} t^{i_j} \tilde d(h_j) t^{i_{j+1}} \cdots t^{i_n}
h_n.
\]

The meaning of the above differentials is explained by Lemma
\ref{l:babg0} below (which we will not actually need, but we will need
the generalization in Lemma \ref{zeta} below). This says that the
differentials $\tilde d_s$ and $\tilde \partial^0+\tilde \partial^1$
may be obtained by transporting the natural differentials $\sd$ and
$\bi_t$ by means of an appropriate isomorphism.  In more detail, let
$A_\tau=A'_\tau=\bk\oplus\bk\cdot\tau$ be a $\bk$-algebra with one
generator $\tau$ such that $\ta^2=0$.

\begin{lemma}\label{l:babg0}
The assignment 
\[\zeta:\en t\mto t,\quad s^\epsilon\cdot 1_\ell\mto 
\ta^\epsilon\,(\sd\ta)^{\ell-\epsilon}:=\,\ta^\epsilon\,
\underbrace{\sd\ta\,\sd\ta\ldots\sd\ta}_{\text{$\ell-\epsilon\ $
    times}},\quad \epsilon\in\{0,1\},\en\ell=1,2,\ldots
\]
yields an isomorphism of bicomplexes
\[\zeta:\en \left(\bigoplus\nolimits_{\ell\geq1}\ M^{\o\ell},\ d_s,\
  \tilde \partial^0+\tilde \partial^1\!\right) \ \iso\
\bigl(\oOm_t(A_\ta),\ \sd,\,\si_t\bigr).\]
\end{lemma}

\begin{proof}[Sketch of Proof]
  The fact that the assignment of the lemma respects the differentials
  can be verified directly.  Further, isomorphism \eqref{aaa} shows
  that the set of elements of the form
  $\ta^\epsilon\,(\sd\ta)^{\ell-\epsilon}$ is a $\bk$-basis of $\oOm
  (A_\ta)$.  Hence, this set combined with the set of words in the
  elements of the form $t^i\,\ta^\epsilon\,(\sd\ta)^{\ell-\epsilon},\
  i\geq 1,$ gives a $\bk$-basis of $\oOm_t(A_\ta)$.  It follows easily
  that the map $\zeta$ is a bijection.  We note that the direct
  summand $M^{\o\ell}$ goes, via the bijection, to the subspace
  $[\oOm_t(A_\ta)]_{(\ell)}$ of $\oOm_t(A_\ta)$ spanned by the words
  in which $\tau$ appears $\ell$ times.
\end{proof}

\subsection{} We will need to compute the homology of the complex
$(M^{\otimes \ell}, \tilde \partial^0)$.

\begin{lemma}\label{l:hcl-baby}
  A basis of $H(M^{\otimes \ell}, \tilde \partial^0)$ as a free
  $\bk$-module is given by elements of the form
  \begin{equation} \label{e:hcm} s^{\epsilon_0} \otimes
    (st+ts)^{\otimes j_0} \otimes t^{i_1} s^{\epsilon_1} \otimes
    (st+ts)^{\otimes j_1} \otimes \cdots \otimes t^{i_k}
    s^{\epsilon_k} \otimes (st+ts)^{\otimes j_k},
\end{equation}
for $\epsilon_0, \ldots, \epsilon_k \in \{0,1\}$, $j_0, \ldots, j_k
\geq 0$, and $i_1, \ldots, i_k \geq 2$.
\end{lemma}
\begin{proof}
  We define an operation $\tilde D$ on $M^{\otimes \ell}$ resembling a
  derivation which takes $st$ and $ts$ to $\frac{1}{2} s$, $t$ to $1$,
  and everything else to zero (such a map cannot actually be a
  derivation, because the Leibniz rule on $st$ is not satisfied).
  Precisely, we define
\[
\tilde D(t^{i}) = \delta_{i,1}, \quad \tilde D(st^i) =
\half\delta_{i,1} s = \tilde D(t^i s), \quad \tilde D(st^is) = 0,
\]
then write
\[
\tilde D(1_{p} \otimes f \otimes 1_{q}) = (-1)^p \cdot 1_{p} \otimes
\tilde D(f) \otimes 1_{q},
\]
and finally put
\[
\tilde D(f_0 t^{i_1} f_1 \cdots t^{i_k} f_k) = \tilde D(f_0 t^{i_1}
f_1) t^{i_2} f_2 \cdots t^{i_k} f_k + (-1)^{|f_0 t^{i_1}|} f_0 t^{i_1}
\tilde D(f_1 t^{i_1} f_2 \cdots t^{i_k} f_k),
\]
where $f_0, \ldots, f_k \in T_S \langle 1,s \rangle \subseteq T_S M$
(i.e., each of these elements can be taken to be either $1_{j}$ for
some $j$, or $s \cdot 1_{j}$ for some $j$).

Note that both $\tilde D$ and $\tilde \partial^0$ preserve the
subcomplexes which are homogeneous of fixed degrees $\geq 2$ in $t$ in
specified tensor components: precisely, such a subcomplex is one
where, in components $1 \leq c_1 < \cdots < c_p \leq \ell$, $t$
appears with degrees $d_1, \ldots, d_p \geq 2$, and $t$ appears with
degree $\leq 1$ in all other components.

We are going to show that the operator $\DDD:=[\tilde D,
\tilde \partial^0]$ is close to an isomorphism, and that its kernel
includes quasi-isomorphically into $(M^{\o \ell},\tilde \partial^0)$.
To this end, we will compute
$\DDD$ explicitly.  

By the observation from two paragraphs above, this computation reduces
to the subspace $\langle 1, t, st, ts \rangle^{\o m}$. More precisely,
$\DDD$ satisfies the following rule
which
resembles the derivation condition:
\begin{equation}\label{e:tdder}
  \DDD(g_1 \o t^i \o g_2)
  = \DDD(g_1) \o t^i \o g_2
  + g_1 \o t^i \o \DDD(g_2), \quad \forall i \geq 2,\
  g_1, g_2 \in T_S M.
\end{equation}
Moreover, it satisfies another rule which resembles an order-two
differential operator condition:
\begin{multline} \label{e:tddiff} \DDD(g_1 \o t \o f \o t \o g_2) =
  \DDD(g_1 \o t \o f) \o t \o g_2 \\ + g_1 \o t \o \DDD(f \o t \o g_2)
  - g_1 \o t \o \DDD(f) \o t \o g_2.
\end{multline}
This reduces the computation to the six equations
\begin{align} 
  &{\rm {(1)}}\en\ \DDD(1_{m}) \ =\ m \cdot 1_{m}, \qquad {\rm
    {(2)}}\en\ \DDD(s \cdot 1_{m})\ =\ m s \cdot 1_{m},\nonumber
  \\
  &{\rm {(3)}}\en\ \DDD(1_{m} \o t \o 1_{n}) \ =\ (m+n+1) (1_{m} \o t
  \o 1_{n}),\nonumber
  \\
  &{\rm {(4)}}\en\ \DDD(1_{m} \o st \o 1_{n}) \ =\ (m+n+\frac{1}{2})
  1_{m} \o st \o 1_{n} - \half\sum_{j=0}^n 1_{m+j} \o ts \o
  1_{n-j}, \label{e:tdfla3}
  \\
  &{\rm {(5)}}\en\ \DDD(1_{m} \o ts \o 1_{n}) \ =\ (m+n+\frac{1}{2})
  \cdot 1_{m} \o ts \o 1_{n} - \half \sum_{j=0}^m 1_{m-j} \o st \o
  1_{n+j},\nonumber
  \\
  &{\rm {(6)}}\en\ \DDD(s \cdot 1_{m} \o t \o s \cdot 1_{n}) \ = \
  (m+n)s \cdot 1_{m} \o t \o s \cdot 1_{n} - \half \sum_{{0 \leq j
      \leq m+n}\atop{j \neq m}} s\cdot 1_{j} \o t \o s\cdot
  1_{m+n-j}.\nonumber
\end{align}

Next, we will show that $\DDD$ is positive and use this to compute its
kernel. More precisely, let us assume without loss of generality that
$\bk = \bQ$ (note that $T_S M$ is the tensor product of its $\bQ$-form
by $\bk$).  We equip $T_S M$ with the inner product induced by the
monomial basis
\begin{equation}\label{e:tsm-basis}
  s^{\epsilon_0} \cdot 1_{j_0} \o t^{i_1} s^{\epsilon_1} \o
  1_{j_1} \o \cdots \o t^{i_r} s^{\epsilon_r} \o 1_{j_r},
\end{equation}
for all $r\geq0$ and $\epsilon_0, \ldots, \epsilon_r \in \{0,1\}$,
$j_0, \ldots, j_r \geq 0$, and $i_1, \ldots, i_r \geq 1$, i.e., the
unique inner product for which this basis is orthonormal.

Next, for each $k\geq0$, we fix $i_1, \ldots, i_k \geq 2$, $j_0,
\ldots, j_k \geq 0$, and $p_0, \ldots, p_k \geq 0$ such that $p_r \leq
j_r$ for all $r$, and $q_0, \ldots, q_k \geq 0$ such that $q_r \leq
p_r+1$ for all $r$.  Then we define the $\DDD$-invariant subspace,
$V_{k,i_\idot, j_\idot, p_\idot, q_\idot}$, spanned by elements of the
form
\[
f_0 \o t^{i_1} \o f_1 \o \cdots 
\o t^{i_k} \o f_k,
\]
where, for all $r \in \{0,\ldots, k\}$, $f_r \in T_S^{j_r} \langle 1,
t, st, ts \rangle$, and it is homogeneous in $s$ and $t$ of degrees
$q_r$ and $p_r$, respectively.
\begin{claim} \label{cl:hcl-p}
\begin{itemize}
\item[(a)] The operator $\DDD$ is positive with kernel of dimension
  $\leq 1$ on $V_{k,i_\idot, j_\idot, p_\idot, q_\idot}$.
\item[(b)] This kernel is nonzero if and only if $p_r = j_r$
 and $q_r \in \{p_r, p_r+1\}$, for all $r$. In this case it is spanned
by the element \eqref{e:hcm}  for $\epsilon_r = q_r-p_r$ where 
$r=0,1,\ldots,k$.
\end{itemize}
\end{claim}
\begin{proof} 
  In view of Lemma \ref{l:tp} and \eqref{e:tdder}, it suffices to
  prove (a) and (b) in the case that $k=0$. We then denote for
  simplicity $j=j_0, p=p_0$, and $q=q_0$.  The fact that $\DDD$ is
  self-adjoint here is immediate from \eqref{e:tddiff} and
  \eqref{e:tdfla3}(1)--(6).  To show that it is positive, observe
  again from \eqref{e:tddiff}--\eqref{e:tdfla3}(5) that the symmetric
  matrix $B = (b_{ij})$ for $\DDD$ on $V_{0,i_\idot, j, p, q}$ in the
  monomial basis of \eqref{e:tsm-basis} satisfies $b_{ii} \geq \sum_{j
    \neq i} |b_{ij}|$ for all $i$.  Equality holds for all $i$ if and
  only if the condition of part (b) is satisfied: $p = j$ and $q \in
  \{p, p+1\}$, i.e., in those formulas \eqref{e:tdfla3}(3)--(6) which
  arise, always $m=n=0$, and also \eqref{e:tdfla3}(2) does not arise
  (and neither \eqref{e:tdfla3}(1)).

  Observe that the (undirected) graph obtained from $B$, with vertex
  set $1,\ldots, \dim V_{0,i_{\idot},j,p,q}$ and with edges between
  $i$ and $j$ if $b_{ij} \neq 0$, is connected. The claim then follows
  from Lemma \ref{l:pos-crit}.
\end{proof}

Now, we complete the proof of Lemma \ref{l:hcl-baby}.  We showed that
$\DDD$ is positive with kernel spanned by the elements
\eqref{e:hcm}. By Lemma \ref{l:gc}, this implies that $(T_S M,
\tilde \partial^0)$ is quasi-isomorphic to the subcomplex spanned by
the given elements. It is easy to check that $\tilde \partial^0$ has
zero differential restricted to this subcomplex, and that the elements
appearing in \eqref{e:hcm} are linearly independent.
\end{proof}

\subsection{The complex $\cll$}\label{Cl}
Write $[s, m] := s \cdot m - (-1)^{|m|} m \cdot_R s$, for $s \in S,\ m
\in M^{\otimes\ell}$.  We keep the notation $\o=\o_{_S}$ and put $\cll
:= M^{\otimes\ell}/[S, M^{\otimes \ell}]$.  We introduce a triple
$d_s, \partial^0$, and $\partial^1$, of differentials on $\cll$, as
follows.

The definition of the differentials breaks up into two cases: the case
of degree zero in $t$ and the case of positive degree in $t$,
respectively.  In the case of degree zero in $t$, we use the formulas
\begin{align*}
&d_s(1_\ell) = s \cdot 1_\ell, &&d_s(s \cdot 1_\ell) = 0, \\
&\partial^0(1_\ell) = 0, &&\partial^0(s \cdot 1_\ell) = 0, \\
&\partial^1(1_\ell) = 0, &&
\partial^1(s \cdot 1_\ell) = (1+(-1)^\ell) \cdot \sum_{0\leq j\leq \ell-1}
1_j \cdot t \cdot 1_{\ell-j}.
\end{align*}

In the case of positive degree in $t$, we observe that the
differentials $\tilde d_s$, $\tilde \partial^0$ and
$\tilde \partial^1$ on $M^{\o\ell}$, introduced in the previous
section, descend to well defined differentials $d_s$, $\partial^0$,
and $\partial^1$, respectively, modulo commutators with $s$.  This
gives the required differentials on $\cll = M^{\otimes\ell} /
[S,M^{\otimes \ell}]$.

Let the cyclic group $\bZ/\ell$ act on $\cll $ by cyclic permutations
of tensor components using the Koszul sign rule, in terms of the
$\bZ/2$ grading above.  We write $(\cll)_{\bZ/\ell}$ for the
corresponding space of $\bZ/\ell$-coinvariants.

The next lemma is a `cyclic counterpart' of Lemma \ref{l:babg0} (which
we will also not actually need).
\begin{lemma}\label{l:babg} The map $\zeta$, of Lemma \ref{l:babg0},
  descends to an isomorphism of bicomplexes
\[\zeta\cc:\en
\left(\bigoplus\nolimits_{\ell \geq 1}\ (\cll)_{\bZ/\ell},\ d_s,\
  \partial^0+\partial^1\!\right)\ \iso\ (\oDR_t A_\tau,\ \sd,\,
\si_t).
\]
\end{lemma}

\begin{proof}[Sketch of Proof] The argument is based on the following
  general result.  Let $A$ be an arbitrary algebra, let $V\sset A$ be
  a $\bk$-submodule, and consider $\alpha=a_0 \sd\! a_1 ... \sd\! a_m$
  and $\alpha'=a'_0 \sd\! a'_1 ... \sd\!  a'_n$ for arbitrary elements
  $a_0, a_1,\ldots, a_m,\ a'_0, a'_1,\ldots, a'_n\in V$.  Then, the
  following implication holds:
\begin{equation}\label{imply}
  V\cdot V=0\quad\Rightarrow\quad
  \bigl((\sd\alpha)\cdot t\cdot \beta\cdot t\cdot \alpha' 
  +
  (-1)^{|\alpha|+|\beta|} \alpha\cdot t\cdot
  \beta\cdot t\cdot 
  (\sd\alpha')\bigr)\cc=0,\quad\forall \beta\in \oOm_t A.
\end{equation}
Now, equation $V\cdot V=0$ obviously holds in the case where
$A=A_\tau$ and $V=\bk\cdot\ta$. Further, it is straightforward to
verify that the map $\zeta$ sends the $\bk$-submodule $\bplus_\ell \
[S, M^{\o\ell}]$, of $\bplus_\ell \ M^{\o\ell}$, onto the
$\bk$-submodule of $\oOm_t(A_\ta)$ spanned by the elements of the form
$(\sd\alpha)\cdot t\cdot \beta\cdot t\cdot \alpha' +
(-1)^{|\alpha|+|\beta|} \alpha\cdot t\cdot \beta\cdot t\cdot
(\sd\alpha')$, which appear in \eqref{imply}.

Using this, and the cyclic symmetry of $\oDR_t A_\tau$, one shows that
the map $\zeta\cc$ is well-defined.  To see it is a bijection, we
first observe that $\oDR_t A_\tau$ can be presented as the quotient of
$\oOm_t A_\tau$ by the above relation, for $V = \bk \cdot \ta$, along
with the relation imposing cyclic symmetry.  Then, the result follows
from the above observations and the fact that $\zeta$ is a bijection.
\end{proof}

\subsection{}\label{compute} We now compute the homology of 
the complexes $(\cll, \partial^0)$ and $(\cll / d_s
\cll, \partial^0)$.

\begin{lemma} \label{l:hcl} \vi A basis for $H(\cll, \partial^0)$ as a
  free $\bk$-module consists of elements of the form
\begin{equation} \label{e:hcl}
(st+ts)^{\o j_0} \o t^{i_1} s^{\epsilon_1} \o (st+ts)^{\o j_1} \o
\cdots \o t^{i_k} s^{\epsilon_k} \o (st+ts)^{\o j_k},
\end{equation}
where $i_1, \ldots, i_k \geq 2$, $j_0, \ldots, j_k \geq 0$, and
$\epsilon_1, \ldots, \epsilon_k \in \{0,1\}$, together with the
elements, for all $k \in \{0, \ldots, \ell-1\}$,
and
$\epsilon \in \{0,1\}$,
\begin{equation}\label{e:hcl2}
  \sum_{{
      1 \leq r \leq k; j_0, \ldots, j_k \geq 0:}\atop
    {j_0 + \cdots + j_k = \ell-k}}
  (-1)^{(\epsilon-1)(j_0+j_1+\cdots+j_{r-1}+r-1)}\cdot
  1_{j_0} \otimes
  ts \otimes \cdots \otimes
  1_{j_{r-1}} \otimes t s^{\epsilon} \otimes \cdots
  \otimes
  1_{j_{k-1}}
  \otimes ts \otimes 1_{j_k}.
\end{equation}

\vii A basis for $H(\cll/ d_s \cll, \partial^0)$ as a free
$\bk$-module consists of elements of the form
\begin{equation} \label{e:hcld} (st+ts)^{\o j_0} \o t \o (st+ts)^{\o
    j_1} \o t s^{\epsilon_2} \o (st+ts)^{\o j_2} \o \cdots \o t
  s^{\epsilon_k} \o (st+ts)^{\o j_k},
\end{equation}
where $i_1, \ldots, i_k \geq 2$, $j_0, \ldots, j_k \geq 0$, and
$\epsilon_2, \ldots, \epsilon_k \in \{0,1\}$, together with, for all
$k \in \{0, \ldots, \ell-1\}$, the element
\begin{equation} \label{e:hcld2} \sum_{{ 1 \leq r \leq k; j_0, \ldots,
      j_k \geq 0:}\atop { j_0 + \cdots + j_k = \ell-k}}
  (-1)^{j_0+j_1+\cdots+j_{r-1}+r-1}\cdot 1_{j_0} \otimes ts \otimes
  \cdots \otimes 1_{j_{r-1}} \otimes t \otimes \cdots \otimes
  1_{j_{k-1}} \otimes ts \otimes 1_{j_k}.
\end{equation}
\end{lemma}
See the next subsection for an explanation of formulas \eqref{e:hcl2}
and \eqref{e:hcld2} (they will appear as the associated graded
expressions of \eqref{psi} with respect to the increasing filtration
by the degree in $s$, i.e., the number of times $s$ appears).

\begin{proof}
  To prove part (i), we will use the proof of Lemma
  \ref{l:hcl-baby}. The key step is to define a homotopy $D$ on $\cll$
  similar to the operator $\tilde D$. As before, assume $f_1, \ldots,
  f_k \in T_S \langle 1,s \rangle \subseteq T_S M$.
Then, for $i_1, \ldots, i_k \geq 1$ and $p_1,p_2 \geq 0$, define
\begin{multline*}
  D(1_{p_1} \o s^\epsilon t^{i_1} f_1 \cdots t^{i_{k-1}} f_{k-1}
  t^{i_k} \o 1_{p_2}) = \tilde D(1_{p_1} \o s^\epsilon t^{i_1} f_1)
  t^{i_2} f_2 \cdots t^{i_{k-1}}f_{k-1} t^{i_k} \o 1_{p_2} \\ +
  (-1)^{\epsilon \cdot\deg(t^{i_1} f_1 \cdots t^{i_{k-1}} f_{k-1}
    t^{i_k} \o 1_{p_2})}\cdot 1_{p_1} \o t^{i_1} \tilde D(f_1 \cdots
  t^{i_{k-1}} f_{k-1} t^{i_k} \o 1_{p_2} \cdot_R s^\epsilon),
\end{multline*}
and so that $D$ kills elements of degree zero in $t$ (elements of the
form $(f_1)\cc$ with $f_1$ as above). By definition, $D$ commutes with
the $\bZ/\ell$ action.

By the same argument as before, we can assume $\bk=\bQ$, and then $[D,
\partial^0]$ is a positive operator on $\cll$, viewed as an inner product
space with orthonormal basis
\begin{equation}\label{e:tsm-basis-cyc} 
  1_{j_0} \o
  t^{i_1} s^{\epsilon_1} \o 1_{j_2} \o \cdots
  \o t^{i_k} s^{\epsilon_k} \o 1_{j_k},
\end{equation}
for $\epsilon_1, \ldots, \epsilon_k \in \{0,1\}$, $j_0, \ldots, j_k
\geq 0$, and $i_1, \ldots, i_k \geq 1$.  The only difference comes in
computing the kernel, i.e., the statement of Claim \ref{cl:hcl-p}(b)
has to be modified (with the same proof, i.e., still using Lemma
\ref{l:pos-crit}).  Here, we just have to add also the case where
$k=0$, $p_0$ is arbitrary, and $q_0 \in \{p_0, p_0+1\}$. One can
verify that kernel in this case is spanned by the element
\eqref{e:hcl2} (note that in this formula, there is a different
parameter $k$, which need not be zero).

For part (ii), note that $D$ commutes with $d_s$, so it descends to an
operator on $\cll / d_s \cll$, which we also denote (abusively) by
$D$.  Since $[D,\partial^0]$ is positive on $\cll$ and preserves $d_s
\cll$, it is positive on $(d_s \cll)^\perp$.  Now, consider the
composition $(d_s \cll)^\perp \into \cll \onto \cll /d_s \cll$, which
is an isomorphism. We claim that it commutes with the actions of
$[D, \partial^0]$.  This will imply that $[D, \partial^0]$ is also
positive on $\cll / d_s \cll$.  The only subtlety here is that $(d_s
\cll)^\perp$ is not preserved by $D$ and $\partial^0$.  However, by
definition, the operators $\partial^0$ and $D$ on $\cll / d_s \cll$
are obtained from the projection $\cll \onto \cll/d_s \cll$, so the
claim follows.

As a result, $(\cll/d_s \cll, \partial^0)$ is quasi-isomorphic to the
subcomplex spanned by the projection of the subcomplex appearing in
part (b) to $\cll / d_s \cll$. This is evidently spanned by the
elements \eqref{e:hcld}--\eqref{e:hcld2}, since this is the collection
of elements from \eqref{e:hcl}--\eqref{e:hcl2} which do not project to
zero, and it is easy to see that these remaining elements
\eqref{e:hcld}--\eqref{e:hcld2} are linearly independent.  As before,
the differential $\partial^0$ is zero on this subcomplex, so the
result follows.
\end{proof}

\subsection{} \label{ss:cllhom} We continue to keep the notation
$\otimes=\otimes_{_S}$.  For each $\ell\geq1,\ q\geq0,$ and
$\epsilon\in\{0,1\}$, we consider the image in $\cll$ of the element
\begin{equation}\label{psi}
  \Upsilon^\epsilon_{q,\ell}\ :=\!
  \underset{{k \geq 0,\, 1 \leq r \leq k,\, i_1, \ldots, i_k \geq 1,\,
      j_0, \ldots, j_k \geq 0:}\atop
    {i_1 + \cdots + i_k
      = q,\, j_0 + \cdots + j_k = \ell-k}}
  \sum \hskip-6mm
  (-1)^{\nu_\epsilon(j_0,\ldots,j_{r-1})}\cdot
  1_{j_0} \otimes 
  t^{i_1}s \otimes \ldots\otimes 1_{j_{r-1}} \otimes
  t^{i_{r}}s^{\epsilon} \otimes \cdots \otimes
  1_{j_{k-1}} \otimes t^{i_k}s \otimes 1_{j_k},
\end{equation}
where
$\nu_\epsilon(j_0,\ldots,j_{r-1}):=(\epsilon-1)(j_0+\cdots+j_{r-1}+r-1)$.

\begin{lemma}\label{page2} \begin{enumerate}
  \item[(1)] For all $q \geq 0$ and all $\epsilon \in \{0,1\}$, we
    have that $\ (\partial^0 + \partial^1) \Upsilon^\epsilon_{q,\ell}
    = 0$.

  \item[(2)] Using the notation of Lemma \ref{l:hcl}, one has\vskip2pt

    \vi The homology $H(\cll, \partial^0 + \partial^1)$ is spanned by
    $\Upsilon^\epsilon_{q,\ell},\ q \geq 0,\, \epsilon \in \{0,1\}$.

    \vii The homology $H(\cll/d_s \cll, \partial^0 + \partial^1)$ is
    spanned by $\Upsilon^0_{q,\ell},\ q \geq 0$.
\end{enumerate}
\end{lemma}
\begin{proof} Part (1) is an explicit (and straightforward)
  computation, which we omit.

  To prove (2), we introduce an ascending filtration on $\cll$ by the
  degree in $s$ (a nonnegative integer).  This descends to a
  filtration on $\cll / d_s \cll$ as well.  The differentials
  $\partial^0$ and $\partial^1$ respect the filtration.  The
  associated graded complexes are
\begin{align}
  \gr_s(\cll,\ \partial^0 + \partial^1) &= (\gr_s \cll,\ \partial^0)
  \cong (\cll,\ \partial^0),\label{grs}\\
  \gr_s(\cll / d_s \cll,\ \partial^0 + \partial^1) &= (\gr_s(\cll /
  d_s \cll),\ \partial^0) \cong (\cll / d_s \cll,
  \ \partial^0).\nonumber
\end{align}

One has the standard spectral sequence associated with the filtered
complex $(\cll, \partial^{0} + \partial^1)$.  From the isomorphism in
\eqref{grs}, we see that the first page of the spectral sequence is
given by the homology of the complex $(\cll/ d_s \cll, \partial^{0})$.
This homology is described by Lemma \ref{page2}.  The second page is
obtained from the first page by taking the homology of the
differential induced by $\partial^1$.

We now prove (i); the proof of (ii) is similar.  Since $\partial^0$
and $\partial^1$ are homogeneous of degree one in $t$, we can prove
the statement in each fixed degree $q$ in $t$.  Let $C_{\ell,q}
\subset \cll$ be the subset of elements of degree $q$ in $t$.  The
proof rests on comparing $C_{\ell,q}$ with the part of the complex
$(H(\gr_s \cll, \partial^0), \partial^1)$ in degree $q+\ell$.

Let $(T_S^\ell M)_{q}$ denote the part of $T_S^\ell M$ of degree $q$
in $t$, and $(T_S^\ell M)_q'$ the subspace spanned by \eqref{e:hcm}.
Let $C'_{\ell,q} \subseteq H(\gr_s \cll,
\partial^0)$ be the subspace of degree $q$ in $t$, which is spanned by
the elements \eqref{e:hcl}--\eqref{e:hcl2}.  Let $C'_\ell =
\bigoplus_{q \geq 0} C'_{\ell,q}$. The differential $\partial'$ on
$\cll'$, in terms of the basis \eqref{e:hcl}--\eqref{e:hcl2}, is given
by $\partial' = pr \ccirc \partial^1$, where $\partial^1$ is given by
the usual formula on $\cll$, and $pr: \cll \to \cll'$ is the
orthogonal projection in the basis \eqref{e:tsm-basis-cyc}: more
precisely, one first projects to $\ker(\partial^0)$ orthogonally and
then takes the quotient by $\im(\partial^0)$; equivalently, one
projects orthogonally to the span of the elements
\eqref{e:hcl}--\eqref{e:hcl2}.  The reason why this is the
differential is that $\cll'$ is obtained from $\cll$ as the kernel of
the positive operator $[D,\partial^0]$, and so the formula for
$\partial'$ follows from the proof of Lemma \ref{l:gc} (namely, from
the fact that $\im([D,\partial^0])$, the orthogonal complement to the
span of \eqref{e:hcl}--\eqref{e:hcl2}, is an acyclic direct summand of
the complex $(C_\ell,\partial^0)$).

A key step is to construct a map $\psi_\ell$, essentially
multiplication by $t^{\otimes \ell}$ with some constants, from
$C_{\ell,q}$ to $C_{\ell,q+\ell}'$, which commutes with the
differential $d_s$, and sends $\partial^0+\partial^1$ to $\partial'$.
The precise definition of $\psi_\ell$ is given on \eqref{e:tsm-basis}
by
\begin{multline}
  \psi_\ell(s^{\epsilon_0} \cdot 1_{j_0} \otimes t^{i_1}
  s^{\epsilon_1} \otimes 1_{j_1} \otimes \cdots \otimes t^{i_k}
  s^{\epsilon_k} \otimes 1_{j_k}) \\ = ((j_0+1)^{-1}s)^{\epsilon_0}
  \cdot (st+ts)^{\otimes j_0} \otimes t^{i_1+1}
  ((j_1+1)^{-1}s)^{\epsilon_1} \otimes (st+ts)^{\otimes j_1} \otimes
  \\ \cdots \otimes t^{i_k+1} ((j_k+1)^{-1}s)^{\epsilon_k} \otimes
  (st+ts)^{\otimes j_k}.
\end{multline}
\begin{claim}\label{c:psiell}
  The map $\psi_\ell$ is an isomorphism $C_{\ell,q} \to
  C'_{\ell,q+\ell}$ for all $q \geq 0$ which intertwines $d_s$ with
  $d_s$ and intertwines $\partial^0+\partial^1$ on $\cll$ with
  $\partial'$ on $\cll'$.
\end{claim}
To prove the claim, one has to verify that $\psi_\ell$ is an
isomorphism, that it preserves $d_s$, and that it sends
$\partial^0+\partial^1$ to $\partial'$. These facts all follow from
straightforward explicit computations, which we omit.

We now complete the proof of Lemma \ref{page2}. We do this by
induction on $q$.  The base case is where $q < \ell$.  In this case,
$C_{\ell,q}'$ is already spanned by the elements $\gr_s
\Upsilon^\epsilon_{q,\ell},\ q\geq0, \epsilon \in \{0,1\}$, which are
exactly \eqref{e:hcl2} up to nonzero constant factors, and, moreover,
these elements are killed by $\partial'$.

For the inductive step, assume that, in degree $q$,
$H(\cll, \partial^0+\partial^1)$ is spanned by the elements
$\Upsilon^\epsilon_{q,\ell}$ for $\epsilon \in \{0,1\}$.  By Claim
\ref{c:psiell}, this implies that, in degree $q+\ell$,
$H(\cll',\partial')$ is spanned by
$\psi_\ell(\Upsilon^\epsilon_{q,\ell})$.  We claim that
\begin{equation}\label{e:psielleqn}
  \psi_\ell(\Upsilon^\epsilon_{q,\ell}) = \gr_s \Upsilon^\epsilon_{q+\ell,\ell}.
\end{equation}
Then, by the spectral sequence for $(\cll, \partial^0 + \partial^1)$,
this implies that the homology of $(\cll, \partial^0 + \partial^1)$ in
degree $q+\ell$ is spanned by the classes
$\Upsilon^\epsilon_{q+\ell,\ell}$, completing the induction.

It remains to prove \eqref{e:psielleqn}.  Clearly, the RHS is given by
restricting the sum in \eqref{psi} to the summands where $k=\ell$.  In
this case, each such summand is obtained from the summand of element
$\Upsilon^\epsilon_{q,\ell}$.  This proves \eqref{e:psielleqn}.
 \end{proof}

\subsection{Proof of Proposition \ref{p:heq}}
The proof amounts to a generalization of the preceding material from
$A_\tau$ to $A'$.  Here we revert to the notation
$\otimes=\otimes_\bk$ used before (and as we will use in the remainder
of the paper).

Below, we will use, for every $p,q\geq0$, a direct sum decomposition
$\oOm^{p,q}_tA'=\oplus_{\ell\geq1}\ [\oOm^{p,q}_tA']_{(\ell)}$, where
$[\oOm^{p,q}_tA']_{(\ell)}$ denotes the span of all elements in which
$\oA$ appears $\ell$ times.  This induces a similar decomposition
$\oDR^{p,q}_tA'=\oplus_{\ell\geq1} \ [\oDR^{p,q}_tA']_{(\ell)}$.
Also, let a superscript of $q$ denote the degree $q$ part in $t$.

For each $\ell\geq1,$ we define a map $\zeta:\ \overline A^{\otimes
  \ell} \otimes M^{\o_S\ell} \too [\oOm_t A']_{(\ell)}$ as follows. In
positive degree in $t$, the map is given, for every $f_1, \ldots,
f_\ell \in \{1, t^i, t^i s \mid i \geq 1\} \subset M$ not all equal to
$1$, and $a_1, \ldots, a_\ell \in \oA$, by the formula
\[
(a_1 \otimes \cdots \otimes a_\ell) \otimes (f_1 \otimes_{_S} \cdots
\otimes_{_S} f_\ell) \mapsto (f_1 \star a_1)(f_2 \star a_2) \cdots
(f_\ell \star a_\ell),
\]
where we put
\[
1 \star a = \sd\! a, \quad t^i \star a = t^i a, \quad t^i s \star a =
t^i \sd\! a, \quad i \geq 1.
\]

In degree zero in $t$, we use the formulas
$$
(a_1 \otimes \cdots \otimes a_\ell) \otimes (1_{\ell})\ \mapsto\ a_1
\sd\! a_2 \cdots \sd\! a_\ell, \qquad (a_1 \otimes \cdots \otimes
a_\ell) \otimes (s \cdot 1_{\ell})\ \mapsto\ \sd\! a_1 \sd\! a_2
\cdots \sd\! a_\ell.
$$

Recall next that the cyclic group $\bZ/\ell$ acts on $\cll$ by cyclic
permutations of tensor components using the Koszul sign rule. We let
$\bZ/\ell$ act on $\oA^{\otimes \ell}$ by ordinary cyclic permutations
and we equip $\overline A^{\otimes \ell} \otimes \cll$ with
$\bZ/\ell$-diagonal action.  Further, we extend each of the three
differentials $d= d_s,\partial^0$, and $\partial^1$ on $\cll$ to the
differentials $d_A:=\Id\o d$ on $\overline A^{\otimes \ell} \otimes
\cll$, where $\Id: \overline A^{\otimes \ell} \to \overline A^{\otimes
  \ell}$ is the identity.  Similarly, we extend the tilde versions
$\tilde d$ to $\tilde d_A:=\Id \o \tilde d$ on $\overline A^{\otimes
  \ell} \otimes M^{\o_S\ell}$.  Finally, for every $q\geq 0$, let
$(\overline A^{\otimes \ell} \otimes M^{\o_S\ell})^q$ and $(\overline
A^{\otimes \ell} \otimes \cll)_{\bZ/\ell}^q$ be the homogeneous
component of $\overline A^{\otimes \ell} \otimes M^{\o_S\ell}$ and
$(\overline A^{\otimes \ell} \otimes \cll)_{\bZ/\ell}$ of degree $q$
in $t$, respectively.
\begin{lemma}\label{zeta} For  $\ell\geq1$, 
\vi For every  $q\geq 0$, the map $\zeta$
induces isomorphisms of complexes
$$\bigl((\overline A^{\otimes \ell}
\otimes M^{\o_S\ell})^q ,\ d_{s,A}\bigr) \iso
\bigl([\oOm^{\hdot,q}_tA']_{(\ell)},\ \sd\bigr), \quad\text{\em
  and}\quad \bigl((\overline A^{\otimes \ell} \otimes
\cll)_{\bZ/\ell}^q ,\, d_{s,A}\bigr) \iso \bigl([\oDR^{\hdot,q}_t
A']_{(\ell)},\ \sd\bigr).
$$

\vii For $\epsilon\in\{0,1\}$ and $\ell\geq1$, put
$\Upsilon^\epsilon_\ell:= \sum_{q\geq 0} \frac{1}{(q+\ell-1)!}\cdot
\Upsilon^\epsilon_{q,\ell}$.  Then, for every $\alpha\in \overline
A^{\otimes \ell}$, the following diagram commutes:
$$\xymatrix{
  \bigl((\overline A^{\otimes \ell} \otimes M^{\o_S\ell})^0,
  \ \partial^0_A+\partial^1_A\bigr)
  \ar[d]^<>(0.5){\zeta}_<>(0.5){\cong}\ar[rrrr]^<>(0.5){\alpha\,\o\,
    s^\epsilon\cdot 1_{\ell} \,\ \mto\ \, \alpha\,\o\,
    \Upsilon^\epsilon_\ell} &&&& \ \bigl(\prod_{q\geq0}\ (\overline
  A^{\otimes \ell} \otimes
  \cll)_{\bZ/\ell}^q,\ \partial^0_A+\partial^1_A\bigr)
  \ar[d]^<>(0.5){\zeta}_<>(0.5){\cong}\\
  \bigl([\oOm A']_{(\ell)},\ \msb\bigr)\ \ar[rrrr]^<>(0.5){\EE'} &&&&\
  \bigl(\prod_{q\geq0}\ [\oDR^{\hdot,q}_tA']_{(\ell)},\ \si_t\bigr).
}
$$
\end{lemma}

\begin{proof}[Sketch of proof] In the special case where $\oA=\bk$ is
  a rank one free $\bk$-module, $A'=A_\tau=\bk\oplus\bk\ta$.  In this
  case, the isomorphisms of part (i) of the lemma, as well as fact
  that the map $\zeta$ respects the differentials, reduce to Lemmas
  \ref{l:babg0} and \ref{l:babg}. The proof in the general case is
  similar. Note that implication \eqref{imply}, used in the proof of
  Lemma \ref{l:babg}, is applicable in the general case as well, since
  in the algebra $A'$ it follows that $\oA\cdot \oA=0$.

  Finally, commutativity of the diagram of part (ii) is verified by
  direct computation using the explicit formula for the Tsygan map and
  for the map $\zeta$. We omit the details.
\end{proof}

\begin{remark}\label{l:upstau} In the special case $A=A_\tau$, i.e.,
$\oA=\bk$,
the   commutative diagram of Lemma \ref{zeta}(ii) reduces to the diagram:
$$\xymatrix{
  \bigl((M^{\o_S\ell})^0, \ \tilde \partial^0+\tilde \partial^1\bigr)
  \ar[d]^<>(0.5){\zeta}_<>(0.5){\cong}\ar[rrrr]^<>(0.5){s^\epsilon\cdot
    1_{\ell} \,\ \mto\ \, \Upsilon^\epsilon_\ell} &&&& \
  \bigl(\prod_{q\geq0}\
  (\cll)_{\bZ/\ell}^q,\ \partial^0+\partial^1\bigr)
  \ar[d]^<>(0.5){\zeta\cc}_<>(0.5){\cong}\\
  \bigl([\oOm A_\tau]_{(\ell)},\ \msb\bigr)\ \ar[rrrr]^<>(0.5){\EE'}
  &&&&\ \bigl(\prod_{q\geq0}\ [\oDR^{\hdot,q}_tA_\tau]_{(\ell)},\
  \si_t\bigr). }
$$
 
We remark that it is immediate from commutativity of this diagram that
$(\partial^0+\partial^1)\Upsilon^\epsilon_\ell$ maps to zero in
$(C_\ell)_{\Z/\ell}$.  This way, one obtains the equation of Lemma
\ref{page2}(1), but only modulo cyclic permutations.
 \end{remark}

Proposition \ref{p:heq} is a direct consequence of
the above lemma and Lemma \ref{page2}.

\section{Periodic and negative extended cyclic homology}
\label{negative}

\subsection{}
We define extended periodic and negative cyclic homology of $A$ as
follows
$$
\oEHC^\per(A) := H(\oDR_t \dpu, \si_t - u\sd), \quad \text{and} \quad
\oEHC^-(A) := H(\oDR_t \dspu, \si_t - u \sd).
$$
Similarly, let us define versions using the $t$-adic completion,
$\oEHC^{\per,\wedge} := H(\ohDRt \dpu, \si_t - u\sd)$ and
$\oEHC^{-,\wedge}(A) := H(\ohDRt \dspu, \si_t - u \sd)$.

\begin{remark}\label{internal} The complexes above can be viewed as
  bicomplexes with
  differentials $u\sd$ and $\si_t$.  Each of the two differentials
  preserves `internal degree,' defined as the sum of de Rham degree
  and degree in $t$ minus degree in $u$. Separating homogeneous
  components with respect to internal degree gives decompositions of
  $(\ohDRt \dpu, \si_t - u\sd)$ and $(\ohDRt \dspu, \si_t - u \sd)$
  into infinite products of subcomplexes.  (Unlike for $(\oDR_t
  \otimes R, \si_t-u\sd)$, the uncompleted versions here are not
  direct sums of subcomplexes in each internal degree, which is why we
  restrict our attention in this section to the $t$-adically completed
  versions, which are better behaved.)
\end{remark}

\begin{proposition}\label{p:ext-per-cyc}
  The projection $(\ohDRt \dpu, \si_t - u \sd) \onto (\oDR \dpu, -u
  \sd)$, modulo $t$, is a morphism of complexes, inducing an
  isomorphism on homology $\oEHC^{\per,\wedge}(A) \iso \oHD(A) \dpu$.
\end{proposition}
\begin{remark} 
  The above isomorphism, as well as all other isomorphisms which
  appear in this section, preserves the `homological grading,' which
  is defined to be de Rham degree minus twice the degree in $u$.
\end{remark}

\begin{proof}[Proof of Proposition]
  Exactly as in the proof of Proposition \ref{p:ext-cyc}, we can take
  homology first with respect to $u\sd$ and use a spectral sequence.
  This spectral sequence is obviously convergent on each homogeneous
  component with respect to the internal degree, cf.~Remark
  \ref{internal}; moreover, it collapses on the first page (the
  homology with respect to $u\sd$).
\end{proof}
\begin{remark}\label{r:ext-neg-cyc}
  One can similarly compute $\oEHC^{-,\wedge}(A)$.  First of all,
  observe that a summand of $\ker_{\oHH(A)}(B \circ I)$ (recall
  \eqref{e:connes-seq-maps} for the definitions of $B$ and $I$) splits
  off in degree zero in both $t$ and $u$. As in the literature, let us
  make the abuse of notation $B := B \circ I: \oHH(A) \to \oHH(A)$ for
  the induced differential on $\oHH(A)$.

\begin{claim}\label{e:ehc-neg-seq} There is a natural
  exact sequence
$$
0\to\prod_{q \geq 1} \frac{H(\oDR_t / \sd \oDR_t,
  \si_t)^{\hdot,q}}{(\sdi \si_t)^{q}(\oHD(A))} \ \stackrel{\sd}\too\
\frac{\oEHC^{-,\wedge}(A)}{\ker_{\oHH(A)}(B)}\ \stackrel{\pi}\too \
\prod_{m \geq 1} u^m \cdot \ker_{\oHD(A)}((\sdi \si_t)^{m+1})\to0.
$$
\end{claim}

Here, the map $\sd$ is induced by the obvious map $\osd: \ohDRt / \sd
\ohDRt \to \ohDRt$, and the map $\pi$ mods by the constant term in $u$
and by positive degree in $t$.  Further, $(\sdi \si_t)^q: \oHD(A) \to
H(\oDR_t / \sd \oDR_t, \si_t)^{\hdot,q}$ is the map considered in
\S\ref{inv_sec} (see the proof of Proposition \ref{inverse}). More
explicitly, the map $\sdi \si_t$ induces an endomorphism of
$H(\oDR_t/\sd \oDR_t, \si_t)$ as follows. Given an element of
$H(\oDR_t/\sd \oDR_t, \si_t)$, one first lifts it to $\oDR_t$, then
applies $\si_t$. After that, one applies $\sdi$, and projects to
$\oDR_t / \sd \oDR_t$, which lands back in $\ker_{\oDR_t/\sd
  \oDR_t}(\si_t)$.  Finally, we precompose with the inclusion $\oHD(A)
\subseteq \oHC(A) \cong H(\oDR_t/\sd \oDR_t, \si_t)^{\hdot,0}
\subseteq H(\oDR_t/\sd\oDR_t,\si_t)$.

By definition, $(\sdi \si_t)^m |_{\oHD_{k}(A)}=0$ when $k < 2m$.
Thus, for fixed $k$ and $m > k/2-1$, the component $u^m \cdot
\ker_{\oHD_k(A)}((\sdi \si_t)^{m+1})$ of the above product is nothing
but $u^m \cdot \oHD_k(A)$.

Further, using isomorphism \eqref{qhc}, the short exact sequence of
Claim \ref{e:ehc-neg-seq} becomes

\begin{equation}\label{e:ehc-neg-seq2}
  0\to\ \prod_{q \geq 1} \frac{\oHC(A) \cdot t^q}{\im(S^{q+1})}\  \too\  
  \frac{\oEHC^{-,\wedge}(A)}{\ker_{\oHH(A)}(B)}\ 
  \too\  \prod_{m \geq 1} u^m \cdot \ker_{\oHD(A)}((S_1 S_2)^{m+1})\ \to 0,
\end{equation}
where we note that the $q$ became $q+1$ in the first factor since the
periodicity operator $S$ is acting on $\oHC(A)$ rather than on
$\oHD(A)=\im(S_1)$.

\begin{proof}[Proof of Claim \ref{e:ehc-neg-seq}]
  We use the spectral sequence for the bicomplex $(\ohDRt \dspu,\ u
  \sd, \si_t)$, taking homology first with respect to $u \sd$.  This
  yields the first page,
\[
(u \cdot \oHD(A) \dspu\ \oplus\ \ker_{\ohDRt }(\sd),\ 0 \oplus \si_t).
\]
Using the isomorphism $\sd: (\ohDRt)^+ / \sd (\ohDRt)^+ \iso \sd
(\ohDRt)^+$, which commutes with $\si_t$, the second page becomes
\[
u \cdot \oHD(A) \dspu\ \oplus\ H^{> 0}(\ohDRt / \sd \ohDRt,
\si_t)/(\sdi \si_t \oHD(A))\ \oplus\ \ker_{\oHH(A)}(\sd).
\]
The terms on the left and right of the exact sequence of Claim
\ref{e:ehc-neg-seq} then appear on the $(m+2)$-nd and $(q+1)$-st
pages.  One sees that the spectral sequence converges to this, i.e.,
there are no other nontrivial differentials that appear.  Therefore,
the associated graded (where by this we mean an infinite product of
homogeneous components) of $\oEHC^{-,\wedge}(A)$ with respect to the
$t$-adic filtration is the above, and one can check that the maps in
the exact sequence of Claim \ref{e:ehc-neg-seq} induce the required
isomorphism on the level of associated graded, so the sequence there
must be exact (with injective first arrow and surjective third arrow).
\end{proof}
\end{remark}
\begin{remark}
  According to \eqref{e:ehc-cyc2}, Proposition \ref{p:ext-per-cyc},
  and \eqref{e:ehc-neg-seq2}, the extended versions of all cyclic
  homology groups $\oEHC(A), \oEHC^{\per,\wedge}(A)$, and
  $\oEHC^{-,\wedge}(A)$ are expressible in terms of the
  \emph{original} Connes sequence \eqref{e:connes-seq}, and do not
  require $\oHC^-(A)$ or $\oHC^\per(A)$, which use infinite series in
  their definition.  In other words, taking the extended version
  breaks the complex up into a direct product of complexes only
  involving finite sums, i.e., elements of $\oDR_t[u,u^{-1}]$ (and via
  Theorem \ref{t:heq}, elements of $\oOm[u,u^{-1}]$), owing to the
  internal grading. As a consequence, the natural morphisms we will
  describe in \S \ref{s:eqcoh} from the various flavors of cyclic
  homology of $A$ to the (equivariant) cohomology of the
  representation varieties of $A$ all factor through (infinite
  products of) subquotients of $\oHC(A)$ and $\oHH(A)$.
\end{remark}

\subsection{} The result below describes the maps induced by $\EE $ on
the versions of cyclic homology.

Let $\pi':\ \oEHC(A)\onto \oHD(A)\o u^{-1}\bk[u^{-1}] \subseteq
\oHD(A) \o R$ be the map given by $f = \sum_{m \geq 0} f_m u^{-m} \mto
\sum_{m \geq 1} f_m u^{-m} \mmod (t)$, cf.~Proposition
\ref{p:ext-cyc}.  Let $\oEHC^\wedge(A) = H(\ohDRt \o R, \si_t-u\sd)
\cong \oHC(A)[\![t]\!] \oplus u^{-1} \oHD(A)[u^{-1}]$, and define
$\pi'$ in the same way on it.

\begin{proposition}\label{p:Timh}
  \vi The composite map $\dis\ \oHC(A)\ \stackrel{\EE}\too\
  \oEHC^\wedge(A)\ \stackrel{\pi'}\too \oHD(A)\o R$ can be expressed
  as $\pi'\ccirc \EE\ =\ \sum_{m \leq 1}\ u^{-m} \cdot S_1 S^{m-1}$.

  \vii The induced map $\ \oHC^\per(A) \to \oEHC^{\per,\wedge}(A)\
  \cong\ \oHD(A)\dpu\ $ is $\ \sum_{m \in \bZ} u^{-m} \cdot p_1 u^m$.

  \viii The induced map $\oHC^-(A)\ \to\ \oEHC^{-,\wedge}(A)$, in
  terms of \eqref{e:ehc-neg-seq}, is the sum of $I^-$ to
  $\ker_{\oHH(A)}(\sd)$ with the map which is $\sum_{m \geq 1} u^m
  (p_1 u^{-m} p^-)$ projected to the subspace of $\prod_{m \geq 1}\
  u^m \cdot \ker_{\oHD(A)}((\sdi \si_t)^{m+1})$, and, restricted to
  the kernel of this map, maps to the zero fiber $\prod_{q \geq 1}\
  H(\oDR_t / \sd \oDR_t, \si_t)^{\hdot,q} / \im((\sdi \si_t)^q)$ by
  $\sum_{q \geq 1} \EE^q\ccirc p\ccirc p^-$.
\end{proposition}
\noindent
The proof is a direct consequence of formula \eqref{ddtH},
Propositions \ref{p:ext-cyc} and \ref{p:ext-per-cyc} and Remark
~\ref{r:ext-neg-cyc}.

\subsection{Extended version of the Connes exact sequence}
There is an extended analogue of diagram \eqref{e:bigdiag}.  One can
obtain this diagram by applying $\EE$ to the diagram of short exact
sequences of complexes that induce \eqref{e:bigdiag} (in the usual
proof). This yields
\begin{equation*}
  \xymatrix{
    0 \ar[r] & \ohDRt \dspu \ar[r]^{\cdot u} \ar@{=}[d] & \ohDRt \dspu 
    \ar[r] \ar[d] & \ohDRt  \ar[r] \ar[d] & 0 \\
    0 \ar[r] & \ohDRt \dspu \ar[r]^{\cdot u} \ar[d] & \ohDRt \dpu \ar[r] 
    \ar[d]^{\cdot u^{-1}} & \ohDRt  \otimes R \ar[r] \ar@{=}[d] & 0 \\
    0 \ar[r] & \ohDRt  \ar[r] & \ohDRt  \otimes R \ar[r]^{\cdot u} & 
    \ohDRt  \otimes R \ar[r] & 0.
  }
\end{equation*}
This induces the diagram \eqref{e:bigdiag} of short exact sequences,
with $\EE$ put in front of every nonzero term and $t$-adic completions
taken everywhere, where $\oEHH$, $\oEHH'$, $\oEHD$, and $\oEHD'$ are
given by:

\begin{gather*}
  \oEHH(A) := H(\oDR_t, \si_t), \qquad \oEHH'(A) := \bigoplus_{q \geq
    0}
  [(\sdi \si_t)^q(\ker_{\oDR}(\sd)) \cap \ker(\si_t)]\ \subseteq\ \oEHH(A), \\
  \oEHD(A) := \oHD(A)[u^{-1}]\ \oplus\ \bigoplus_{q \geq 1} \ker(\sd:
  H(\oDR_t / \sd \oDR_t, \si_t) \to H(\oDR_t, \si_t)), \\ \oEHD'(A) :=
  \oHD(A)[u^{-1}] \ \oplus\ \bigoplus_{q \geq 1} (\sdi
  \si_t)^q(\oHD(A)) \subseteq \oEHD(A),
\end{gather*}
and the $t$-adically completed versions are obtained by replacing
$\oDR_t$ with $\ohDRt$, and the direct sums above by direct products.
Then, $\EE$ maps \eqref{e:bigdiag} commutatively onto the resulting
diagram.

\begin{remark}
  By Theorem \ref{t:heq}, one has isomorphisms $\EE^q: \oHH(A) \iso
  \oEHH(A)^{\hdot,q},\ q \geq 0$. We deduce that $\oEHH'(A)^{\hdot,0}
  \cong \ker_{\oHH(A)}(B) = I^{-1}(\im S)$, where $S: \oHC(A) \to
  \oHC(A)$ is the periodicity operator.  Then one obtains:
\begin{gather*}
  \oEHH(A)\ \cong\ \oHH(A)[t] \quad \supseteq \quad
  \oEHH'(A)\  \cong\  \oplus_{q \geq 0}\ I^{-1}(\im(S^{q+1})) \cdot t^q, \\
  \oEHD(A) \cong \oHD(A)[u^{-1}] \bplus\, t \oHD(A)[t] \supseteq
  \oEHD'(A) \cong \oHD(A)[u^{-1}] \bplus \left(\oplus_{q \geq 1}
    \im(S_1 S_2)^{q} \cdot t^q\right).
\end{gather*}
The $t$-adically completed versions of $\oEHH(A), \oEHH'(A)$,
$\oEHD(A)$, and $\oEHD'(A)$ are all obtained by $t$-adically
completing the above formulas (replacing sums by products and
polynomials in $t$ by power series in $t$).
\end{remark}
\begin{remark}
  One can explicitly compute the extended, $t$-adically completed
  version of \eqref{e:bigdiag} in terms of the preceding remark along
  with \eqref{e:ehc-cyc2}, Proposition \ref{p:ext-per-cyc}, and
  \eqref{e:ehc-neg-seq2}, and verify that the rows are exact using
  only exactness of \eqref{e:bigdiag}. We describe just the horizontal
  maps in these terms:

\emph{The bottom row:}

$ES_2$: This is the identity on $u^{-1}\oHD(A)[u^{-1}]$ and the
inclusion $S_2$ on $\oHD(A)[\![t]\!] \subseteq \oHC(A)[\![t]\!]$,
preserving degree in $t$.

$EB: \oEHC^\wedge(A) \cong u^{-1} \oHD(A)[u^{-1}] \oplus
\oHC(A)[\![t]\!] \to \oEHH^\wedge(A) \cong \oHH(A)[\![t]\!]$ is zero
on the first factor of the source and $B$ on the second factor.

$EI: \oEHH^\wedge(A) \cong \oHH(A)[\![t]\!] \to \oEHC^\wedge(A) =
u^{-1} \oHD(A)[u^{-1}] \oplus \oHC(A)[\![t]\!]$ is zero on the first
factor of the target, and $I$ on the second factor of the target.

$ES_1$ is the tautological map
(preserving the degree in $t$).

\emph{The middle row:}

$EB': \oEHC^\wedge(A) \to \oEHC^{-,\wedge}(A)$: in terms of
\eqref{e:ehc-neg-seq}, this kills $u^{-1}\oHD[u^{-1}]$, maps by $B$ on
$\oHC(A)$ in degree $q=0$ of $t$, and for positive degree $q \geq 1$
in $t$, it is the projection $\oHC(A) \onto \oHC/\im S^{q+1} \subseteq
\oEHC^{-,\wedge}(A)$ by \eqref{e:ehc-neg-seq}.

$Ep^-: \oEHC^{-,\wedge}(A) \to \oEHC^{\per,\wedge}(A)$ projects onto
$\ker_{\oHH(A)}(B) \oplus \prod_{m \geq 1} u^m \ker_{\oHD(A)}(S_1
S_2)^{m+1}$, and then applies $I$ to the first factor and the natural
inclusion on the second factor into $u \oHD\dspu$, preserving degree
in $u$.

$Ep_1$ is the identity on $\oHD(A)[u^{-1}]$ and sends $u^m \oHD(A)$ to
the factor for $q=m$ by the map $S^{m} S_2 u^{-m} t^m$.

$E p_2$ is the tautological inclusion (preserving degree in $t$).

\emph{The first row:}

The first map is the tautological inclusion.

$EB^-$ maps $\oHH(A)$ via $B$ to $\ker_{\oHH(A)}(B)$ in degree $q=0$
of $t$, and in higher degrees $q \geq 1$ maps $\oHH(A)$ to $\oHC(A)$
via $I$ and projects to $\oHC(A)/\im S^{q+1}$.

$ES^-$ applies $t \cdot S$ to $\oHC(A) t^q / \im S^{q+1}$ for $q \geq
1$, increasing degree in $t$ by one. Modulo these factors (the kernel
of the sequence \eqref{e:ehc-neg-seq}), it multiplies by $u$: on the
quotient of \eqref{e:ehc-neg-seq} these are the inclusions
$\ker_{\oHD(A)}(S_1 S_2)^{m+1} \subseteq \ker_{\oHD(A)}(S_1
S_2)^{m+2}$, and on $\ker_{\oHH(A)}(B)$ this is the natural map
$\ker_{\oHH(A)}(B) \to \ker_{\oHD(A)}(S_1 S_2)$ (this follows because
$\ker_{\oHH(A)}(B) = \ker_{\oDR}(\sd) \cap \ker_{\oDR}(\si_t)$, so
projecting mod $\sd \oDR$ we land in $\ker_{\oHD(A)}(S_1 S_2)
\subseteq \ker_{\oHD(A)}(S_1 S_2)^2$).

$EI^-$ is the identity on $\ker_{\oHH(A)}(B) = I^{-1}\im(S)$ and
applies $B$ to the kernel of \eqref{e:ehc-neg-seq}, which gives a
well-defined map to $\oHH(A)$ since $B \ccirc S^q = 0$ for $q \geq
1$. The image of the latter map is the subspace $\ker(I)=\im(B)$ of
$\oEHH'^\wedge(A)=\bigoplus_{q \geq 0} I^{-1} \im S^{q+1}$. Then, on
the quotient \eqref{e:ehc-neg-seq}, $EI^-$ applies $(S_1S_2)^{m}
u^{-m}$ to $u^m \ker(S_1 S_2)^{m+1}$ for $m \geq 1$, which gives a
well-defined map to $\oHH(A) / \ker(I) \cong \im(I) \subseteq
\oHC(A)$, with image equal to $\im_{\oHD(A)}(S_1 S_2)^{m} \cap
\ker_{\oHD(A)}(S_1 S_2)$, which is $\im(S^{m+1}) \cap \ker(S) =
\im(S^{m+1}) \cap \im(I)$ as a subspace of $\im(I) \subseteq \oHC(A)$.
\end{remark}

\section{The representation functor and equivariant Deligne
  cohomology}\label{s:eqcoh}

\subsection{Equivariant Deligne cohomology}
Let $X$ be an affine\footnote{It is possible to generalize everything
  below to the case of not necessarily affine varieties by replacing
  the space of global differential forms on $X$ by the sheaf of
  differential forms.} variety equipped with an action of a connected
reductive algebraic group $G$ with Lie algebra $\mfg$.  Let $\act:
\mfg \to \Vect(X)$ be the infinitesimal action and let $ \Om^1_X$ be
the space of K\"ahler differentials on $X$.  Thus, the algebra
$\Om_X=\wedge_{\bk[X]}(\Om^1_X)$ of algebraic differential forms on
$X$ acquires a natural structure of $\g$-module.  One also has a
natural $\g$-action on $\bk[\g]$, a polynomial algebra, induced by the
adjoint $\g$-action on $\g$ itself.  We let $\g$ act diagonally on the
algebra $\Om_X[\mfg]:=\Om_X \otimes \bk[\mfg]$.

Recall that the Cartan model of the equivariant (algebraic) de Rham
complex of $X$ is defined ~as
\begin{equation}\label{cartan}
\bigl((\Om_X[\mfg])^\mfg,\ \sd_X - i_\mfg\bigr).
\end{equation}
where $\sd_X$ is the de Rham differential on $X$ and $i_\mfg$ is the
equivariant differential, given by
\[
i_\mfg(\sd f)(x) = \act(x)(f), \quad i_\mfg|_{\Om^0_X} = 0 =
i_{\mfg}|_{\bk[\mfg]}.
\]

By analogy with cyclic homology, we introduce three new versions of
the Cartan model, all equipped with the differential $i_\g-u \sd_X$,
where $u$ is an extra parameter:
\begin{equation}\label{equivu}
  \CC^\per_{\mfg}(X)\ :=\ (\Om_X[\mfg])^\mfg\dpu,
  \quad
  \CC_\g(X)\ :=\ (\Om_X[\mfg])^\mfg\o R,
  \quad
  \CC^-_{\mfg}(X)\ :=\ (\Om_X[\mfg])^\mfg\su.
\end{equation}

\begin{remark}\label{per}
  The complex $(\CC^\per_{\mfg}(X),\ i_\g-u \sd_X)$ is isomorphic, up
  to sign in the differential, to the standard version \eqref{cartan}
  tensored by $\bk\bbu$, via the map $\mfg^* \iso u^{-1} \mfg^*, g
  \mapsto u^{-1} g$.
\end{remark}

We define the {\em homological grading} $\CC_\g(X)=
\oplus_n\CC_\g^n(X)$ by placing the space
$(\Om^p_X\o\bk[\mfg])^\mfg\cdot u^{-r}$ in degree $n=p+2r$, for every
$p$ and $r$.  Thus, the tensor factor $\bk[\g]$, is assigned
homological degree {\em zero}, and $u$ has homological degree
$|u|=-2$.  Each of the differentials $u \sd_X$ and $i_\mfg$, and hence
also $i_\g-u \sd_X$, has homological degree $(-1)$.  We define {\em
  equivariant Deligne cohomology}, which will serve as the natural
receptacle of cyclic homology under the representation functor, by
\[ \HDe^\hdot_{\!\g}(X) := H(\CC^\hdot_\mfg(X),\ i_\g-u \sd).\]

Periodic and negative versions are defined in an obvious way.

Similarly to \S\ref{ext_sec}, we also introduce an {\em internal
  grading} $\CC_\mfg(X)=\oplus_{k\in\Z}\ \CC_\mfg(X)_k$, by assigning
the space $(\Om^p_X \o \bk^q[\g])^\mfg\cdot u^{-r}$ internal degree
$p+q+r$, for every $p,q,$ and $r$.  Here, we use the notation
$\bk^q[\g]$ for the space of degree $q$ homogeneous polynomials on
$\g$.  Each of the differentials $u \sd_X$ and $i_\g$ preserves the
internal grading. Therefore, the complex $(\CC_\mfg(X),\ i_\g-u\sd_X)$
breaks up into a direct sum over all $k\in\Z$ of the subcomplexes
$(\CC_\mfg(X)_k,\ i_\g-u \sd_X)$.

For $\ell\geq 0$ and $m\in \Z$, we define the {\em weight} $m$
component of $\HDe_{\!\g}^\ell(X)$, cf.~\S\ref{ext_sec}, by
\[\HDe_{\!\g}^\ell(X,m)\ :=\ H^\ell(\CC_\mfg(X)_{\ell-m},\
i_\g-u\sd_X).\]

The natural multiplication in the algebra $\Om_X\bbu\o\bk[\g]$ induces
a canonical $(\bk[\g]^\g)[u]$-module structure on the equivariant
Deligne cohomology. The module structure and the gradings are related
by
\begin{equation}\label{periodicity}(\bk^q[\g]^\g)\, u^r\
  \,\mbox{$\bigotimes$}\,\ \HDe_{\!\g}^\ell(X,m)
  \ \too\ \HDe_{\!\g}^{\ell-2r}(X,m-q-r),\qquad\forall q,r\geq0.
\end{equation}

There is also an algebra structure on the equivariant Deligne
cohomology, similar to the one on the ordinary Deligne cohomology.  In
the equivariant case, the product is defined by the formula $\omega
\cdot \omega' = \omega \cup u^{-1}(i_\g - u \sd_X)\omega'$.  Here
$u^{-1}(i_\g - u\sd_X) \omega'$ should be computed by first lifting
$\omega'$ to $(\Om_X[\mfg])^\mfg\o \bk(\!(u)\!)$, then after
performing the operation, projecting back to $(\Om_X[\mfg])^\mfg \o
R$.  Compare this with \cite[\S 3.6.2]{L} regarding usual Deligne
cohomology.  (Note that, when $\omega'$ is a cocycle, i.e.,
$(i_\g-u\sd_X)\omega'=0$ in $\CC^\hdot_\mfg(X)$, then the above
formula can be rewritten as $\omega \cup -\sd_X \omega'_0$, where
$\omega'_0$ is the constant term in $u$ of $\omega'$.)  The algebra
structure will not play a role below.

Equivariant Deligne cohomology is functorial with respect to
$G$-equivariant maps.

\begin{example}\label{e:triv}
  If $\mfg$ acts trivially on $X$, we obtain $\dis \CC_\mfg(X) \cong
  (\Om_X \otimes R, u\sd) \otimes \bk[\mfg]^{\mfg}$.  The cohomology
  of the first term on the RHS is
\[
H(\Om_X \otimes R,\ u\sd)\ \cong\ \Om_X/\sd \Om_X \en \bplus\en
\bigl(H_{\DR}(X) \otimes u^{-1} \!\cdot \bk[u^{-1}]\bigr).
\]

We find that the group
$\HDe^\ell_{\!\g}(X,m)$ vanishes for $m>\il$, and 
$$\HDe^\ell_{\!\g}(X,m)=
\begin{cases} 
  \underset{{}}{\left({\bigoplus\limits_{1\leq j \leq [\ell/2]}}_{_{}}
      H^{\ell-2j}_{\DR}(X)\o\bk^{j-m}[\g]^\g\right)}\en\bplus\en
  \left(\Om_X^\ell/ \sd \Om_X^{\ell-1}\o\bk^{-m}[\g]^\g\right),&\quad m\leq0\\
  \hskip 3pt{\bigoplus\limits_{m\leq j \leq[\ell/2]}
    H^{\ell-2j}_{\DR}(X) \o\bk^{j-m}[\g]^\g}^{^{}}, & \quad
  0<m\leq\il.
\end{cases}
$$

In total degree $\ell$, we therefore obtain $\HDe^\ell_{\!\g}(X) \
\cong\ \HDe^\ell(X)\o \bk[\mfg]^{\mfg}$, where we write
$$
\HDe^\ell(X)\ :=\ \bigoplus_{1\leq j \leq [\ell/2]}
H^{\ell-2j}_{\DR}(X) \en\bplus\en\Om_X^\ell / \sd \Om_X^{\ell-1}
$$
for the even or odd-degree part of the usual Deligne cohomology of $X$
for the integer $\ell+1$ shifted down in degree by one. The Deligne
cohomology is denoted ${{}^{\mathcal D}\tilde{\mathrm H}}(X,
\bZ(n+1))$ in ~\cite[\S 3.6]{L}.
\end{example}

\begin{example}\label{e:factor} Let $p: X\to Y$ be a principal 
  $G$-bundle on an affine variety $Y$, the latter being viewed as a
  $G$-variety with the trivial $G$-action. Then, analogously to the
  case of usual equivariant cohomology, there is a canonical
  isomorphism
  \[\HDe^\hdot_{\!\g}(X)\ \cong\ \HDe^\hdot(Y).\]

  To prove this, one observes first that a $G$-bundle on an affine
  variety admits an algebraic connection.  Using a connection, one
  shows that the pull-back morphism $p^*:\ \Om_Y \to \Om_X$ induces an
  isomorphism $\bk[X]\o_{\bk[Y]}\Om_Y \iso H(\Om_X[\mfg],\ i_\g)$.
  Taking $G$-invariants on each side and using that $G$ is connected
  and reductive, we deduce an isomorphism
  $\dis\Om_Y=(\bk[X]\o_{\bk[Y]}\Om_Y)^\g \iso H((\Om_X[\mfg])^\g,\
  i_\g)$.  The required isomorphism $\HDe^\hdot_{\!\g}(X)\ \cong\
  \HDe^\hdot(Y)$ then follows from the spectral sequence of a double
  complex, by taking homology of the differential $i_\g$ first.
\end{example}

\subsection{The representation functor} Associated with an algebra $A$
and $n\geq1$, there is an affine scheme $\Rep_n A$ that parametrizes
$n$-dimensional representations of $A$. Its closed points are algebra
homomorphisms $\rho:\ A\to \Mat_n(\bk),\ a\mto\rho(a)$.  The scheme
$\Rep_n A$ comes equipped with a natural action of $\GL_n$, the
general linear group, by base change transformations.  We write
$\gl_n={\mathrm{Lie}}\GL_n$ for the corresponding Lie algebra.

Below, we will use simplified notation: $\Rep_n =\Rep_n A,\
\DR_t=\DR_t A$, etc.  The connection between $\DR_t$ and the
$\GL_n$-equivariant de Rham complex of $\Rep_n$ is through the
evaluation map, which is a `tautological' homomorphism of algebras,
\[
\ev: A \to (\bk[\Rep_n] \otimes \Mat_n(\bk))^{\gl_n}, \quad a \mapsto
(\rho \mapsto \rho(a)).
\]

The above can be defined scheme-theoretically by extending this
functorially to representations with coefficients in arbitrary
$\bk$-algebras; see, e.g., \cite[\S 12]{Glect}.

The evaluation map extends uniquely to a homomorphism of dg algebras,
\[
\ev: \Om A \to (\Om_{\Rep_n} \otimes \Mat_n(\bk))^{\gl_n}, \quad
\ev(\sd\!a) = \sd \ev(a).
\]
Finally, to extend this to $\Om_t A$, let $\ev(t): \gl_n \to
\Mat_n(\bk)$ be the identity map, viewed as a linear function on
$\gl_n$ valued in $\Mat_n(\bk)$, i.e., as an element of the algebra
$\bk[\gl_n] \otimes \Mat_n(\bk)$. Put together, we obtain a dg algebra
map $\dis\ \ev: \Om_t A \to (\Om_{\Rep_n}[\gl_n] \otimes
\Mat_n(\bk))^{\gl_n}$.  Further, taking a trace gives a chain of maps
\[
\xymatrix{ \Om_t A\ \ar[r]^<>(0.5){\ev}& \ (\Om_{\Rep_n}[\gl_n]
  \otimes \Mat_n(\bk))^{\gl_n}\ \ar[rr]^<>(0.5){\Id\o\text{trace}}&& \
  (\Om_{\Rep_n}[\gl_n] \o\bk)^{\gl_n}\ =\
  (\Om_{\Rep_n}[\gl_n])^{\gl_n}.  }
\]
The resulting composite map descends to a linear map
\[\tr \ccirc \ev:\
\DR_t \too (\Om_{\Rep_n}[\gl_n])^{\gl_n}.\]

Then, it is not difficult to verify

\begin{theorem}[\cite{GScyc}, Theorem 6.2.5] \label{t:rep} For all $n
  \geq 1$, the map $\tr \ccirc \ev$ induces a map of bicomplexes
$$
\tr \ccirc \ev:\ (\DR_t,\ \sd, \si_t) \too
\bigl((\Om_{\Rep_n}[\gl_n])^{\gl_n},\ \sd_{\Rep_n}, i_{\gl_n}\bigr).
$$
\end{theorem}
\subsection{From cyclic to equivariant Deligne cohomology}
Let us define reduced versions of the previous equivariant complexes
and cohomology groups of $X$ by replacing $\Om_X$ by $\oOm_X = \Om_X /
\bk$ (where $\bk \subseteq \Om_X$ are the scalars in degree zero), and
leaving everything else the same.  We will put lines over everything
to denote the reduced versions.

By Theorem \ref{t:rep} and construction, we deduce that the map $\tr
\ccirc \ev$ induces, for all $n \geq 1$, a canonical morphism of
$\bk[\gl_n]^{\gl_n}[u]$-modules:
\begin{equation}\label{hmap}
  \oEHC(A)=H(\oDR_t \otimes R,\ \si_t-u\sd)\ 
  \too\ \oHDe_{\!\gl_n}(\Rep_n)=H\bigl((\oOm_{\Rep_n}[\gl_n])^{\gl_n} 
  \otimes R,\
  i_{\gl_n}-u \sd_{\Rep_n}\bigr).
\end{equation}

The main result of this section is an analogue of 
the above for the usual cyclic homology:

\begin{theorem}\label{Psi} \vi For every $n \geq 1$,
one has 
  canonical maps given by the
 composition
 \[
 \Psi_\hdot^m:\ \xymatrix{ \oHC_\idot(A) \
   \ar[rr]^<>(0.5){\pr_{\hdot-m}\ccirc \EE}&& \ \oEHC_\idot(A,m)\
   \ar[rr]^<>(0.5){\tr \ccirc \ev}&& \ \oHDe^\hdot_{\!\gl_n}(\Rep_n,
   m),}\quad m\in \Z.
\]

\vii For $m=0$ and every $\ell\geq0$, the composite map
\[\Psi'_\ell:\
\xymatrix{\ker(\si: \oDR^\ell\!/\sd\oDR^{\ell-1}\!\! \to
  \oOm^{\ell-1}\!/\sd\oOm^{\ell-2})\
  \ar@{=}[rr]^<>(0.5){\eqref{hcdef}} &&\ \oHC_\ell(A)\
  \ar[r]^<>(0.5){\Psi_{\ell}^0} & \oHDe^\hdot_{\!\gl_n}(\Rep_n, 0)}
\]
is given by the explicit formula 
\begin{equation}\label{psi'}
  \Psi'_\ell( f)=\mbox{$\frac{1}{\ell!}$}\sum_{0\leq j\leq [\ell/2]}\ 
  u^{-j}\cdot\tr \ccirc
  \ev\ccirc(\sdi\si_t)^j( f).
\end{equation}
\end{theorem}
\begin{remark}
  \textsf{(1)}\en In the above formula, $\tr \ccirc
  \ev\ccirc(\sdi\si_t)^j(f)\in \oOm^{\ell-2j}_{\Rep_n}\o
  \bk^j[\gl_n]$. Thus, the sum in the RHS of \eqref{psi'} {\em a
    priori} lives in
$$\left({\bplus_{1\leq j \leq\lfloor \ell/2 \rfloor}}\
  u^{-j}\cdot\oOm^{\ell-2j}_{\Rep_n}\o\bk^j[\gl_n]\en\bplus\en
  \oOm_{\Rep_n}^\ell/\sd \oOm_{\Rep_n}^{\ell-1}\right)^{\gl_n}.
$$
One checks further that this sum is, in fact, annihilated by the
differential $i_{\gl_n}-u\sd_{\Rep_n}$, and hence gives a well defined
class in $\oHDe^\hdot_{\!\gl_n}(\Rep_n, 0)$.

\textsf{(2)}\en Note that formula \eqref{psi'}, unlike the definition
of the map $\Psi^m_\idot$ in the general case of an arbitrary $m$,
does not involve the Tsygan map $\EE$.
\end{remark}

\begin{proof}[Proof of Theorem \ref{Psi}] Part (i) follows from
  \eqref{hmap}, by separating individual homogeneous components and
  using isomorphism \eqref{qq}.  To prove (ii), we use diagram
  \eqref{pidiag} and Proposition \ref{inverse}.  In more detail, we
  consider the sequence of maps
\begin{gather}
  \xymatrix{\ker(\si: \oDR^\ell\!/\sd\! \oDR^{\ell-1}\to
    \oOm^{\ell-1}\!/\sd\!\oOm^{\ell-2})
    \ar@{=}[r]^<>(0.5){\eqref{hcdef}}& \ H^\ell(\oOm\o R,\
    \msb-u\msB)\ \ar@{=}[r]^<>(0.5){\eqref{cqmap}}& \
    H^\ell(\oOm/\msB\oOm,\msb)}\label{chain2}
  \\
  \qquad\hskip20mm\xymatrix{ \ar[r]^<>(0.5){\EE^0}_<>(0.5){\cong}& \
    H(\oDR_t/\sd\oDR_t, \si_t)^{\ell,0} \ar[r]^<>(0.5){\eta}&\
    H^\ell\bigl((\oDR_t,\ \si_t-u\sd)_{\ell}\bigr)\ \ar[r]^<>(0.5){\tr
      \ccirc \ev}& \ \oHDe^\ell_{\!\gl_n}(\Rep_n, 0),}\nonumber
\end{gather}
where the map $\eta$
is given by the formula $\eta=(\Id-u\inv\sdi\si_t)\inv=
\sum_j u^{-j}\cdot (\sdi\si_t)^j$.

Observe first that, for $f\in\oOm^\ell$, the component of $\EE(f)$ of
degree zero in $t$ equals $\frac{1}{\ell!}\cdot f\cc$ (one can see
this either directly from formula \eqref{e:Tdefn} or from Lemma
\ref{n!}).  Therefore, the restriction of the map $\EE^0$ to
$H^\ell(\oOm/\msB\oOm,\msb)$ is essentially the map
$f\mto\frac{1}{\ell!}\cdot f\cc$.  It follows that the composition of
the first four isomorphisms in \eqref{chain2} is the isomorphism from
\eqref{orem} times $\frac{1}{\ell!}$.  We deduce that the composite of
all the maps in \eqref{chain2} equals the map given by formula
\eqref{psi'}.

On the other hand, the map $\eta$ provides, thanks to Proposition
\ref{inverse}, an inverse to the isomorphism $\oEHC_\ell(A,0) \iso
H(\oDR_t/\sd\oDR_t, \si_t)^{\ell,0}$, in \eqref{qq}.  Therefore, from
the commutativity of diagram \eqref{pidiag} we deduce that the
composite of the second, third, and fourth maps in \eqref{chain2} is
equal to the map $\dis \pr_\ell\ccirc\EE:\ H^\ell(\oOm\o R,\
\msb-u\msB)\to \oEHC_\ell(A,0)$. We conclude that the composite of the
last four maps in \eqref{chain2} equals the map $\Psi^0_\ell$,
completing the proof.
\end{proof}

\vskip3pt One can also restrict the map \eqref{hmap} to the Karoubi-de
Rham cohomology $\oHD(A) \subseteq \oHC(A)$ and $\oEHD(A) \subseteq
\oEHC(A)$. This will land in $\oHDe_{\cl,\gl_n}(\Rep_n(A))$, the
cohomology of the complex
$$
\oCC_{\cl,\gl_n}(\Rep_n) := \bigl(\oOm^{\gl_n}_{\cl, \Rep_n}\en \bplus\en
(\oOm_{\Rep_n} 
\otimes \bk[\gl_n]_+)^{\gl_n}\o R,\;\ \si_{\gl_n}-u\sd_{\Rep_n}\bigr),
$$
where $\oOm^{\gl_n}_{\cl,{\Rep_n}}$ denotes the subspace of closed
forms and $\bk[\gl_n]_+\sset\bk[\gl_n]$ denotes the augmentation
ideal.  Let $\oCC_{\cl,\gl_n}^\wedge({\Rep_n})$ be the corresponding
$\bk[\gl_n]_+$-adic completion.
\begin{corollary}
  By restriction, we obtain maps
\[
\oEHD(A) \to H(\oCC_{\cl,\gl_n}({\Rep_n})),
\quad\text{\emph{and}}\quad \oHD(A) \to
H(\oCC_{\cl,\gl_n}^\wedge({\Rep_n})).\eqno\Box
\]
\end{corollary}
\begin{remark}
  There are also `periodic' and `negative' versions of the above
  constructions, which produce maps
\[
\oEHC^{\per,\wedge}(A) \to H(\oCC^{\per,\wedge}_{\gl_n}(\Rep_n A))
\cong \overline{H}_{\!\gl_n}(\Rep_n A)\dpu, \quad \oEHC^{-,\wedge}(A)
\to H(\oCC^{-,\wedge}_{\gl_n}(\Rep_n A)).
\]

Composing with $\EE$, we obtain maps
\[
\oHC^\per(A) \to H(\oCC_{\gl_n}^{\per,\wedge}(\Rep_n A))\dpu, \qquad
\oHC^-(A) \to H(\oCC^{-,\wedge}_{\gl_n}(\Rep_n A)).
\]

Moreover, we obtain commutative diagrams of commutative squares with
exact rows, from \eqref{e:bigdiag} to the version with $\EE$ in front
of all groups (and with $t$-adic completions), to the commutative
square connecting the $\bk[\gl_n]_+$-adic completed versions of the
above equivariant cohomology groups (and versions one can similarly
define analogous to $\oHH', \oHH$, and $\oHD'$; note that the
equivariant analogue of $\oHD$, $\oHDe_{\cl,\mfg}({\Rep_n})$, was
defined above).
\end{remark}

\subsection{Special cases}
\underline{{\em The case of $n=1$}}:\quad If $n=1$, then $\GL_1$ is a
1-dimensional torus, and $\bk[\gl_1] \cong \bk[t]$. Furthermore, since
the algebra $\Mat_1(\bk)$ is commutative, the evaluation map for
$\oOm_tA$ may be factored as a composition
\[
\xymatrix{ \oOm_tA\ \ar@{->>}[rr]^<>(0.5){f\mto f\cc}&&\ \oDR_t\
  \ar[r]^<>(0.5){\ab_t}&\ \oDR[t]\ \ar[rr]^<>(0.5){\ev\o\Id_{\bk[t]}}
  &&\ \oOm^\hdot_{\Rep_1A}[t]\o\Mat_1(\bk),}
\]
where $\ab_t$ is the map that makes $t$ a central variable, cf.~Lemma
\ref{n!}. Given $f\in \oOm A$, we will use simplified notation $\bar
f:=\tr\ccirc\ev(f\cc)\in \oOm_{\Rep_1A}$.  Thus, using the formula of
Lemma \ref{n!} and the fact that the trace map yields an {\em algebra}
isomorphism $\tr: \Mat_1(\bk)\iso\bk$, we deduce
\begin{equation}\label{bbb}\tr\ccirc\ev\ccirc\EE(f)=
  \mbox{$\frac{1}{k!}$}\cdot
  \bar f\,\exp(t),\qquad \forall\ f\in\oOm^kA.
\end{equation}

One can use the above formula to obtain an explicit description of the
map of Theorem \ref{Psi} in the special case that $n=1$ and $m=0$.  To
this end, observe first that the group $\GL_1$ acts trivially on the
scheme $\Rep_1 A$.  Therefore, we are in the setting of Example
\ref{e:triv}.  According to the isomorphism given there, in the
special case where $\g=\gl_1$ and $m=0$, Theorem \ref{Psi} yields a
map
\begin{equation}\label{ququ}
  \Psi_\ell^0:\ \oHC_\ell(A)\ \to\
  \oHDe^\ell(\Rep_1,0)=\oOm_{\Rep_1 A}^\ell / \sd \oOm_{\Rep_1 A}^{\ell-1}
  \en
  \bplus\en
  \left(\oplus_{j \geq 1}\ \overline{\!H\!}\,^{\ell-2j}
    (\Rep_1 A)\,t^j\,u^{-j}\right)
\end{equation}

To compute this map explicitly, we plug formula \eqref{bbb} for the
map $\EE$ into \eqref{hmap} and separate the relevant homogeneous
components. Then, one finds that the map \eqref{ququ} is given, for
every $\sum_{0\leq j\leq [\ell/2]}\ f_{\ell-2j}\, u^{-j}$ representing
a class in $\oHC_\ell(A)=H^\ell(\oOm A\o R,\ \msb-u\msB)$, by the
formula
\begin{equation}\label{ddd}
\sum_{0\leq j\leq [\ell/2]}\ f_{\ell-2j}\, u^{-j}\ \mto\
\sum_{0\leq j\leq [\ell/2]}\ \frac{1}{(\ell-2j)!(\ell-2j)!}\cdot
\bar f_{\ell-2j}\, t^j\, u^{-j},\qquad 
f_{\ell-2j}\in \oOm^{\ell-2j}A.
\end{equation}

Next, let $A_{{\rm{ab}}}:=A/([A,A])$ be the {\em abelianization} of
$A$, a quotient of the algebra $A$ by the two-sided ideal generated by
the set $[A,A]$.  A well known construction in the theory of cyclic
homology of {\em commutative} algebras produces a map
$$
\oHC_\ell(A_\abb)=H^\ell(\oOm A_\abb\o R,\ \msb-u\msB)\ \to\ 
\bplus_{j \geq 1}\ H^{\ell-2j}(\Spec A_\abb) \en
 \bplus\en
\oOm_{\Spec A_\abb}^\ell / \sd \oOm_{\Spec A_\abb}^{\ell-1}.
$$
Specifically, according to \cite[Proposition ~2.3.7]{L}, this map is
defined by the assignment
\begin{equation}\label{loday}
  \Psi^{\text{\rm classical}}_\ell:\
  \sum_{0\leq j\leq [\ell/2]}\ f_{\ell-2j}\, u^{-j}\ \mto\
  \bigoplus_{0\leq j\leq [\ell/2]}\ \frac{1}{(\ell-2j)!}\cdot
  \bar f_{\ell-2j}.
\end{equation}

Let $\abb:\ A\to A_\abb$ be the abelianization homomorphism and write
$\HC(\abb): \oHC(A) \to \oHC(A_\abb)$ for the induced map on cyclic
homology.  The abelianization map $\HC(\abb)$ clearly intertwines the
Tsygan maps for the algebras $A$ and $A_\abb$, respectively.  Note
further that Hilbert's Nullstellensatz yields canonical isomorphisms
$\Rep_1 A =\Rep_1 A_\abb$ $= \Spec A_{{\rm{ab}}}$. With these
observations in mind, comparing formulas \eqref{ddd} and \eqref{loday}
yields the comparison result

\begin{corollary}\label{psipsi} One has
  $\ \sN!\ccirc \Psi_\idot^0=\Psi^{\text{\rm classical}}_\idot\ccirc
  \HC(\abb)$, where $\sN!$ is the map that acts by scalar
  multiplication by $k!$ in de Rham degree $k$.
\end{corollary}

Restricting the map \eqref{ququ} to $\oHD_\ell(A)$ via the embedding
$S_2$ from the Connes exact sequence, cf.~Theorem \ref{con}, we obtain
a map
\begin{equation}\label{HDmap}\Psi_{\ell}^0\ccirc S_2:\ 
  \oHD_\ell(A)\ \too\  \bplus_{j \geq 0}\ \overline{\!H\!}\,^{\ell-2j}(\Rep_1 A). 
\end{equation}

\begin{remark}
  In the case where $A$ is a {\em smooth} commutative algebra the map
  $\Psi^{\text{classical}}_\ell$ is known to be an isomorphism, see
  \cite[Theorem 3.4.12]{L}. It follows that, in this case, each of the
  maps \eqref{ququ} and \eqref{HDmap} is an isomorphism as well.
\end{remark}

\underline{{\em Restriction of $\gl_n$ to scalars}}:\quad For general
$n \geq 1$, the subalgebra of scalar matrices $\gl_1 \subseteq \gl_n$
still acts trivially on $\Rep_n A$. Thus, restricting the general
construction to $\gl_1$ produces a map

\begin{equation}
\label{e:hc-eq-repn}
\oHC_\ell(A)\ \to\ \oOm_{\Rep_n A}^\ell  / \sd \oOm_{\Rep_n A}^{\ell-1}\en
\bplus\en
\left(\oplus_{j \geq 1}\ \overline{\!H\!}\,^{\ell-2j}(\Rep_n A)\right).
\end{equation}

There is also a comparison result, similar to Corollary \ref{psipsi},
saying that the map \eqref{e:hc-eq-repn} is equal, up to a twist by
the automorphism $\sN!$, to the composition of the evaluation map
\[\oHC_\ell(A)\ \to\  \oHC_\ell(\bk[\Rep_n A] \otimes
  \Mat_n(\bk))\  \cong\  \oHC_\ell(\bk[\Rep_n A])\] with 
isomorphism $\Psi^{\text{\rm classical}}$.

If we don't restrict to scalars, the maps $\oHC(A) \to
\oHDe_{\!\gl_n}(\Rep_n A)$ should capture finer information, having to
do with the equivariant geometry of the representation scheme.

\bibliographystyle{amsalpha}

\begin{thebibliography}{CBEG}

\bibitem[BKR]{BKR} Yu. Berest, G. Khachatryan, and A. Ramadoss,
\emph{Derived representation schemes and cyclic homology},
 arXiv:1112.1449.

\bibitem[CBEG]{CBEG} W.~Crawley-Boevey, P.~Etingof, and V.~Ginzburg,
  \emph{Noncommutative geometry and quiver algebras},
  Adv. Math. \textbf{209} (2007), no.~1, 274--336,
  arXiv:math.AG/0502301.

\bibitem[Con]{Con-ndg} A.~Connes, \emph{Noncommutative differential
    geometry}, Inst. Hautes \'Etudes Sci. Publ. Math. (1985), no.~62,
  257--360.

\bibitem[CQ1]{CQ1} J.~Cuntz and D.~Quillen, \emph{Algebra extensions
    and nonsingularity}, J. Amer.  Math. Soc. \textbf{8} (1995),
  no.~2, 251--289.

\bibitem[CQ2]{CQ2} \bysame, \bysame, \emph{Operators on noncommutative
    differential forms and cyclic homology}, Geometry, topology, \&
  physics for Raoul Bott (Cambridge, MA) (S.-T. Yau, ed.),
  Conf. Proc. Lecture Notes Geom. Topology, vol.~4, Internat.  Press,
  Inc., 1995, pp.~77--111.

\bibitem[FT]{FT-aktcc} B.~Feigin and B.~Tsygan, \emph{Additive
    $K$-theory and crystalline cohomology}, Funktsional. Anal. i
  Prilozhen. (Russian) (1985), no.~19, 52--62.

\bibitem[Gi]{Glect} V.~Ginzburg, \emph{Lectures on noncommutative
    geometry}, arXiv:math.AG/0506603, 2005.

\bibitem[GS]{GScyc} V.~Ginzburg and T.~Schedler, \emph{Free products,
    cyclic homology, and the {G}auss-{M}anin connection},
  arXiv:0803.3655; to appear in Adv. Math.

\bibitem[Go]{Go} T. Goodwillie, {\em
Cyclic homology, derivations, and the free loopspace.}
Topology 24 (1985), 187-215. 

\bibitem[LQ]{LQ-chlahm} J.-L.~Loday and D.~Quillen, \emph{Cyclic
    homology and the Lie algebra homology of matrices},
  Comment. Math. Helv. (1984), no.~59, 569--591.

\bibitem[Kar]{Karhckt} M.~Karoubi, \emph{Homologie cyclique et
    {$K$}-th\'eorie}, Ast\'erisque (1987), no.~149, 147.

\bibitem[Lo]{L} J.-L. Loday, \emph{Cyclic homology}, Grundlehren der
  Mathematischen Wissenschaften, vol. 301, Springer-Verlag, Berlin,
  1998.

\bibitem[Ri]{Rin-dfgca} G.~Rinehart, \emph{Differential forms on
    general commutative algebras}, Trans.  Amer. Math. Soc. (1963),
  no.~108, 195--222.

\bibitem[TT]{TT} D. Tamarkin and B. Tsygan, {\em The ring of
    differential operators on forms in noncommutative calculus.}
  105--131, Proc. Sympos. Pure Math., 73, A.M.S., Providence, RI,
  2005.
\bibitem[Tsy1]{Tsy-hmlarhh} B.~Tsygan, \emph{Homology of matrix Lie
    algebras over rings and the Hochschild homology}, Uspekhi
  Mat. Nauk (Russian) (1983), no.~38, 217--218.

\bibitem[Tsy2]{T11} \bysame, \emph{On the morphism from periodic
    cyclic homology to equivariant cohomology}, Appendix to
  [GS08]. (Preprint 2011, to appear).

\end{thebibliography}
\def\cprime{$'$}
\providecommand{\bysame}{\leavevmode\hbox to3em{\hrulefill}\thinspace}
\providecommand{\MR}{\relax\ifhmode\unskip\space\fi MR }
\providecommand{\MRhref}[2]{%
  \href{http://www.ams.org/mathscinet-getitem?mr=#1}{#2}
}
\providecommand{\href}[2]{#2}

\end{document}